\documentclass[oneside,reqno,a4paper,11pt]{amsart}
\usepackage{amsmath,amsthm,amssymb,amscd,soul,bropd}
\usepackage[scaled=.98,sups,osf]{XCharter}
\usepackage[scaled=1.04,varqu,varl]{inconsolata}
\usepackage[type1]{cabin}
\usepackage[xcharter,varg,varvw,smallerops,subscriptcorrection]{newtxmath}
\usepackage{color,verbatim}
\usepackage{graphicx}
\usepackage{hyperref}
\usepackage{mathrsfs}
\usepackage{cite}
\usepackage{color}
\usepackage{ifpdf}
\usepackage{fancyhdr}
\usepackage{multirow}
\usepackage{array}
\usepackage{appendix}
\usepackage{longtable}
\usepackage{bm}

\newtheorem{thm}{Theorem}[section]

\newtheorem{cor}[thm]{Corollary}
\newtheorem{prop}[thm]{Proposition}
\newtheorem{define}[thm]{Definition}
\newtheorem{rem}[thm]{Remark}

\newtheorem{lem}[thm]{Lemma}

\newcommand{\beq}{\begin{equation}}
	\newcommand{\eeq}{\end{equation}}
\newcommand{\ben}{\begin{eqnarray}}
	\newcommand{\een}{\end{eqnarray}}
\newcommand{\beno}{\begin{eqnarray*}}
	\newcommand{\eeno}{\end{eqnarray*}}

\numberwithin{equation}{section}
\subjclass[]{}

\begin{document}
	\title{Regularity of the free boundaries for the two-phase axisymmetric inviscid fluid}

\author{LILI DU$^{1,2}$}
\author{FENG JI$^2$}
\thanks{* This work is supported by National Nature Science Foundation of China Grant 12125102, and Sichuan Youth Science and Technology Foundation 2021JDTD0024.}
\thanks{E-mail: dulili@szu.edu.cn. \quad E-mail: jifeng_math@126.com}

\maketitle 

\begin{center}
	$^1$College of Mathematical Sciences, Shenzhen University,
	
	Shenzhen 518061, P. R. China.
\end{center}

\begin{center}
	$^2$Department of Mathematics, Sichuan University,
	
	Chengdu 610064, P. R. China.
\end{center}

\begin{abstract}
	In the seminal paper (Alt, Caffarelli and Friedman, Trans. Amer. Math. Soc., 282, (1984).), the regularity of the free boundary of two-phase fluid in two dimensions via the so-called ACF energy functional was investigated. It was shown the $C^1$ regularity of the free boundaries and asserted that the two free boundaries coincide under some additional assumptions. Later on the standard technique of Harnack inequality could be applied to improve the regularity to $C^{1,\eta}$. A recent significant breakthrough in the regularity of two-phase fluid is due to De Philippis, Spolaor and Velichkov, who investigated the free boundary of the two-phase fluid with the two-phase functional (De Philippis, Spolaor and Velichkov, Invent. Math., 225, (2021).), and the $C^{1,\eta}$ regularity of the whole free boundaries was given in dimension two. Moreover, the free boundaries of the two-phase fluids do not coincide and the zero level set  may process positive Lebesgue measure. In this paper, we consider the free boundaries for the two-phase axisymmetric fluid and show the free boundary is $C^{1,\eta}$ smooth. The Lebesgue measure of the zero level set of may also be positive, and the main difference lies in the degenerate elliptic operator and the free boundary conditions. More precisely, we use partial boundary Harnack inequalities and establish a linearized problem, whose regularity of the solutions implies the flatness decay of the two-phase free boundaries. Then the iteration argument gives the smoothness of the free boundaries.
	
\end{abstract}

\tableofcontents    
\section{Introduction and main results}

\subsection{Introduction}

In this paper we investigate a two-phase Bernoulli-type free boundary problem in axisymmetric case, obtained by minimizing the energy functional
\begin{equation} \label{func}
J_{\text{a,tp}}(u,D):=\int_D \left[\frac{\vert\nabla u\vert^2}{x_2} +x_2\left(\lambda_+^2I_{\{u>0\}}+\lambda_-^2I_{\{u<0\}}\right)\right]dX
\end{equation}
in a relatively open subset $D\subset\mathbb{R}_+^2:=\mathbb{R}\cap\{x_2\geq0\}$. Here $dX=dx_1dx_2$, $x_1$ is the symmetric axis, $\lambda_{\pm}$ are positive numbers, and $I_{A}$ is the characteristic function of the set $A$.

By a minimizer, we understand a function $u\in W^{1,2}_{\text{w}}(D)$ such that
$$J_{\text{a,tp}}(u, D)\leq J_{\text{a,tp}}(v, D)$$
for any $v\in\mathcal{K}$, where
\begin{equation}
\mathcal{K}:=\left\{u \in W_{\text{w}}^{1,2}(D;\mathbb{R}) \ | \ u=-1 \ \text{on} \ \{x_2=0\}\right\}.
\end{equation}
Here, $W_{\text{w}}^{1,2}(D;\mathbb{R})$ is the weighted space
$$W_{\text{w}}^{1,2}(D;\mathbb{R}):=\left\{ v\in W^{1,2}(D;\mathbb{R}) \ | \int_D \frac{|\nabla v|^2}{x_2}dX+\int_D\frac{|v|^2}{x_2}dX<\infty \right\}.$$

It should be noted that the critical points of the functional $J_{\text{a,tp}}$ solves an elliptic equation except on their zero level sets, and the gradient $|\nabla u|$ jumps across the free boundaries. More precisely, the Euler-Lagrange equation to the energy functional $J_{\text{a,tp}}$ reads that
\begin{equation}\label{eq0}
\begin{cases}
\Delta u-\frac{1}{x_2}\partial_{x_2} u=0 \qquad\qquad\qquad\quad\ \; \text{in} \quad \{u\neq0\}\cap D, \\
|\nabla u^+|^2-|\nabla u^-|^2=x_2^2\left(\lambda_+^2-\lambda_-^2\right) \quad \text{on} \quad (\partial\{u>0\}\cap\partial\{u<0\})\cap D, \\
|\nabla u^\pm|=x_2\lambda_\pm \qquad\qquad\qquad\qquad \ \ \text{on} \quad \left(\partial( \{u>0\}\cup\{u<0\} ) \cap D \right) \backslash\left(\partial\{u>0\}\cap\partial\{u<0\}\cap D\right)
\end{cases}
\end{equation}
for $u^+=\max\{u, 0\}$ and $u^-=-\min\{u, 0\}$.

The problem (\ref{eq0}) should be viewed in the general framework of two-phase free boundary problems in incompressible inviscid axisymmetric fluid. We postpone the detailed argument in Section 1.4.

Now we introduce some notations. We will simply denote $J_{\text{a,tp}}(u)$ or $J_{\text{a,tp}}$ without causing confusion. The two-phase fluids seperated by the zero level set $\{u=0\}$ are noted fluid 1 in $\{u>0\}$ and fluid 2 in $\{u<0\}$, and we denote the positive set
\begin{equation*}
\Omega_u^+:=\{u>0\}
\end{equation*}
as fluid 1 and the negative set
\begin{equation*}
\Omega_u^-:=\{u<0\}
\end{equation*}
as fluid 2. Moreover, we denote the \emph{two-phase part of the free boundary}
\begin{equation} \label{tp}
\Gamma_{\text{tp}}:=\partial\Omega_u^+\cap\partial\Omega_u^-\cap D,
\end{equation}
and the \emph{one-phase part of the free boundary}
\begin{equation} \label{op}
\Gamma_{\text{op}}^+:=\left(\partial\Omega_u^+\cap D\right)\backslash \Gamma_{\text{tp}} \quad \text{and} \quad \Gamma_{\text{op}}^-:=\left(\partial\Omega_u^-\cap D\right)\backslash \Gamma_{\text{tp}},
\end{equation}
See Figure 1.

\begin{figure}[!h] 
	\includegraphics[width=100mm]{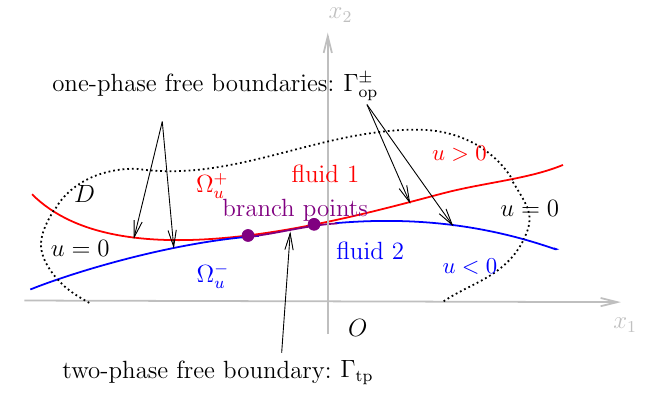}
	\caption{Two-phase axisymmetric fluid}
\end{figure}

Then (\ref{eq0}) can be rewritten as
\begin{equation} \label{eq}
\Delta u-\frac{1}{x_2}\partial_{x_2} u=0 \quad \text{in} \quad (\Omega_{u}^+\cup\Omega_{u}^-)\cap D
\end{equation}
with the Bernoulli type free boundary conditions
\begin{equation}\label{bc}
\begin{cases}
|\nabla u^+|^2-|\nabla u^-|^2=(x_2)^2(\lambda_+^2-\lambda_-^2) \ \ \quad \text{on} \quad \Gamma_{\text{tp}}, \\
|\nabla u^{\pm}| = x_2\lambda_{\pm} \qquad\qquad\qquad\qquad\qquad \text{on} \quad \Gamma_{\text{op}}^{\pm}.
\end{cases}
\end{equation}
Note that there is an additional free boundary condition
\begin{equation}\label{bc2}
|\nabla u^{\pm}|\geq x_2\lambda_{\pm} \quad \text{on} \quad \Gamma_{\text{tp}}
\end{equation}
which naturally arises from the minimizing problem (\ref{func}). We will verify the fact (\ref{bc}) and (\ref{bc2}) in Appendix A.

Furthermore, the two-phase free boundary points can be further divided into branch points and non-branch points. We say $x_0\in\Gamma_{\text{tp}}$ is a \emph{branch point} if $|B_r(x_0)\cap\{u=0\}|>0$ for any $r>0$. Otherwise we say $x_0\in\Gamma_{\text{tp}}$ is a \emph{non-branch point} if $|B_{r_0}(x_0)\cap\{u=0\}|=0$ for some $r_0>0$.

\subsection{Analysis of the two-phase functionals}

The regularity of the minimizers of the two-phase functional was first addressed by Alt, Caffarelli and Friedman in the pioneering paper \cite{ACF84}, which considered the following ACF functional
\begin{equation*}
J_{\text{acf}}(u)=\int_D \left(\vert\nabla u\vert^2 +\lambda_1^2I_{\{u>0\}}+\lambda_2^2I_{\{u<0\}}+\lambda_0^2I_{\{u=0\}}\right)dX.
\end{equation*}
They had a good observation that if $\lambda_0\geq\min\{\lambda_1, \lambda_2\}$ in $J_{\text{acf}}$, then the measure of the zero level set $\{u=0\}$ has to vanish. To see this fact, we assume that $\min\{\lambda_1, \lambda_2\}=\lambda_2\leq\lambda_0$, and $u$ is a minimizer of $J_{\text{acf}}$ locally in a ball $B$, with $|\{u=0\}\cap B|>0$. See Figure 2.

\begin{figure}[!h] 
	\includegraphics[width=85mm]{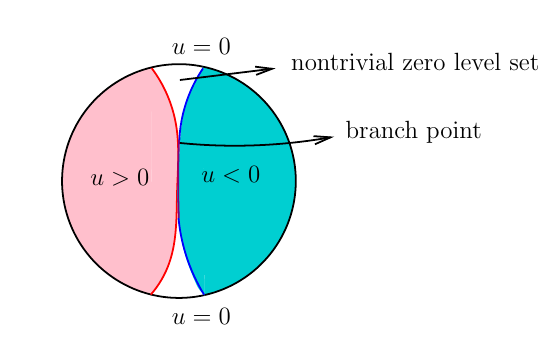}
	\caption{\{u=0\} with positive measure}
\end{figure}

We give a rough illustration about this observation, and the readers can find more rigorous details in \cite{ACF84}, Chapter 6. Under the assumption $\lambda_0\geq\min\{\lambda_1,\lambda_2\}$ if we set a harmonic function $v$ in $\{u\leq0\}\cap B$ such that $v$ equals $u$ on the boundary, then the Dirichlet energy of $v$ does not exceed that of $u$. Hence for the function
$$w=u^+-v,$$
as in Figure 3, we have
$$\int_B |\nabla w|^2 dX < \int_B |\nabla u|^2 dX$$
since $v$ is harmonic in $\{u<0\}\cap B_1$, but $u^-$ is not harmonic in $\{u<0\}\cap B_1$. Furthermore,
\begin{equation*}
\begin{aligned}
\int_B \big(\lambda_1^2I_{\{w>0\}}+\lambda_2^2I_{\{w<0\}}+\lambda_0^2I_{\{w=0\}}\bigr)dX&=\int_B \bigl(\lambda_1^2I_{\{u>0\}}+\lambda_2^2I_{\{w<0\}}\big)dX \\
&=\int_B \left(\lambda_1^2I_{\{u>0\}}+\lambda_2^2I_{\{u\leq0\}}\right)dX \\
&\leq\int_B \left(\lambda_1^2I_{\{u>0\}}+\lambda_2^2I_{\{u<0\}}+\lambda_0^2I_{\{u=0\}}\right)dX.
\end{aligned}
\end{equation*}
This implies that
$$J_{\text{acf}}(w) < J_{\text{acf}}(u),$$
a contradiction to the fact that $u$ is a minimizer of $J_{\text{acf}}$.

\begin{figure}[!h] 
	\includegraphics[width=50mm]{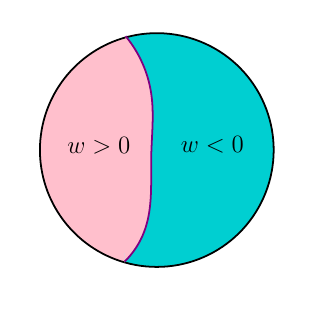}
	\caption{The function $w$}
\end{figure}

Therefore, the three mathematicians deduced that there is no cavity $\{u=0\}$ in the fluid, and the free boundaries of the minimizer are continuously differentiable for $\lambda_0=\min\{\lambda_1, \lambda_2\}$. Namely, the two free boundaries $\partial\{u>0\}$ and $\partial\{u<0\}$ coincide, and the zero level set $\{u=0\}$ has zero Lebesgue measure. That is, there is no branch point.

How about the case $\lambda_0<\min\{\lambda_1,\lambda_2\}$? And how to investigate the two-phase fluid with branch point? As a recent breakthrough by De Philippis, Spolaor and Velichkov in \cite{PSV21}, the following two-phase functional
$$J_{\text{tp}}(u)=\int_D \left(\vert\nabla u\vert^2 +\lambda_+^2I_{\{u>0\}}+\lambda_-^2I_{\{u<0\}}\right)dX$$	
was investigated. There is no additional term $\lambda_0|\{u_0=0\}|$ in this functional. It is noteworthy that under the assumption $\lambda_0<\min\{\lambda_1,\lambda_2\}$ for $J_{\text{acf}}$, there is an equivalence between $J_{\text{tp}}$ and $J_{\text{acf}}$. In fact, we can assume that $\lambda_+^2=\lambda_1^2-\lambda_0^2$ and  $\lambda_-^2=\lambda_2^2-\lambda_0^2$. Then
\begin{equation*}
\begin{aligned}
J_{\text{tp}}(u)
&=\int_D \left(\vert\nabla u\vert^2 +\lambda_+^2I_{\{u>0\}}+\lambda_-^2I_{\{u<0\}}\right)dX \\
&=\int_D \left(\vert\nabla u\vert^2 +(\lambda_1^2-\lambda_0^2)I_{\{u>0\}}+(\lambda_2^2-\lambda_0^2)I_{\{u<0\}}\right)dX \\
&=\int_D \left(\vert\nabla u\vert^2 +\lambda_1^2I_{\{u>0\}}+\lambda_2^2I_{\{u<0\}}-\lambda_0^2I_{\{u>0\}\cup\{u<0\}}\right)dX \\
&=\int_D \left(\vert\nabla u\vert^2 +\lambda_1^2I_{\{u>0\}}+\lambda_2^2I_{\{u<0\}}+\lambda_0^2I_{\{u=0\}}-\lambda_0^2I_D\right)dX \\
&=J_{\text{acf}}(u)-\lambda_0^2|D|,
\end{aligned}
\end{equation*}
which gives the equivalence between the two functionals $J_{\text{acf}}$ and $J_{\text{tp}}$. Hence with positive parameters $\lambda_{\pm}$, the $C^{1,\eta}$ regularity of the free boundary for local minimizers was obtained in two dimensions, and the two-phase fluid with nontrivial nodal set was firstly investigated in the elegant work \cite{PSV21}. However, some essential difficulties arose, such as the regularity of the free boundaries near the branch points. They introduced some novel ideas on the free boundaries near the branch points and developed the results of Silva in \cite{S11} for two-phase flow and gave a full description of the free boundary of the two-phase minimizer. The key point of their argument was to establish the compactness of a suitable sequence of functions and to get the limiting "linearized" problem. They observed that the "linearization" at the branch point is the two-membrane problem and reached the compactness of the linearizing sequence. Furthermore, the two-phase part $\Gamma_{\text{tp}}$ of the free boundaries is of $C^{1,\eta}$ regularity in any dimension, while either of the one-phase part $\Gamma_{\text{op}}^\pm$ follows the known result in \cite{V23}, and in contrast with the two-phase part, there is a critical dimension $d^*\in\{5,6,7\}$ for singular sets. Moreover, in 2023, David, Engelstein, Garcia and Toro constructed a family of minimizers for $J_{\text{tp}}$ whose free boundaries contain branch points in the strict interior of the domain in \cite{DEGT23}.

In this paper we follow the main guidelines in \cite{PSV21} to study the axisymmetric two-phase incompressible inviscid fluid in dimension three. The zero level set of the minimizer $u$ of $J_{\text{a,tp}}$ with $\lambda_{\pm}>0$ will have positive measure, which implies that $u$ has both one-phase free boundary points and two-phase free boundary points. The presence of a branch point requires us to face the situation as in \cite{PSV21}, however there are some additional difficulties here, such as the possible singularity near the axis of symmetry and the degeneracy of the operator near the axis of symmetry. We have to restructure the non-degeneracy and the Lipschitz regularity for the minimizer $u$, and furthermore study the regularity of the whole free boundaries.

In the following sections we assume that $\lambda_+\geq\lambda_->0$ without loss of generality.

\subsection{Mathematical background of two-phase fluid}

The free boundary mathematical theory of two-phase flow problems was first introduced by Alt, Caffarelli and Friedman in 1980s. They employed the variational method to prove the existence of the minimizer of the two-phase flow and established the $C^1$ regularity of the free boundaries in \cite{ACF84}. Based on the well-posedness and regularity theory they considered an incompressible inviscid flow of two jets in a pipe without branch points, investigating its existence and uniqueness. From then on, there has been a surge in studying such incompressible inviscid flows and their free boundaries. Caffarelli developed a standard and powerful approach in \cite{C87}\cite{C89}\cite{C88} in 1987-1989 to get the smoothness of the free boundary by viscosity method, which was widely used in research on the regularity of the free boundary problems for one-phase and two-phase problems. Recently, Silva has developed a new approach in \cite{S11} for this series of problems through the partial boundary Harnack inequality to improve the flatness. This new approach in \cite{SFS14} was applied to study the two-phase free boundary problems with distributed source, and in \cite{SFS15} for fully nonlinear non-homogeneous problems. In \cite{SFS17}, Silva, Ferrari and Salsa investigated the existence and the smoothness of viscosity solutions and their free boundaries. They also claimed some open problems for the existence of Lipschitz viscosity solutions in fully nonlinear case, and the analysis of singularities of the free boundary in non-homogeneous case. Very recently, the existence and structure of branch points in two-phase free boundary problem based on the ACF functional is investigated and an example of a two-phase problem with branch points is given in \cite{DEGT23}. The two-phase model can also describe the appearance of a phase transition from ice to water, see \cite{S18}, Section 5.4.1.

On the other hand, there have been extensive study and applications about the axisymmetric flow, which were developed by Serrin in \cite{S52}, Garabedian in \cite{G56}, Alt, Caffarelli and Friedman in \cite{ACF83}. Recently in 2014, V$\check{a}$rv$\check{a}$ruc$\check{a}$ and Weiss classified and analysed the degenerate points for a steady axisymmetric flow with gravity of dimension three in \cite{VW14}. Another important application of their model was in \cite{GVW16} to study the axisymmetric electrohydrodynamic equations.

There is a widespread application in hydrology and hydrodynamics for two-phase fluid. A typical example was the Prandtl-Bachelor model in fluid-dynamics in \cite{B56} and \cite{EM95}, where the stream functions may satisfy different equations in the two phases. Moreover, a great deal of mathematical efforts have been devoted to the study of the two-phase CFD model. For instance, the investigation of solid-liquid slurry flow was based on the Eulerian two-fluid model to simulate the flow in \cite{MM13}, the analysis on sediment water mixtures was based on a two-phase model in\cite{SCK21}, and so on. Additionally, this type of two-phase problem also arose in eigenvalue problem in magneto-hydrodynamics in \cite{FL94} and in flame propagation models in \cite{LW06} with forcing term.

Our prime goal is to consider the two-phase axisymmetric inviscid fluid of dimension three without external force. We will develop the method in the celebrated work \cite{PSV21} and get the $C^{1,\eta}$ regularity for the whole free boundaries.

\subsection{Mathematical formulation for two-phase axisymmetric inviscid fluids}

We are concerned with the axisymmetric ideal two-phase fluids, incompressible fluid 1 and incompressible fluid 2, in a three-dimensional space without swirl, which is originated from the incompressible Euler equations. 
\par
Suppose $U=(u_1,u_2,u_3)=(u_1(x,y,z),u_2(x,y,z),u_3(x,y,z))$ to be the velocity field of the fluid, with $U^+=(u_1^+,u_2^+,u_3^+)$ in fluid 1 and $U^-=(u_1^-,u_2^-,u_3^-)$ in fluid 2, and the $x$-axis to be the axis of symmetry. Then $U^\pm$ is a solution to the steady incompressible Euler system
\begin{equation*}
\begin{cases}
\nabla\cdot\rho_{\pm} U^\pm=0 \qquad\qquad\qquad \text{in} \quad \Omega_\pm, \\
\rho_{\pm}(U^\pm\cdot\nabla)U^\pm+\nabla p_{\pm}=0 \quad\, \text{in} \quad \Omega_\pm
\end{cases}
\end{equation*}
respectively in fluid 1 and fluid 2 with $\Omega_\pm$ denoting the two fluid fields, $\rho_{\pm}$ denoting the constant density and $p_{\pm}$ denoting the pressure of the two fluids. In addition, the flow is assumed to be irrotational, namely
\begin{equation*}
\nabla \times U^\pm =0.
\end{equation*}

\begin{figure}[!h]
	\includegraphics[width=95mm]{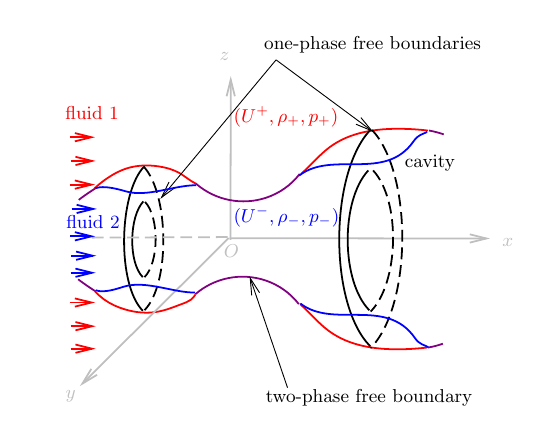}
	\caption{The axisymmetric two-phase free boundary problem}
\end{figure}

The Euler's equations in cylindrical polar coordinates can be derived as in Section 3.7.3 in \cite{C09}. Under the assumption that the flow is axisymmetric without swirl, we rewrite $x_1=x$, $x_2=\sqrt{y^2+z^2}$ and let $v_\pm(x_1,x_2), w_\pm(x_1,x_2)$ denote the radial velocity and the axial velocity of the two-phase fluids, respectively, i.e. $U^\pm=(v_\pm(x_1,x_2),w_\pm(x_1,x_2))$. Then
\begin{equation*}
u_1(x,y,z)=v(x_1,x_2), \ \
u_2(x,y,z)=w(x_1,x_2)\frac{y}{x_2}, \ \
u_3(x,y,z)=w(x_1,x_2)\frac{z}{x_2}.
\end{equation*}
Hence we obtain the following axisymmetric Euler system
\begin{equation} \label{eq2}
\begin{cases}
\partial_{x_1}(\rho_{\pm}x_2v) + \partial_{x_2}(\rho_{\pm}x_2w) =0, \\
\partial_{x_1}(\rho_{\pm}x_2v^2)+\partial_{x_2}(\rho_{\pm}x_2vw)+x_2\partial_{x_1}p_{\pm}=0, \\
\partial_{x_1}(\rho_{\pm}x_2vw)+\partial_{x_2}(\rho_{\pm}x_2w^2)+x_2\partial_{x_2}p_{\pm}=0
\end{cases}
\end{equation}
with irrotational condition
\begin{equation*}
\partial_{x_2}v-\partial_{x_1}w=0.
\end{equation*}

Consider the situations of the fluid 1 and fluid 2, respectively. Combining with the first equation in (\ref{eq2}), there is a stream function
\begin{equation*}
u(x_1,x_2)=
\begin{cases}
u^+ \quad \text{in fluid field 1}, \\
u^- \quad \text{in fluid field 2}, \\
0 \quad \ \ \text{otherwise},
\end{cases}
\end{equation*}
such that $v_{\pm}=\frac{1}{\sqrt{\rho_{\pm}}x_2}\partial_{x_2}u^\pm$ and  $w_{\pm}=-\frac{1}{\sqrt{\rho_{\pm}}x_2}\partial_{x_1}u^\pm$ respectively in fluid 1 and fluid 2. Consequently, the conservation of momentum and the irrotational condition give that
\begin{equation}
\Delta u-\frac{1}{x_2}\partial_{x_2}u=0
\end{equation}
respectively in fluid 1 and fluid 2.

As we know, on every streamline the stream function remains a constant. Hence, without loss of generality, we can define $\{u>0\}$ and $\{u<0\}$ to be the two fluid fields, respectively. Moreover, the $x_1$-axis is a level set of the stream function, and then we can normalize the value of the stream function on the axis to be
\begin{equation*}
u=-1 \ \text{on} \ \{x_2=0\}.
\end{equation*}
The free boundaries are defined as $\partial\{u>0\}\cup\partial\{u<0\}$. Here, we can define the two-phase free boundary $\Gamma_{\text{tp}}$ as in (\ref{tp}) and the one-phase free boundaries $\Gamma_{\text{op}}^\pm$ as in (\ref{op}). Notice that there might be a cavity $\{u=0\}$ with positive measure. See Figure 5 below.

\begin{figure}[!h]
	\includegraphics[width=90mm]{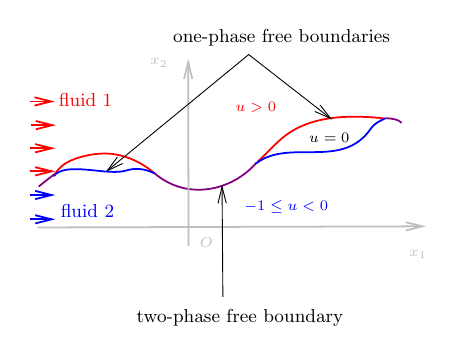}
	\caption{The axisymmetric two-phase free boundary problem}
\end{figure}

\par

On account of the Bernoulli's law we obtain that for the velocity field $U^\pm$, there are so-called Bernoulli's constants $\mathscr{B}_{\pm}$ such that
\begin{equation}
\frac{p_{\pm}}{\rho_{\pm}}+\frac12|U^\pm|^2=\mathscr{B}_{\pm}
\end{equation}
along streamlines of the incompressible and inviscid flow. Then we have
\begin{equation} \label{ber}
\frac{\rho_+}{2}(v_+^2+w_+^2)+p_+=\rho_+\mathscr{B}_+ \quad \text{and} \quad \frac{\rho_-}{2}(v_-^2+w_-^2)+p_-=\rho_-\mathscr{B}_-
\end{equation}
along streamlines for fluid 1 and fluid 2 respectively. Moreover, on the one-phase free boundaries $\Gamma_{\text{op}}^\pm$, the pressure is assumed to be the given constant pressure $p_0$, and on the two-phase free boundary $\Gamma_{\text{tp}}$, the pressure is assumed to be continued across it. Hence, we have
\begin{equation*}
p_\pm=p_0 \quad \text{on} \quad \Gamma_{\text{op}}^\pm \quad \text{and} \quad p_+=p_- \quad \text{on} \quad \Gamma_{\text{tp}}.
\end{equation*}
This together with (\ref{ber}) implies that
$$\frac{\rho_\pm}{2}(v_\pm^2+w_\pm^2)=\rho_\pm\mathscr{B}_\pm-p_0 \quad \text{on} \quad \Gamma_{\text{op}}^\pm$$
and
$$\frac{\rho_+}{2}(v_+^2+w_+^2) - \frac{\rho_-}{2}(v_-^2+w_-^2)=\rho_+\mathscr{B}_+-\rho_-\mathscr{B}_- \quad \text{on} \quad \Gamma_{\text{tp}}.$$

Define the positive parameters $\lambda_+$ and $\lambda_-$ as
$$\lambda_+^2=2\left(\rho_+\mathscr{B}_+-p_0\right) \quad \text{and} \quad \lambda_-^2=2\left(\rho_-\mathscr{B}_--p_0\right),$$
with $p_0$ the pressure of the cavity, and we have
$$\lambda_+^2-\lambda_-^2=2\left(\rho_+\mathscr{B}_+-\rho_-\mathscr{B}_-\right).$$
In fact, $\frac12\lambda_{\pm}^2$ represents the kinetic energy of the fluids per unit volumn on their one-phase free boundaries, and $\frac12(\lambda_+^2-\lambda_-^2)$ means the jump of the kinetic energy per unit volumn across the two-phase free boundary.

Recalling that $v_{\pm}=\frac{1}{\sqrt{\rho_{\pm}}x_2}\partial_{x_2}u$ and  $w_{\pm}=-\frac{1}{\sqrt{\rho_{\pm}}x_2}\partial_{x_1}u$, we have
$$\frac{1}{x_2^2}(|\nabla u^+|^2-|\nabla u^-|^2)=\lambda_+^2-\lambda_-^2$$
on the two-phase free boundary $\Gamma_{\text{tp}}$, and
$$\frac{1}{x_2}|\nabla u^{\pm}| = \lambda_{\pm}$$
on the one-phase free boundaries $\Gamma_{\text{op}}^\pm$. Thus we obtain the governing equation (\ref{eq}) and kinetic boundary conditions (\ref{bc}) on the free boundaries, which is a two-phase free boundary problem with Bernoulli's type boundary conditions.

\subsection{Main results}

Before giving the main result we first introduce the definition of local minimizers to the two-phase fluid problem.

The main purpose of this paper is to locally study the regularity of the free boundary. Our model is given by the functional
\begin{equation}
J_{\text{a,tp}}(u, B_r):=\int_{B_r} \left[\frac{\vert\nabla u\vert^2}{x_2} +x_2\left(\lambda_+^2I_{\{u>0\}}+\lambda_-^2I_{\{u<0\}}\right)\right]dX
\end{equation}
for
\begin{equation}
u \in \mathcal{K}':=\left\{u \in W_{\text{w}}^{1,2}(B_r; \mathbb{R}) \ | \ u=-1 \ \text{on} \ \{x_2=0\}\cap B_r\right\},
\end{equation}
where $B_r\Subset D$.

\begin{define}(Local minimizers)
	We say $u$ is a local minimizer of the functional $J_{\rm a,tp}$ in $D$, provided that
	$$J_{\rm a,tp}(u, B_r)\leq J_{\rm a,tp}(v, B_r)$$
	for any $v\in\mathcal{K}'$, $u-v\in W_0^{1,2}(B_r)$ and $B_r\Subset D$.
\end{define}

The main result in this paper discuss the regularity of the whole free boundary, including the one-phase parts $\Gamma_{\text{op}}^\pm$ and the two-phase part $\Gamma_{\text{tp}}$. The first result is about the uniform distance of the free boundary from the $x_1$-axis.

\begin{thm} \label{thm0}
	There exists a uniform constant $b\in(0,1)$ depending only on $D$ such that there is no free boundary point in $\{x_2\leq b\}\cap D$.
\end{thm}

The key point of this observation is that the gradient of the minimizer $u$ should be uniformly bounded near the axis, hence there must be a positive distance between the level sets $\{u=-1\}$ and $\{u=0\}$.

The second result says that the free boundary of the local minimizers is $C^{1,\eta}$ smooth. By Theorem \ref{thm0} we know that the gradient of the minimizers do not vanish on the free boundaries, so we can expect to get a good regularity for free boundary points.

\begin{thm}  \label{thm1}
	(Main result)
	Let $u \in \mathcal{K}'$ be a local minimizer of $J_{\rm a,tp}$ in $D$. Then for every free boundary point $x_0$, there is $r_0>0$ such that $\partial\Omega_u^+ \cap B_{r_0}(x_0)$ and $\partial\Omega_u^- \cap B_{r_0}(x_0)$ are $C^{1,\eta}$ graphs for some $\eta>0$. That is, $\partial\Omega_u^{\pm} \cap D$ are locally $C^{1,\eta}$ graphs.
\end{thm}

\begin{rem}
	Our approach may be applied to more general settings, such as
	\\
	(1) $\lambda_+(x)$, $\lambda_-(x)\in C_{loc}^{0,\alpha}(D)$ with a positive lower bound $\lambda_0$.
	\\
	(2) The axisymmetric two-phase flow with constant vorticity in three dimensions. The stream function $u$ solves $div\left( \frac{1}{x_2}\nabla u \right)=x_2\Lambda_+I_{\{u>0\}}+x_2\Lambda_-I_{\{u<0\}}$ in $\{u\neq0\}\cap D$, where $\Lambda_+$, $\Lambda_-$ are given constants.
\end{rem}

\begin{rem}
	In the present paper we consider the case $\lambda_+,\lambda_->0$. Our method can also be applied to the case $\lambda_+>0, \lambda_-=0$ (resp. $\lambda_->0, \lambda_+=0$) to get that $\partial\{u>0\}$ (resp. $\partial\{u<0\}$) is locally $C^{1,\eta}$.
\end{rem}

Utilizing the standard technique of iteration and bootstrapping we can get higher regularity.

\begin{thm}
	Let $u$ be as in Theorem \ref{thm1}. Then the free boundaries $\partial\Omega_u^{\pm} \cap D$ are locally $C^\infty$ graphs.
\end{thm}

\begin{rem}
	The elliptic operator $\mathcal{L}=\Delta -\frac{1}{x_2}\partial_2$ is singular near the axis $\{x_2=0\}$. However, in Section 2 we prove that the free boundary has a uniform distance from the axis, which implies that $\mathcal{L}$ is uniformly elliptic. Furthermore, we have to be careful under coordinate rotation since $\mathcal{L}$ does NOT keep invariant as the Laplacian operator does.
\end{rem}

\begin{rem}
	One of the main differences with the method in \cite{ACF84} is the lack of the ACF's monotonicity, where Caffarelli assumed that $\lambda_0=\min\{\lambda_1, \lambda_2\}$ in the functional $J_{\rm acf}(u)$ and got the Lipschitz regularity for the minimizer $u$. In fact, without such a condition in our functional $J_{\rm a,tp}(u)$, we can prove in Appendix C an ACF-type monotonicity formula, which is first addressed in \cite{C88}. Furthermore, David, Engelstein, Garcia and Toro gave a monotonicity formula for almost minimizers for $J_{\rm tp}(u)$ in Section 7 of \cite{DEGT21}, which implies the Lipschitz regularity of $u$ across the free boundaries.
\end{rem}

\begin{rem}
	Here, in the present paper, the value of $|\nabla u^\pm|$ involves $x_2$ along the free boundaries. Compared to the elegant work \cite{PSV21}, when we construct the "linearized" function sequence to measure the difference between the blow-up sequence and the half-plane solution, it is technically more involved as the free boundary conditions for the blow-up sequence do not remain invariant. We will deal with it in Section 3.
\end{rem}

\begin{rem}
	A tantalizing question may be addressed is that whether we can develop the well-posedness result in \cite{ACF84-1}\cite{ACF84-2} to establish the existence of a cavity in incompressible jets with two fluids, even in two-dimensional case, see Figure 6. One of the key steps is to seek a mechanism to guarantee the continuous fit condition, namely, the free boundary will connect the endpoint of the solid nozzle wall. This might be a challenging issue, which will be explored in our forthcoming paper.
\end{rem}

\begin{figure}[!h]
	\includegraphics[width=90mm]{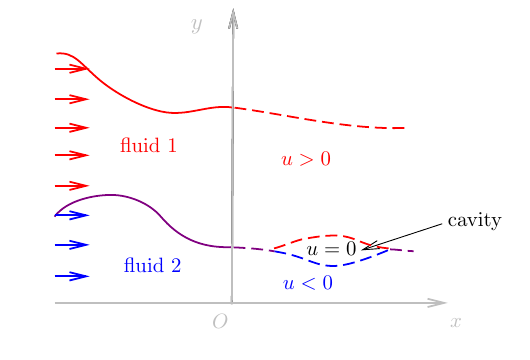}
	\caption{}
\end{figure}

The main underlying idea of this paper is to set up an iterative improvement of flatness argument in a neighborhood of a free boundary point. We follow the strategy of De Silva et al. developed in \cite{SFS14}, and De Philippis et al. in \cite{PSV21}. The key ingredients of the proof are the partial Harnack inequality and the analysis of the linearized problem. We first use the standard technique of blow-up analysis. The partial boundary Harnack inequality for the elliptic operator $\mathcal{L}=\Delta -\frac{1}{x_2}\partial_2$ allows the compactness of the linearizing sequence, and we obtain the limiting problem by viscosity means. The regularity of the limiting problem allows the desired decay of flatness.

Our paper is organized as follows. In Section 2, we exhibit some basic properties of the minimizer, including the positive distance between the free boundary and the axis of symmetry, the non-degeneracy and the Lipschitz regularity of the minimizer. Moreover, we introduce the viscosity solutions, derive their optimal boundary conditions and establish the relationship between the viscosity solution and the minimizer. In Section 3, we investigate the improvement of flatness of the blow-up sequence. We set a linearizing sequence, deduce the partial boundary Harnack inequality, get the "linearized" problem and then argue by contradiction to show the flatness decay. In Section 4, we prove our main result. The appendices are prepared for some supplementary details. In Appendix A, we check the free boundary conditions of the minimizer $u$ by the method of domain variation. In Appendix B and C, we study the non-degeneracy and the Lipschitz regularity of $u$ for sake of completeness. In Appendix D, we give a preparing lemma for partial boundary Harnack inequality. In Appendix E, we prove a touching lemma, which is used to derive the viscosity boundary conditions. In Appendix F, we list two regularity theorems for the limiting problem we get in Section 3.

\section{Some basic properties of the minimizer}

The main purpose of this section is to present some basic properties of the minimizer $u$. Notice that every property is based on the result of the first subsection, which gives a uniform distance between the free boundaries and the $x_1$-axis.

\subsection{Uniform distance between free boundaries and the axisymmetric axis}
In axisymmetric problems, the elliptic operator $\mathcal{L}=\Delta -\frac{1}{x_2}\partial_2$ appears to be quite different from the Laplacian operator $\Delta$. The presence of the singularity near the axis makes the maximum principle and elliptic estimates unavailable, and the Lipschitz regularity and the non-degeneracy of $u$ may fail. It is of great importance to prove the uniform distance between the free boundaries and the symmetric axis, which is different from the works in \cite{ACF84} and \cite{PSV21}, and requires delicate arguments.

\begin{prop} \label{prop1}
	(Uniform distance between free boundaries and $x_1$-axis)
	Suppose that $u$ is a minimizer of $J_{\rm a,tp}$ in $D$. Then there is a uniform constant $b\in(0,1)$ independent of the free boundary point $x$ such that $\partial\Omega_u^{\pm} \subset \{x_2>b\}$. 
\end{prop}
\begin{proof}
	It is valid to claim that the minimizer $u$ is continuous in $D$. In fact, $u$ is H\"older continuous in any subset of $D\cap\{x_2>0\}$ where the elliptic operator $\mathcal{L}$ is strictly elliptic. The symmetry axis $x_2=0$ is in fact inside the fluid domain, and we can remove the singularity of $\mathcal{L}$ by considering the minimizers of the approximating functionals
	$$\mathcal{J}_m = \int_{B_r} \left[ \frac{|\nabla u|^2}{x_2+m} + (x_2+m)\left( \lambda_+^2 I_{\{u>0\}} + \lambda_-^2 I_{\{u<0\}} \right) \right]dX$$
	for $m\rightarrow0+$ to get the H\"older continuous of $u$ near $x_2=0$.
	
	From the continuity of $u$ in $D$ we know that $\partial\Omega_u^+$ lies above $\partial\Omega_u^-$, and it suffices to show $\partial\Omega_u^-\subset\{x_2>b\}$. We suppose, by way of contradiction, that for any $0<b\ll 1$ there is a point $M=(m_1,m_2)\in\partial\Omega_u^-$ such that $m_2<b$. Let $N=(m_1,0)$ be the injection of $M$ on $\{x_2=0\}$. (Please see Figure 7.)
	
	\begin{figure}[!h]
		\includegraphics[width=100mm]{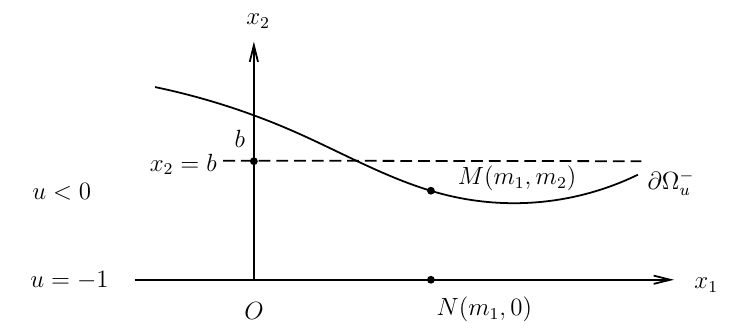}
		\caption{}
	\end{figure}
	
	We claim that after proper choice of $m_2<b$, the segment $\overline{MN}$ is totally contained in $\Omega_u^-$ except for the endpoint $M$. That is, $tM+(1-t)N\in\{u<0\}$ for $0\leq t<1$. In fact if not, then for any point $P_1(m_1,p_1)\in\partial\Omega_u^-$ with $0<p_1<m_2$, there exists another point $P_2(m_1,p_2)$ with $0<p_2<p_1$, and for such $P_2$ there exists $P_3(m_1,p_3)$ with $0<p_3<p_2$. Repeat the process we will get a sequence of points $\{P_k\}_{k=1}^\infty$ and a decreasing sequence $\{p_k\}_{k=1}^\infty$ satisfying $m_2>p_1>p_2>\dots>p_n>\dots$, which converges to $p_\infty\geq 0$ up to a subsequence as in Figure 8. The fact that $\partial\Omega_u^-$ is close implies $(m_1,p_\infty)=:P_{\infty}\in\partial\Omega_u^-$. Notice that $u=-1$ on $\{x_2=0\}$ leads to $p_\infty>0$, and we can take $M=P_\infty$. Hence $\overline{MN}\backslash\{M\}\subset\Omega_u^-$, a contradiction.
	
	\begin{figure}[!h]
		\includegraphics[width=105mm]{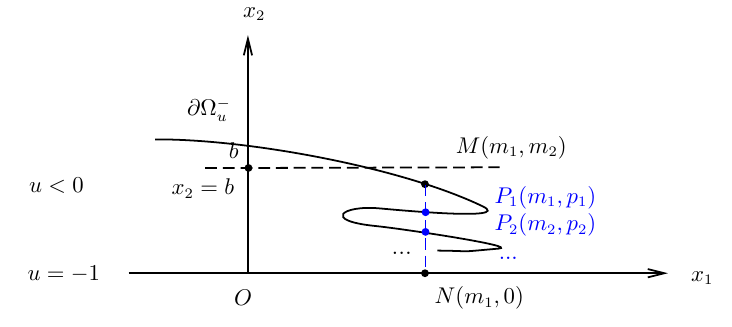}
		\caption{}
	\end{figure}
	
	The subsequent proof is based on the idea of \cite{ACF83} and \cite{CDZ21} to derive the Lipschitz regularity of the minimizer $u$ near the axis. Set
	\begin{equation*}
	u_0(x)=\frac{1}{t^2}u(x_0+tx)
	\end{equation*}
	to remove the singularity near $x_2=0$, where $x_0=(x_1^0,x_2^0)\in B_r(M)\cap\{u<0\}, x_2^0<b$ and $t=\frac{x_2^0}{2}$.
	Then $u_0$ solves the equation
	\begin{equation*}
	div\left(\frac{\nabla u_0}{2+x_2}\right)=0
	\end{equation*}
	in a small neighborhood of $x_0$. Using the elliptic estimate in \cite{GT01}, Chapter 8,
	\begin{equation*}
	\left|\frac{\nabla u(x_0)}{x_2^0}\right|=\frac{1}{2}|\nabla u_0(0)|\leq C
	\end{equation*}
	where $C$ is a constant independent of $x_0$.
	Clearly,
	\begin{equation*}
	|\nabla u(x_0)|\leq Cx_2^0<Cb
	\end{equation*}
	for $\forall x_0\in B_r(M)\cap\{u<0\}$.
	
	On the other hand, 
	\begin{equation*}
	1=u(M)-u(N)\leq\int_{0}^{1}|\nabla u(tM+(1-t)N)|dt \leq Cb,
	\end{equation*}
	which implies $b\geq\frac1C$, a contradiction. This completes the proof.
\end{proof}

\subsection{Non-degeneracy and Lipschitz regularity}

Now we can establish the non-degeneracy and the Lipschitz regularity of the minimizer, which were first proved in \cite{ACF84} for the functional $J_{\text{acf}}$ in two-dimensional case.

\begin{prop}  \label{prop2}
	
	(1) (Non-degeneracy of the minimizer)
	Suppose that $u$ is a minimizer of $J_{\rm a,tp}$ in $D$. For every $x_0 \in \partial\Omega_u^{\pm} \cap D$, $B_r(x_0)\subset D$ with $r\leq\frac b2$ and any  $0<\kappa<1$, there is a constant $c=c(\lambda_{\pm},\kappa)$ such that if
	\begin{equation*}
	\frac1r\left(\fint_{\partial B_r(x_0)}(u^{\pm})^2dS\right)^{1/2} <c,
	\end{equation*}
	then $u^{\pm}\equiv0$ in $B_{\kappa r}(x_0)$.
	
	(2)	(Lipschitz regularity of the minimizer)
	Let $u$ be a minimizer of $J_{\rm a,tp}$ in $D$. Then $u \in C_{loc}^{0,1}(D)$.
\end{prop}

In order to keep the presentation clean, we refer the two proofs to Appendices B and C. It is noteworthy that the proof of (2) in Proposition \ref{prop2} falls into two cases, one for those points near the axis which tend to be the interior points in the fluid, and the other for those away from the axis crossing the free boundaries.

\subsection{Classification of the blow-up limit}

Let $u$ be a local minimizer of $J_{\text{a,tp}}$ in a ball $B\Subset D$. We consider its blow-up sequence
\begin{equation}
u_{x_0,r}(x):=\frac{u(x_0+rx)}{r}
\end{equation}
at $x_0=(x_1^0,x_2^0) \in \partial\Omega_{u}^{\pm}$ for $0<r<dist(x_0,\partial B)$. Then $u_{x_0,r}$ is well-defined in $B_R\subset\{x\in\mathbb{R}^2\ |\ x_0+rx\in B\}$ and vanishes at the origin. We simply denote $u_r=u_{x_0,r}$ without causing misunderstanding. Given a sequence $r \rightarrow 0$ we call $u_r$ a \emph{blow-up sequence}, and $r$ its \emph{blow-up radius}. Utilizing the Ascoli-Arzela lemma together with the Lipschitz regularity of $u$, we obtain that there is a subsequence of $u_r$ that converges uniformly to $u_0$ in $B_R$, where $u_0$ is a Lipschitz function vanishing at the origin. We call $u_0$ a \emph{blow-up limit} at $x_0$, and we denote $\mathcal{BU}(x_0)$ to be the set of all blow-ups at $x_0$.

The following lemma gives a classification of the functions in $\mathcal{BU}(x_0)$. Notice that the uniqueness of blow-up limits remains unknown now, and would be proved in Section 4 after we get the flatness decay.

\begin{prop}  \label{prop4}
	(Classification of the blow-up limits)
	Let $u$ be a local minimizer of $J_{\rm a,tp}$ in $D$, and $u_0$ a blow up limit at $x_0=(x_1^0,x_2^0) \in \partial\Omega_u$. Then, there exists a pair $\alpha=\alpha(x_0)$, $\bm{e}=(e_1,e_2)=\left(e_1(x_0),e_2(x_0)\right)$, such that
	\begin{equation} \label{bu}
	u_0(x)=
	\begin{cases}
	x_2^0\alpha(x\cdot \bm{e})^+ - x_2^0\beta(x\cdot \bm{e})^-=:H_{\alpha,\bm{e}}, \ \ \quad x_0\in \Gamma_{\rm{tp}}, \\
	x_2^0\lambda_+(x\cdot \bm{e})^+,\qquad\qquad\qquad\qquad\qquad x_0\in \Gamma_{\rm{op}}^+, \\
	x_2^0\lambda_-(x\cdot \bm{e})^-, \qquad\qquad\qquad\qquad\qquad x_0 \in \Gamma_{\rm{op}}^-,
	\end{cases}
	\end{equation}
	where $e$ is a unit vector and $\alpha, \beta$ satisfying $\alpha^2-\beta^2=\lambda_+^2-\lambda_-^2$ and $\alpha\geq\lambda_+$, $\beta\geq\lambda_-$.
\end{prop}

\begin{proof}
	For $x_0 \in \Gamma_{\text{op}}^{\pm}$, the blow-up limit is a half-plane solution when the dimension $n=2$, referring to \cite{V23}. For $x_0 \in \Gamma_{\text{tp}}$, the proof is similar as in \cite{PSV21}. We use the Weiss monotonicity formula to get the one-homogeneous of the function $u_0(x)$. Then the eigenfunction of the spherical Laplacian gives the form of $u_0(x)$. We omit the details here.
\end{proof}

Proposition \ref{prop4} says that the blow-up sequence at a two-phase free boundary point $x_0$ is close to a two-plane solution $H_{\alpha,\bm{e}}$. In fact, in a small neighborhood of $x_0$, the blow-up sequence is uniformly close to $H_{\alpha,\bm{e}}$.

\begin{prop}  \label{prop5}
	Suppose that $u$ is a minimizer of $J_{\rm a,tp}$ in $D$ and $x_0$ is a free boundary point on $\Gamma_{\rm tp}$. 
	\\
	Then for every $\epsilon > 0$, there are $r > 0, \rho > 0$ and a two-plane function $H_{\alpha,\bm{e}}$ defined as in (\ref{bu}) such that
	\begin{equation} \label{eq1}
	\Vert u_{y_0,r} - H_{\alpha,\bm{e}} \Vert_{L^\infty(B_1)} \leq \epsilon
	\end{equation}
	for every $y_0 \in B_{\rho}(x_0)$.
\end{prop}

\begin{proof}
	Thanks to Proposition \ref{prop2}, up to extracting a subsequence, for any given $\epsilon >0$ there exists an $r>0$ such that
	\begin{equation*}  
	\Vert u_{x_0,r} - H_{\alpha,\bm{e}} \Vert_{L^\infty(B_1)} \leq \epsilon/2.
	\end{equation*}
	On the other hand, by Proposition \ref{prop2} the Lipschitz regularity implies that
	\begin{equation*}
	\Vert u_{x_0,r} - u_{y_0,r} \Vert_{L^\infty(B_1)} \leq \frac{L}{r}\vert x_0-y_0 \vert,
	\end{equation*}
	where $L>0$ is the Lipschitz constant. Hence we can get \eqref{eq1} if we choose $\rho$ small enough satisfying $\frac{L\rho}{r} \leq \epsilon/2$.
\end{proof}

\begin{rem}
	If $\lambda_+>0$ and $\lambda_-=0$ at $x_0\in\Gamma_{\text{tp}}$, then the blow-up limit at $x_0$ writes $u_0(x)=x_2^0\alpha(x\cdot \bm{e})^+ - x_2^0\beta(x\cdot \bm{e})^-$ for $\alpha>0$ and $\beta=0$. See Figure 9 for a possible case.
	\begin{figure}[!h]
		\includegraphics[width=105mm]{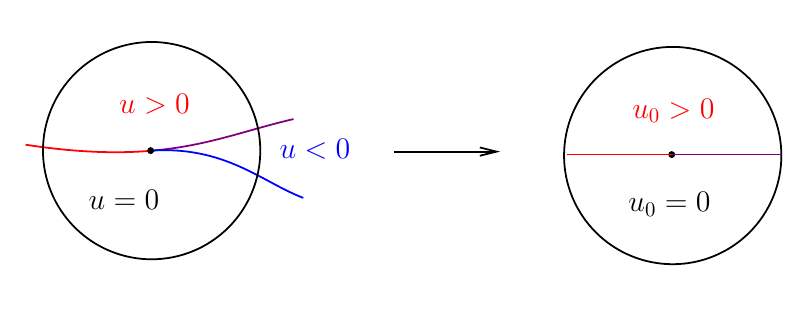}
		\caption{}
	\end{figure}
\end{rem}

\subsection{Optimal conditions at the free boundaries}
In this part we will give the definition of the viscosity boundary conditions and establish a connection between minimizers of $J_{\text{a,tp}}$ and viscosity solutions. It is standard to infer that the local minimizer $u$ satisfies the equation inside the fluid in weak sense, and thus in viscosity sense. So we mainly focus on the viscosity boundary condition.

We do not expect to get the $C^1$ regularity of the minimizer $u$ on the free boundary, and we will use the optimal condition in viscosity sense to describe the behaviour of $u$. The concept of viscosity solutions in free boundary problems was first addressed by Caffarelli in \cite{C88}, and we borrows some definitions for two-phase free boundary problems from \cite{PSV21}. We first give the concept of touch functions and comparison functions.

\begin{define}
	Let $D$ be an open set and $Q(x)$, $w(x)$ be two functions on $D$.
	
	(1) We say a function $Q(x)$ touching $w(x)$ from below (or above) at $x_0=(x_1^0,x_2^0)\in D$ if $Q(x_0)=w(x_0)$ and $Q(x)-w(x)\leq 0$ (or $Q(x)-w(x)\geq 0$) for every $x$ in a neighborhood of $x_0$. We say $Q(x)$ touching $w(x)$ strictly from below (or above) if the inequality is strict for $x\neq x_0$.
	
	(2) We say that $Q(x)$ is a comparison function in $D$ if
	
	(2a) $Q(x)\in C^1(\overline{\{Q(x)>0\}}\cap D)\cap C^1(\overline{\{Q(x)<0\}}\cap D)$;
	
	(2b) $Q(x) \in C^2(\{Q(x)>0\}\cap D)\cap C^2(\{Q(x)>0\}\cap D)$;
	
	(2c) $\partial\{Q(x)>0\}$ and $\partial\{Q(x)<0\}$ are smooth manifolds in $D$.
\end{define}

\begin{define}(Viscosity boundary conditions)
	We say that $u$ satisfies the viscosity boundary conditions of (\ref{bc}) on the free boundaries if the following holds.
	
	(A) Suppose that $Q(x)$ is a comparison function touching $u$ from below at $x_0=(x_1^0,x_2^0)\in\partial\Omega_{u}^{\pm}$.
	
	\quad\quad (A.1) If $x_0\in\Gamma_{\rm{op}}^+$, then 
	\[|\nabla Q^+(x_0)|\leq x_2^0\lambda_+;\]
	
	\quad\quad (A.2) If $x_0\in\Gamma_{\rm{op}}^-$, then $Q^+(x)\equiv 0$ in a neighborhood of $x_0$ and 
	\[|\nabla Q^-(x_0)|\geq x_2^0\lambda_-;\]
	
	\quad\quad (A.3) If $x_0\in\Gamma_{\rm{tp}}$, then
	\[|\nabla Q^-(x_0)|\geq x_2^0\lambda_-\] 
	and
	\[|\nabla Q^+(x_0)|^2-|\nabla Q^-(x_0)|^2\leq (x_2^0)^2(\lambda_+^2-\lambda_-^2)^2.\]
	
	(B) Suppose that $Q(x)$ is a comparison function touching $u$ from above at $x_0=(x_1^0,x_2^0)\in\partial\Omega_{u}^{\pm}$.
	
	\quad\quad (B.1) If $x_0\in\Gamma_{\rm{op}}^+$, then $Q^-(x)\equiv 0$ in a neighborhood of $x_0$ and
	\[|\nabla Q^+(x_0)|\geq x_2^0\lambda_+;\]
	
	\quad\quad (B.2) If $x_0\in\Gamma_{\rm{op}}^-$, then
	\[|\nabla Q^-(x_0)|\leq x_2^0\lambda_-;\]
	
	\quad\quad (B.3) If $x_0\in\Gamma_{\rm{tp}}$, then
	\[|\nabla Q^+(x_0)|\geq x_2^0\lambda_+\]
	and
	\[|\nabla Q^+(x_0)|^2-|\nabla Q^-(x_0)|^2\geq (x_2^0)^2(\lambda_+^2-\lambda_-^2)^2.\]
\end{define}

Notice that the boundary conditions are optimal. For instance, the right side of the inequality in case (A.1) cannot be smaller than $x_2^0\lambda_+$.

Before closing this subsection, we set the connection between local minimizers and viscosity solutions.

\begin{lem} \label{vis}
	(The local minimizers are viscosity solutions)
	Let $u$ be a local minimizer of $J_{\rm a,tp}$ in any compact set $D'\Subset D$, which means that $u$ satisfies
	\begin{equation}
	\mathcal{L}u=\Delta u-\frac{1}{x_2}\partial_2 u=0
	\end{equation}
	for $x\in\Omega_{u}^{\pm}\cap D'$ and
	\begin{equation}
	\begin{cases}
	|\nabla u^+|^2-|\nabla u^-|^2=(x_2)^2(\lambda_+^2-\lambda_-^2)  \quad\ \ \text{on} \quad \Gamma_{\rm{tp}}\cap D', \\
	|\nabla u^{\pm}| = x_2\lambda_{\pm} \qquad\qquad\qquad\qquad\qquad \text{on} \quad \Gamma_{\rm{op}}^{\pm}\cap D', \\
	|\nabla u^{\pm}|\geq x_2\lambda_{\pm} \qquad\qquad\qquad\qquad\qquad \text{on} \quad \Gamma_{\rm{tp}}\cap D'.
	\end{cases}
	\end{equation} 
	
	Then, $u$ satisfies the optimal viscosity boundary conditions on $\partial\Omega_{u}^{\pm}\cap D'$.
	
\end{lem}

\begin{proof}
	For one-phase points, the proof follows by \cite{V23}. For two-phase points in two dimensions, it follows by \cite{PSV21}. We only sketch the proof for two-phase points here in axisymmetric case.
	
	In the case $x_0\in\Gamma_{\text{tp}}$, suppose that $Q$ is a comparison function touching $u$ from below at $x_0$. Then up to a subsequence assume $u_{x_0,r_k}\rightarrow H_{\alpha,\bm{e}}$ uniformly for $H_{\alpha,\bm{e}}=x_2^0\alpha(x\cdot\bm{e})^+-x_2^0\beta(x\cdot\bm{e})^-$. On the other hand, $Q^{\pm}$ is differentiable at $x_0$ respectively in $\overline{\{Q>0\}}$ and $\overline{\{Q<0\}}$, and we get
	$$Q_{x_0,r_k}=\frac{Q(x_0+r_kx)}{r_k}\leq u_{x_0,r_k}$$
	and the blow-up limit
	$$H_Q(x)=|\nabla Q^+(x_0)|(x\cdot\bm{n})^+-|\nabla Q^-(x_0)|(x\cdot\bm{n})^-,$$
	where $\bm{n}=|\nabla Q(x_0)|^{-1}\nabla Q^+(x_0) = |\nabla Q(x_0)|^{-1}\nabla Q^-(x_0)$. Now since $H_Q$ touches $H_{\alpha,\bm{e}}$ from below, we have $\bm{n}=\bm{e}$ and
	$$|\nabla Q^+(x_0)|^2-|\nabla Q^-(x_0)|^2\leq(x_2^0)^2(\alpha^2-\beta^2),$$
	$$|\nabla Q^+(x_0)|\leq x_2^0\alpha, \quad |\nabla Q^-(x_0)|\geq x_2^0\beta,$$
	which lead to (A.3). The remainings are analogous.
\end{proof}

For future benefit, let us state the optimal viscosity boundary conditions in another way.

\begin{rem}
	Let $u$ be a local minimizer of $J_{\rm a,tp}$ in any compact set $D'\Subset D$. Then $u$ satisfies the following optimal boundary conditions.
	
	(1) Suppose that $Q$ is a comparison function touching $u$ from below at $x_0\in\Gamma_{\rm{op}}^+$(resp. $-u^-$ from above at $x_0\in\Gamma_{\rm{op}}^-$), then
	\[|\nabla Q^+(x_0)|\leq x_2^0\lambda_+ \quad\quad\quad (\text{or} \ |\nabla Q^-(x_0)|\leq\lambda^-).\]
	
	(2) Suppose that $Q$ is a comparison function touching $u$ from above at $x_0\in\partial\Omega_u^+$(resp. $-u^-$ from below at $x_0\in\partial\Omega_u^-$), then
	\[|\nabla Q^+(x_0)|\geq x_2^0\lambda_+ \quad\quad\quad (\text{or} \ |\nabla Q^-(x_0)|\geq\lambda^-).\]
\end{rem}

In the following parts we will consider the blow-up sequence $u_k$ at $x_0=(x_1^0,x_2^0)$, which locally satisfies
\begin{equation}
\mathcal{L}_ku_k=\Delta u_k-\frac{r_k}{x_2^0+r_kx_2}\partial_2 u_k=0 \quad \text{in} \quad \Omega_{u_k}^{\pm}\cap B_1,
\end{equation}
and
\begin{equation}
\begin{cases}
|\nabla u_k^+|^2-|\nabla u_k^-|^2=(x_2^0+r_kx_2)^2(\lambda_+^2-\lambda_-^2) \quad\ \quad \text{on} \quad \Gamma_{\text{tp}}\cap B_1, \\
|\nabla u_k^{\pm}| = (x_2^0+r_kx_2)\lambda_{\pm} \qquad\qquad\qquad\qquad\qquad \text{on} \quad \Gamma_{\text{op}}^{\pm}\cap B_1, \\
|\nabla u_k^{\pm}|\geq (x_2^0+r_kx_2)\lambda_{\pm} \qquad\qquad\qquad\qquad\qquad \text{on} \quad \Gamma_{\text{tp}}\cap B_1,
\end{cases}
\end{equation}
and the optimal boundary conditions for viscosity solution will change accordingly.

\section{Improvement of flatness}

The main underlying idea of reaching the $C^{1,\alpha}$ regularity of the free boundary is to "improve the flatness" of the blow-up in a smaller scale. See \cite{S11} for one-phase problem with distributed sources, and \cite{PSV21}\cite{SFS14} for two-phase problem with Laplacian operator.

We consider only those free boundary points on $\Gamma_{\text{tp}}$. This section is structured as follows. In the first part we construct a linearizing sequence related to the blow-up sequence and get its compactness by partial boundary Harnack's inequality. In the second part we describe the formulation of the linearized problem. In the last part we present the proof of the "flatness decay", setting up an iterative improvement of flatness argument in a neighborhood of $x_0$. Notice that the blow-up point $x_0$ matters, since the limiting problem of the linearizing sequence is different for branch point and non-branch points. All the proofs are distinguished into two cases.

Suppose that $u$ is a local minimizer of $J_{\text{a,tp}}$ in $B\Subset D$. Consider the blow-up sequence $u_k$ at $y_0=(y_1^0,y_2^0)\in\Gamma_{\text{tp}}$. Then $u_k$ minimizes
\begin{equation*}
J_{\text{a,tp},k}(u_k)=\int_{B_k}\left[\frac{|\nabla u_k|^2}{y_2^0+r_kx_2}+(y_2^0+r_kx_2)(\lambda_+^2I_{\{u_k>0\}}+\lambda_-^2I_{\{u_k<0\}})\right]dX,
\end{equation*}
for $B_k\subset \{x\in\mathbb{R}^2 \ | \ y_0+r_kx\in B\}$, and solves
\begin{equation}
\mathcal{L}_ku_k=\Delta u_k-\frac{r_k}{y_2^0+r_kx_2}\partial_2u_k=0 \quad \text{in} \quad \{u_k\neq0\}\cap B_k.
\end{equation}
Thanks to Proposition \ref{prop4}, there exists $H_{\alpha_k,\bm{e}}$, defined as in (\ref{bu}), that
\begin{equation}
\frac{\epsilon_k}{2}:=\Vert u_k-H_{\alpha_k,\bm{e}}\Vert_{L^\infty(B_1)}\rightarrow 0 \quad \text{for} \quad \lambda_+\leq\alpha_k\leq L,
\end{equation}
where $L$ is the uniform Lipschitz constant for $u_k$ and $\bm{e}=(e_1,e_2)$. Without loss of generality assume $e_2\geq0$.

\begin{rem}
	Here we cannot take $\bm{e}=(0,1)$ for simplicity, since the operator $\mathcal{L}_k$ will also change under the coordinate rotation.
\end{rem}

For the sake of subsequent proof we attempt to extract a subsequence $u_{r_{k'}}$, still noted as $u_k$, such that the blow-up radius $r_{k'}$ satisfies
\begin{equation} \label{rk}
r_{k'}=O(\epsilon_k^2).
\end{equation}
In fact, for such $\epsilon_k$, there is a positive number $\tilde{r}_k$ depending on $r_k$ such that for any $r_{k'},s_k<\tilde{r}_k$, the Cauchy sequence $u_{r_{k'}}$ satisfies
\begin{equation*}
\Vert u_{r_{k'}}-u_{s_k}\Vert_{L^\infty(B_1)}\leq\epsilon_k/2.
\end{equation*}
Then for any $r_{k'}<\min\{\tilde{r}_k, \epsilon_k^2\}$ and $s_k<\tilde{r}_k$,
\begin{equation}  \label{uk}
\Vert u_{r_{k'}}-H_{\alpha_k,\bm{e}}\Vert_{L^\infty(B_1)}\leq\Vert u_{r_{k'}}-u_{s_k} \Vert_{L^\infty(B_1)}+\Vert u_{s_k}-H_{\alpha_k,\bm{e}}  \Vert_{L^\infty(B_1)}\leq\epsilon_k\rightarrow 0
\end{equation}
and we get the desired order of the blow-up radius.

Set the linearizing sequence
\begin{equation}
v_k(x)=
\begin{cases}  \label{vk}
v_{+,k}(x)=\cfrac{u_k(x)-y_2^0\alpha_k(x\cdot\bm{e})^+}{y_2^0\alpha_k\epsilon_k}, \quad x\in\Omega_{u_k}^+\cap B_1,\\
\\
v_{-,k}(x)=\cfrac{u_k(x)+y_2^0\beta_k(x\cdot\bm{e})^-}{y_2^0\beta_k\epsilon_k}, \quad x\in\Omega_{u_k}^-\cap B_1,
\end{cases}
\end{equation}
and let
\begin{equation} \label{l}
l:=\lambda_+^2\lim_{k\rightarrow\infty}\frac{\alpha_k^2-\lambda_+^2}{2\alpha_k^2\epsilon_k}=\lambda_-^2\lim_{k\rightarrow\infty}\frac{\beta_k^2-\lambda_-^2}{2\beta_k^2\epsilon_k}.
\end{equation}
We have that $0\leq l<\infty$ for branch points and $l=\infty$ for non-branch points.

\begin{rem}
	If $e_2<0$, then we set
	\begin{equation*}
	v_k(x)=
	\begin{cases}
	v_{+,k}(x)=\cfrac{u_k(x)+y_2^0\alpha_k(x\cdot\bm{e})^-}{y_2^0\alpha_k\epsilon_k}, \quad x\in\Omega_{u_k}^+\cap B_1,\\
	\\
	v_{-,k}(x)=\cfrac{u_k(x)-y_2^0\beta_k(x\cdot\bm{e})^+}{y_2^0\beta_k\epsilon_k}, \quad x\in\Omega_{u_k}^-\cap B_1.
	\end{cases}
	\end{equation*}
	The argument will be quite similar with the case $e_2\geq0$.
\end{rem}

Now we distinguish the two cases by the value of $l$. The proofs for compactness of $v_k$ is divided into two cases as well. The value of $l$ determines the type of the limiting problem. The free boundary of $u_{x_0,0}$ at a branch point $x_0$ contains both one-phase part and two-phase part in $B_1$ ($0\leq l<\infty$), while it contains only two-phase part in $B_1$ ($l=\infty$) of the free boundary of $u_{x_0,0}$ at an interior two-phase point $x_0$.

\subsection{The case for branch points}
In this case we assume $0\leq l<\infty$. Notice that as stated in the elegant work \cite{PSV21},

\emph{" In order to get the compactness of the linearizing sequence, the partial improvement of flatness is not needed just at two-phase point $x_0$, but in all the points in a neighborhood of $x_0$."}

Our main differences here are the elliptic operator and the free boundary conditions, which bring some complicated calculus but cause not too much essential difficulties.

\subsubsection{Compactness}

We will get the compactness of $v_{\pm,k}$, and the trick of the proof is to establish partial boundary Harnack's inequality. We give the convergence theorem first, which holds for both $0\leq l<\infty$ and $l=\infty$.

\begin{prop} \label{comp}
	(Compactness of the linearizing sequence $v_k$)
	For a blow-up sequence $u_k$ and $v_k, \alpha_k, \epsilon_k$ defined as above, there are H\"older continuous functions
	\begin{equation*}
	v_{\pm}:\overline{B_{1/2}\cap\{(x\cdot\bm{e})^{\pm}>0\}} \rightarrow \mathbb{R}
	\end{equation*}
	such that the sequence of graphs
	\begin{equation*}
	\Gamma_k^{\pm}:=\left\{ (x,v_{\pm,k}(x)) \ | \ x\in\overline{\Omega_{u_k}^{\pm}\cap B_{1/2}} \right\}
	\end{equation*}
	converge in the Hausdorff distance to
	\begin{equation*}
	\Gamma_{\pm}:=\left\{ (x,v_{\pm}(x)) \ | \ x\in\overline{\{(x\cdot\bm{e})^{\pm}>0\}\cap B_{1/2}} \right\}
	\end{equation*}
	up to a subsequence.
	
	Furthermore, $v_{\pm}$ have the following properties:
	
	(1) Uniform convergence: $v_{\pm,k} \rightarrow v_{\pm}$ uniformly on $B_{1/2}\cap\{(x\cdot\bm{e})^{\pm}>\delta\}$ for any $\delta>0$.
	
	(2) Pointwise convergence: $v_{\pm}(x)=\lim_{k\rightarrow\infty}v_{\pm,k}(x_k)$ for every sequence $x_k\rightarrow x$, where $x_k\in\overline{\Omega_{u_k}^{\pm}\cap B_1}$ and $x\in \overline{\{y\in B_{1/2}|(y\cdot\bm{e})^{\pm}>0\}}$. In particular, for $x\in\{y\in B_{1/2}|(y\cdot\bm{e})^{\pm}>0\}$, $v_{\pm}(x)=\mp\lim_{k\rightarrow\infty}\frac{(x_k\cdot\bm{e})^{\pm}}{\epsilon_k}$ for $x_k\in\partial\Omega_{u_k}^{\pm}$ and $x_k\rightarrow x$.
	
\end{prop}

As a direct consequence we have the following corollary. Here we follow the notations in \cite{PSV21}. Let
$$\mathcal{J}=\{v_+<v_-\}\cap\{x\cdot\bm{e}=0\}\cap B_{1/2}$$
and
$$\mathcal{C}=\{v_+=v_-\}\cap\{x\cdot\bm{e}=0\}\cap B_{1/2}.$$

\begin{cor} \label{cor1}
	The limit functions $v_{\pm}$ in Proposition \ref{comp} satisfy $v_+\leq v_-$ on $\{x\cdot\bm{e}\}\cap B_{1/2}$, and the set $\{x\cdot\bm{e}=0\}\cap B_{1/2}=\mathcal{J}\cup\mathcal{C}$. Moreover, if $x\in\mathcal{J}$, then $x$ has a uniform distance to the two-phase points of $u_k$. That is,
	\begin{equation}
	\liminf_{k\rightarrow\infty}dist\left(x,\partial\Omega_{u_k}^+\cap\partial\Omega_{u_k}^-\right)>0.
	\end{equation}
	If $x\in\mathcal{C}$, then there is a sequence $x_k\in\partial\Omega_{u_k}^+\cap\partial\Omega_{u_k}^-$ such that
	\begin{equation*}
	x_k\rightarrow x.
	\end{equation*}
\end{cor}

\begin{proof}
	Exploit (2) in Proposition \ref{comp} we can simply get that $v_+\leq v_-$ on $\{x\cdot\bm{e}=0\}\cap B_{1/2}$. Moreover, for $x_k\in\partial\Omega_{u_k}^+\cap\partial\Omega_{u_k}^-$ converging to $x$, it equals $v_{+,k}(x_k)=v_{-,k}(x_k)$ and thus $v_+(x)=v_-(x)$, which gives the conclusion.
\end{proof}

Now we deal with the proof of the compactness. The spirit is mainly borrowed from \cite{PSV21}.

Without loss of generality, suppose that $\alpha_k-\lambda_+>\beta_k-\lambda_-$, and we have that for sufficiently large $k$, 
\begin{equation*}
\begin{aligned}
\Vert u_k-H_{\lambda_+,\bm{e}} \Vert_{L^\infty(B_1)} &\leq \Vert u_k-H_{\alpha_k,\bm{e}} \Vert_{L^\infty(B_1)}+\Vert H_{\alpha_k,\bm{e}}-H_{\lambda_+,\bm{e}} \Vert_{L^\infty(B_1)} \\
&\leq\epsilon_k+\Vert(\alpha_k-\lambda_+)(x\cdot\bm{e})^+\Vert_{L^\infty(B_1)}+\Vert(\beta_k-\lambda_-)(x\cdot\bm{e})^-\Vert_{L^\infty(B_1)} \\
&\leq\epsilon_k+2(\alpha_k-\lambda_+) \\
&=\epsilon_k+O(\epsilon_k) \\
&=:\bar{\epsilon_k}.
\end{aligned}
\end{equation*}
The last equality comes from the fact that $0\leq l<\infty$, where
\begin{equation*}
\begin{aligned}
l &= \lambda_+^2\lim_{k\rightarrow\infty}\frac{\alpha_k^2-\lambda_+^2}{2\alpha_k^2\epsilon_k} \\
&=\lambda_+^2\lim_{k\rightarrow\infty}(\alpha_k-\lambda_+)\left(\frac{1}{2\alpha_k\epsilon_k}+\frac{\lambda_+}{2\alpha_k^2\epsilon_k}\right) \\
&=\lambda_+^2\lim_{k\rightarrow\infty}\frac{1}{\epsilon_k}(\alpha_k-\lambda_+)\left(\frac{1}{2\lambda_+}+\frac{1}{2\lambda_+}\right) \\
&=\lambda_+\lim_{k\rightarrow\infty}\frac{\alpha_k-\lambda_+}{\epsilon_k}
\end{aligned}
\end{equation*}
and hence $\alpha_k-\lambda_+=O(\epsilon_k)$.

This implies
$$y_2^0\lambda_+\left( x\cdot\bm e-\frac{\bar{\epsilon_k}}{\lambda_+} \right)^+ - y_2^0\lambda_-\left( x\cdot\bm e-\frac{\bar{\epsilon_k}}{\lambda_+} \right)^- \leq u_k \leq y_2^0\lambda_+\left( x\cdot\bm e+\frac{\bar{\epsilon_k}}{\lambda_-} \right)^+ - y_2^0\lambda_-\left( x\cdot\bm e+\frac{\bar{\epsilon_k}}{\lambda_-} \right)^-$$
in $B_1$, thus
\begin{equation}
y_2^0\lambda_+\left( x\cdot\bm e-\frac{\bar{\epsilon_k}}{\lambda_+} \right)^+\leq  u_k^+\leq y_2^0\lambda_+\left( x\cdot\bm e+\frac{\bar{\epsilon_k}}{\lambda_-} \right)^+ \quad \text{in} \quad B_1,
\end{equation}
and
\begin{equation}
y_2^0\lambda_-\left( x\cdot\bm e+\frac{\bar{\epsilon_k}}{\lambda_-} \right)^-\leq u_k^- \leq y_2^0\lambda_-\left( x\cdot\bm e-\frac{\bar{\epsilon_k}}{\lambda_+} \right)^- \quad \text{in} \quad B_1.
\end{equation}

We need to introduce a test function $\phi$ before we prove the compactness, since the subsequent proof is based on the comparison with $\phi$.

\begin{lem}
Let $Q=(Q_1,Q_2)=\frac15\bm{e}$ be a point and $\phi: B_1\rightarrow\mathbb{R}$ be a function defined by
\begin{equation} \label{phi}
\phi(x)=
\begin{cases}
1, \qquad\qquad\qquad\qquad\quad\ \ x\in B_{1/20}(Q), \\
\kappa\left(|x-Q|^{-2}-(\frac34)^{-2}\right), \quad x\in B_{3/4}(Q)\backslash\overline{B_{1/20}}(Q), \\
0, \qquad\qquad\qquad\qquad\quad\ \ \text{otherwise},
\end{cases}
\end{equation}
where $\kappa=\frac{1}{400-(3/4)^2}$. Then it is easy to check that $\phi$ has the following properties:

(1) $0\leq\phi\leq1$ in $B_1$ and $\phi=0$ on $\partial B_1$.

(2) $\mathcal{L}_k\phi=\Delta\phi-\frac{r_k}{y_2^0+r_kx_2}\partial_2\phi=2\kappa|x-Q|^{-4}\bigl( 2+\frac{r_k}{y_2^0+r_kx_2}(x_2-Q_2) \bigr)>0$ in $\{\phi>0\}\backslash\overline{B_{1/20}}(Q)$.

(3) $\partial_e\phi=-2\kappa|x-Q|^{-4}(x-Q)\cdot\bm{e}=-2\kappa|x-Q|^{-4}(x\cdot\bm{e}-\frac15)>0$ in $\{\phi>0\}\cap\{|x\cdot\bm{e}|<\frac15\}$.

(4) $\phi\geq c$ in $B_{1/6}$ for some constant $c>0$.
\end{lem}

The next lemma is an instrumental tool to "improve" the two-plane solution defined as in (\ref{bu}).

\begin{lem} \label{pbh1}
	(Partial Boundary Harnack for branch case)
	Suppose that $\{u_k\}$ is a blow-up sequence of $u$. Then there exist constants $\bar{\epsilon}=\bar{\epsilon}(\lambda_{\pm})>0$ and $\tilde{c}=\tilde{c}(\lambda_{\pm})\in(0,1)$ such that the following property holds.
	
	If
	$$y_2^0\lambda_+(x\cdot\bm{e}+b_0)^+\leq u_k^+\leq y_2^0\lambda_+(x\cdot\bm{e}+a_0)^+ \quad \text{in} \quad B_4$$
	and
	$$-y_2^0\lambda_-(x\cdot\bm{e}+d_0)^-\leq -u_k^-\leq -y_2^0\lambda_-(x\cdot\bm{e}+c_0)^- \quad \text{in} \quad B_4$$
	for $a_0, b_0, c_0, d_0\in(-1/10,1/10)$, $b_0\leq d_0\leq c_0\leq a_0$ and $(a_0-b_0)+(c_0-d_0)\leq\bar{\epsilon}$,
	then there exist $a_1, b_1, c_1, d_1\in(-1/10,1/10)$ with $b_1\leq d_1\leq c_1\leq a_1$ and $a_1-b_1\leq\tilde{c}(a_0-b_0)$, $c_1-d_1\leq\tilde{c}(c_0-d_0)$ such that for $x\in B_{1/6}$,
	\begin{equation} \label{f3-1}
	y_2^0\lambda_+(x\cdot\bm{e}+b_1)^+\leq u_k^+\leq y_2^0\lambda_+(x\cdot\bm{e}+a_1)^+,
	\end{equation}
	and
	\begin{equation}\label{f3-2}
	-y_2^0\lambda_-(x\cdot\bm{e}+d_1)^-\leq -u_k^-\leq -y_2^0\lambda_-(x\cdot\bm{e}+c_1)^-.
	\end{equation}
	In particular, for $\epsilon_k, \bar{\epsilon_k}$ defined as before, let $k$ be sufficiently large we have $a_0=c_0=\frac{\bar{\epsilon_k}}{\lambda_-}$, $b_0=d_0=-\frac{\bar{\epsilon_k}}{\lambda_+}$, with $a_1-b_1=c_1-d_1\leq\tilde{c}\left(\frac{1}{\lambda_-}-\frac{1}{\lambda_+}\right)\bar{\epsilon_k}$ satisfying (\ref{f3-1}) and (\ref{f3-2}).
\end{lem}

Compared with the proof in \cite{PSV21}, the main difference here is the elliptic operator $\mathcal{L}$. Recall that the free boundaries are away from the $x_1$-axis, $\mathcal{L}$ is uniformly elliptic, and we can use the maximum principle. We contain the proof for the sake of completeness.

\begin{proof}
	We state the proof for $u_k^+$.
	
	Set $P=2\bm{e}$ and we distinguish two cases.
	
	\emph{Case 1. Improvement from below.}
	
Assume
$$u_k^+(P)\geq y_2^0\lambda_+(2+b_0)^++\frac{y_2^0\lambda_+(a_0-b_0)}{2},$$
which means that $u_k^+(P)$ is closer to $y_2^0\lambda_+(2+a_0)$ than to $y_2^0\lambda_+(2+b_0)^+$. In this case we will show
$$u_k^+(x)\geq y_2^0\lambda_+(x\cdot\bm{e}+b_1)^+$$
in a smaller ball centered at the origin.

Note that $\epsilon:=a_0-b_0\leq\bar{\epsilon}$, and we have
$$u_k^+\geq y_2^0\lambda_+(1+\tau\epsilon)(x\cdot\bm{e}+b_0)^+ \quad \text{in} \quad B_1$$
for a dimensional constant $\tau$. Next we distinguish two further sub-cases.

\emph{Case 1.1. $0\leq d_0-b_0\leq \eta\epsilon$ for $\epsilon$ being a small constant. }

For $x\in B_1$, we deduce that
\begin{equation} \label{f4}
\begin{aligned}
u_k(x) &\geq y_2^0\lambda_+(1+\tau\epsilon)(x\cdot\bm{e}+b_0)^+-y_2^0\lambda_-(x\cdot\bm{e}+d_0)^- \\
&\geq y_2^0\lambda_+(1+\tau\epsilon)(x\cdot\bm{e}+b_0)^+-y_2^0\lambda_-(x\cdot\bm{e}+b_0)^-.
\end{aligned}
\end{equation}
Now we set a new function
$$f_t(x)=y_2^0\lambda_+(1+\tau\epsilon/2)(x\cdot\bm{e}+b_0+\eta t\epsilon\phi)^+-y_2^0\lambda_-(x\cdot\bm{e}+b_0)^-$$
with $\phi$ defined as in (\ref{phi}), $t\in[0,1]$ and $\eta=\eta(\tau, \epsilon)$ a small universal constant satisfying
$$(1+\tau\epsilon)(x\cdot\bm{e}+b_0)^+\geq(1+\tau\epsilon/2)(x\cdot\bm{e}+b_0+\eta\epsilon)^+ \quad \text{in} \quad B_{1/20}(Q).$$
Hence,
\begin{equation*}
\begin{aligned}
u_k(x) &\geq y_2^0\lambda_+(1+\tau\epsilon)(x\cdot\bm{e}+b_0)^+ \\
&\geq y_2^0\lambda_+(1+\tau\epsilon/2)(x\cdot\bm{e}+b_0+\eta\epsilon)^+ \\
&\geq f_t(x) \quad \text{in} \quad B_{1/20}(Q).
\end{aligned}
\end{equation*}
Notice that $f_0(x)\leq u_k(x)$ in $B_1$. Let $\bar{t}\in[0,1]$ be the largest $t$ such that $f_t(x)\leq u_k(x)$ in $B_1$. We claim that $\bar{t}=1$. Indeed, assume $\bar{t}<1$, then there exists a point $\bar{x}\in B_1$ such that
$$u_k(x)-f_{\bar{t}}(x)\geq u_k(\bar{x})-f_{\bar{t}}(\bar{x})=0$$
for all $x\in B_1$. Then $\bar{x}\in\{0<\phi<1\}$.

We have for $k$ sufficiently large,
\begin{equation*}
\begin{aligned}
\mathcal{L}_kf_{\bar{t}}(x) &=y_2^0\lambda_+(1+\tau\epsilon/2)\left(\Delta-\frac{r_k}{y_2^0+r_kx_2}\partial_2\right)(x\cdot\bm{e}+b_0+\eta\bar{t}\epsilon\phi) \\
&=y_2^0\lambda_+(1+\tau\epsilon/2)\left[ -\frac{r_ke_2}{y_2^0+r_kx_2}+\eta\bar{t}\epsilon\cdot2\kappa|x-Q|^{-4}(2+\frac{r_k(x_2-Q_2)}{y_2^0+r_kx_2}) \right] \\
&>0
\end{aligned}
\end{equation*}
in $\{x\cdot\bm{e}+b_0+\eta\bar{t}\epsilon\phi>0\}\cap\{f_{\bar{t}}\neq0\}\cap B_1.$ Thanks to the maximum principle, $\bar{x}\notin\{f_{\bar{t}}\neq0\}$. Hence $\bar{x}$ is a free boundary point of $u_k$. Moreover, it follows from the fact that $f_{\bar{t}}$ changes sign in a neighborhood of $\bar{x}$, either $\bar{x}\in\partial\Omega_{u_k}^+\backslash\partial\Omega_{u_k}^-$ or $\bar{x}\in\partial\Omega_{u_k}^+\cap\partial\Omega_{u_k}^-$.

If $\bar{x}\in\partial\Omega_{u_k}^+\backslash\partial\Omega_{u_k}^-$, then thanks to the definition of viscosity solutions,
\begin{equation*}
\begin{aligned}
(y_2^0+r_k\bar{x}_2)^2\lambda_+^2 &\geq |\nabla f_{\bar{t}}(\bar{x})|^2 \\
&=(y_2^0\lambda_+)^2(1+\tau\epsilon/2)^2+2(y_2^0\lambda_+)^2\eta\bar{t}\epsilon\partial_e\phi(\bar{x})+O(\epsilon^2) \\
&>(y_2^0+r_k\bar{x}_2)^2\lambda_+^2
\end{aligned}
\end{equation*}
for $k$ sufficiently large, provided $\epsilon\leq\bar{\epsilon}\ll1$.

If $\bar{x}\in\partial\Omega_{u_k}^+\cap\partial\Omega_{u_k}^-$, then the definition of viscosity solutions gives
\begin{equation*}
\begin{aligned}
(y_2^0+r_k\bar{x}_2)^2(\lambda_+^2-\lambda_-^2) &\geq |\nabla f_{\bar{t}}^+(\bar{x})|^2-|\nabla f_{\bar{t}}^-(\bar{x})|^2 \\
&=(y_2^0\lambda_+)^2(1+\tau\epsilon/2)^2-(y_2^0\lambda_-)^2+2(y_2^0)^2(\lambda_+^2-\lambda_-^2)\eta\bar{t}\epsilon\partial_e\phi(\bar{x})+O(\epsilon^2) \\
&>(y_2^0+r_k\bar{x}_2)^2(\lambda_+^2-\lambda_-^2)
\end{aligned}
\end{equation*}
for $k$ sufficiently large, provided $\epsilon\leq\bar{\epsilon}\ll1$ and $\eta\ll\tau$.

These contradictions imply that $\bar{t}=1$. Notice that $\phi$ has a strictly positive lower bound in $B_{1/6}$,
\begin{equation*}
\begin{aligned}
u_k(x) &\geq y_2^0\lambda_+(1+\tau\epsilon)(x\cdot\bm{e}+b_0+\eta\epsilon\phi)^+-y_2^0\lambda_-(x\cdot\bm{e}+b_0)^- \\
&\geq y_2^0\lambda_+(x\cdot\bm{e}+b_0+\bar{c}\epsilon)^+-y_2^0\lambda_-(x\cdot\bm{e}+b_0)^-
\end{aligned}
\end{equation*}
for a suitable $\bar{c}$.

Set $a_1=a_0$, $b_1=b_0+\bar{c}\epsilon$ and it concludes the proof in this subcase.

\emph{Case 1.2. $d_0-b_0>\eta\epsilon$ for $\eta$ being a small constant.}

In this subcase we consider the function
$$f_t(x)=y_2^0\lambda_+(1+\tau\epsilon/2)(x\cdot\bm{e}+b_0+c_2 t\epsilon\phi)^+-y_2^0\lambda_-(1-c_1\eta\epsilon)(x\cdot\bm{e}+b_0+c_2t\epsilon\phi)^-.$$
Then $u_k(x)\geq f_0(x)$ in $B_1$ for $\eta$ determined in case 1.1, since
\begin{equation*}
\begin{aligned}
u_k(x)&\geq y_2^0\lambda_+(1+\tau\epsilon)(x\cdot\bm e+b_0)^+ - y_2^0\lambda_-(x\cdot\bm e+d_0)^- \\
&\geq y_2^0\lambda_+(1+\tau\epsilon)(x\cdot\bm e+b_0)^+ - y_2^0\lambda_-(1-c_1\eta\epsilon)(x\cdot\bm e+b_0)^-
\end{aligned}
\end{equation*}
for some $c_1=c_1(\eta,\epsilon)$.

Consider again $\bar{t}\in[0,1]$ be the largest $t$ such that $f_t\leq u_k$ in $B_1$ and $\bar{x}$ be the touching point in $B_1$. Assume $\bar{t}<1$, we can deduce as before that $\bar{x}\in\{0<\phi<1\}\cap B_1$.

It is straightforward to deduce that
\begin{equation*}
\mathcal{L}_kf_{\bar{t}}(x)=y_2^0\lambda_+(1+\tau\epsilon/2)\left[ \frac{-r_ke_2}{y_2^0+r_kx_2}+c_2\bar{t}\epsilon\cdot2\kappa|x_2-Q_2|^{-4} \right]>0
\end{equation*}
in $\{(x\cdot\bm{e}+b_0+c_2 t\epsilon\phi)^+>0\}$ and
\begin{equation*}
\mathcal{L}_kf_{\bar{t}}(x)=y_2^0\lambda_-(1-c_1\eta\epsilon)\left[ \frac{-r_ke_2}{y_2^0+r_kx_2}+c_2\bar{t}\epsilon\cdot2\kappa|x_2-Q_2|^{-4} \right]>0
\end{equation*}
in $\{(x\cdot\bm{e}+b_0+c_2 t\epsilon\phi)^->0\}$. Thus
we know from the maximum principle that $\bar{x}\in\{f_{\bar{t}}(x)=0\}$, which implies that $\bar{x}$ is a free boundary point of $u_k$. By the definition of $f_{\bar{t}}(x)$, we have $\bar{x}\in\partial\Omega_{u_k}^+\backslash\partial\Omega_{u_k}^-$ or $\bar{x}\in\partial\Omega_{u_k}^+\cap\partial\Omega_{u_k}^-$.

Recalling the definition of viscosity solution, for $\bar{x}\in\partial\Omega_{u_k}^+\backslash\partial\Omega_{u_k}^-$, one gets
\begin{equation*}
\begin{aligned}
(y_2^0+r_k\bar{x}_2)^2\lambda_+^2 &\geq |\nabla f_{\bar{t}}^+(\bar{x})|^2 \\
&=(y_2^0\lambda_+)^2(1+\tau\epsilon/2)^2+2(y_2^0\lambda_+)^2c_2\bar{t}\epsilon\partial_e\phi(\bar{x})+(y_2^0\lambda_-)^2 \\
& \quad +2(y_2^0)^2\lambda_+\lambda_-(1+c_2\bar{t}\epsilon\partial_e\phi(\bar{x})+\tau\epsilon/2)+O(\epsilon^2) \\
&>(y_2^0+r_k\bar{x}_2)^2\lambda_+^2,
\end{aligned}
\end{equation*}
and for $\bar{x}\in\partial\Omega_{u_k}^+\cap\partial\Omega_{u_k}^-$, one has
\begin{equation*}
\begin{aligned}
(y_2^0+r_k\bar{x}_2)^2(\lambda_+^2-\lambda_-^2) &\geq |\nabla f_{\bar{t}}^+(\bar{x})|^2 - |\nabla f_{\bar{t}}^-(\bar{x})|^2 \\
&>(y_2^0+r_k\bar{x}_2)^2(\lambda_+^2-\lambda_-^2).
\end{aligned}
\end{equation*}
These contradictions imply $\bar{t}=1$. Then
\begin{equation*}
\begin{aligned}
u_k^+ &\geq y_2^0\lambda_+(1+\tau\epsilon/2)(x\cdot\bm{e}+b_0+c_2\epsilon\phi)^+ \\
&\geq y_2^0\lambda_+(x\cdot\bm{e}+b_0+\bar{c}_2\epsilon)^+
\end{aligned}
\end{equation*}
in $B_{1/6}$, where $k$ is a suitable constant.

Set $a_1=a_0$, $b_1=b_0+\bar{c}_2\epsilon$ and it concludes the proof in this subcase.

\emph{Case 2. Improvement from above.}

	Suppose
	$$u_k^+(P)\leq y_2^0\lambda_+(2+a_0)^+-\frac{y_2^0\lambda_+(a_0-b_0)}{2},$$
	which means that $u_k^+(P)$ is closer to $\lambda_+(2+b_0)^+$ than to $\lambda_+(2+a_0)^+$. We will show
	$$u_k^+(x)\leq y_2^0\lambda_+(x\cdot\bm{e}+a_1)^+$$
	in a smaller ball centered at the origin.
	
	As in case 1, set
	$$f_t(x)=y_2^0\lambda_+(1-\tau\epsilon/2)(x\cdot\bm{e}+a_0-tc\epsilon\phi)^+$$
	with $\phi$ defined as in (\ref{phi}) and $c=c(\tau, \epsilon, a_0)$ a small constant satisfying
	$$u_k(x)\leq y_2^0\lambda_+(1-\tau\epsilon)(x\cdot\bm{e}+a_0)^+ \leq y_2^0\lambda_+(1-\tau\epsilon/2)(x\cdot\bm{e}+a_0-c\epsilon)^+ \leq f_t(x)$$
	for any $x\in B_{1/20}(Q)$ and $t\in[0,1]$.
	Define $\bar{t}$ as in case 1. Utilizing the maximum principle again for $\mathcal{L}_k$ we can deduce that $\bar{t}=1$, the property of $\phi$ gives the desired consequence. We omit the details here.
\end{proof}

Now we give the proof of Proposition \ref{comp}.

\begin{proof}[Proof of Proposition \ref{comp} for $0\leq l<\infty$]
	Utilizing lemma \ref{pbh1} we have that for $u_k^+$, there are constants $a_1, b_1$ and $\tilde{c}\in(0,1)$ with $a_1-b_1<\tilde{c}\left(\frac{1}{\lambda_-}-\frac{1}{\lambda_+}\right)\bar{\epsilon_k}$ such that
	$$y_2^0\lambda_+(x\cdot\bm{e}+b_1)^+\leq u_k^+\leq y_2^0\lambda_+(x\cdot\bm{e}+a_1)^+ \quad \text{in} \quad B_{\frac12\cdot\frac{1}{24}}(x_0)$$
	for any $x_0\in B_{1/2}$ and $B_{1/2}(x_0)\subset B_2$.
	
	Let $n>0$ be an integer that $\frac12\left( \frac1{24} \right)^{n+1}<\frac{\epsilon_k}{\bar{\epsilon_k}}\leq\frac12\left( \frac1{24} \right)^n$. We carry out the iteration and get that
	$$y_2^0\lambda_+(x\cdot\bm{e}+b_n)^+\leq u_k^+\leq y_2^0\lambda_+(x\cdot\bm{e}+a_n)^+ \quad \text{in} \quad B_{\frac12\cdot\frac{1}{24^n}}(x_0)$$
	for $a_n-b_n\leq\tilde{c}^n\left(\frac{1}{\lambda_-}-\frac{1}{\lambda_+}\right)\bar{\epsilon_k}$, where $n$ is a positive integer. Hence
	\begin{equation*}
	\begin{aligned}
	0\leq u_k^+(x)-y_2^0\lambda_+(x\cdot\bm{e}+b_n)^+&\leq y_2^0\lambda_+[(x\cdot\bm{e}+a_n)^+-(x\cdot\bm{e}+b_n)^+] \\
	&\leq y_2^0\lambda_+(a_n-b_n)^+ \\
	&\leq\tilde{c}^n y_2^0 \left(\frac{\lambda_+}{\lambda_-}-1\right)\bar{\epsilon_k}
	\end{aligned}
	\end{equation*}
	in $B_{\frac12\cdot\frac{1}{24^n}}(x_0)$, and we have
	$$|u_k^+-y_2^0\lambda_+(x\cdot\bm{e})^+-y_2^0\lambda_+b_n^+|\leq\tilde{c}^n y_2^0 \left(\frac{\lambda_+}{\lambda_-}-1\right)\bar{\epsilon_k}$$
	in $B_{\frac12\cdot\frac{1}{24^n}}(x_0)$.
	
	Now define a sequence $w_k$ by
	\begin{equation} \label{f5}
	w_k=
	\begin{cases}
	w_{+,k}=\frac{u_k(x)-y_2^0\lambda_+(x\cdot\bm{e})^+}{y_2^0\alpha_k\epsilon_k}, \quad x\in\Omega_{u_k}^+\cap B_1, \\
	w_{-,k}=\frac{u_k(x)+y_2^0\lambda_-(x\cdot\bm{e})^-}{y_2^0\beta_k\epsilon_k}, \quad x\in\Omega_{u_k}^-\cap B_1.
	\end{cases}
	\end{equation}
	Then $|w_{+,k}-\frac{y_2^0\lambda_+b_n^+}{y_2^0\alpha_k\epsilon_k}|\leq\frac{\tilde{c}^n \left(\frac{\lambda_+}{\lambda_-}-1\right)}{\alpha_k\epsilon_k}\bar{\epsilon_k}$, which leads to
	\begin{equation*}
	\begin{aligned}
	|w_{+,k}(x)-w_{+,k}(x_0)| &\leq \left|w_{+,k}(x)-\frac{\lambda_+b_n^+}{\alpha_k\epsilon_k}\right|+\left|\frac{\lambda_+b_n^+}{\alpha_k\epsilon_k}-w_{+,k}(x_0)\right| \\
	&\leq\frac{2 \left(\frac{\lambda_+}{\lambda_-}-1\right)}{\alpha_k\epsilon_k}\tilde{c}^n\bar{\epsilon_k}
	\end{aligned}
	\end{equation*}
	for any $x\in\Omega_{u_k}^+\cap B_{\frac12\cdot\frac{1}{24^n}}(x_0)$. Now choose $\gamma=\gamma(\tilde{c})$ such that $(\frac{1}{24})^\gamma=\tilde{c}$. Then for $\frac12(\frac{1}{24})^{n+1}\leq|x-x_0|<\frac12(\frac{1}{24})^n$,
	$$|w_{+,k}(x)-w_{+,k}(x_0)|\leq\frac{2 \left(\frac{\lambda_+}{\lambda_-}-1\right)}{\alpha_k\epsilon_k}\bar{\epsilon_k}\left(\frac{1}{24}\right)^{n\gamma}\leq C\left(\frac{1}{24}\right)^{\gamma(n+1)}\leq C|x-x_0|^\gamma.$$
	Hence
	$$|w_{+,k}(x)-w_{+,k}(x_0)|\leq C(n)|x-x_0|^\gamma \quad \text{in} \quad \Omega_{u_k}^+\cap(B_{\frac12\cdot\frac{1}{24^n}}(x_0)\backslash B_{\frac12\cdot\frac{1}{24^{n+1}}}(x_0)).$$
	Due to the arbitrariness of $x_0$ in $\Omega_{u_k}^+\cap B_{1/2}$, we conclude that
	$$|w_{+,k}(x)-w_{+,k}(y)|\leq C(n)|x-y|^\gamma$$
	for $x, y\in \Omega_{u_k}^+\cap B_{1/2}$ and $|x-y|>\frac{\epsilon_k}{\bar{\epsilon_k}}$.
	
	By the Ascoli-Arzela Theorem, there is a H\"older-continuous function $w_+\in C^{0,\gamma}\left(B_1\cap\{x\cdot\bm{e}>0\}\right)$ such that
	$$w_{+,k}\rightarrow w_+ \quad \text{in} \quad B_1\cap\{x\cdot\bm{e}>0\}$$
	uniformly under a subsequence. The detailed proof is referred to \cite{V23}, Theorem 7.15.
	
	Set
	$$\tilde{\Gamma}_k^{+}:=\bigl\{ (x,w_{+,k}(x)): x\in\overline{\Omega_{u_k}^{+}\cap B_{1/2}} \bigr\}.$$
	The H\"older convergence of $w_{+,k}$ together with the Ascoli-Arzela Theorem gives the Hausdorff convergence of $\tilde{\Gamma}_k^{+}$ to
	$$\tilde{\Gamma}_{+}:=\bigl\{ (x,w_+(x)):x\in B_{1/2} \bigr\}.$$
	
	Now set another function with $l$ defined as in (\ref{l}),
	\begin{equation*}
	h_k(x)=
	\begin{cases}
	\frac{H_{\alpha_k,\bm{e}}-H_{\lambda_+,\bm{e}}}{y_2^0\alpha_k\epsilon_k}\rightarrow l\frac{x\cdot\bm{e}}{\lambda_+^2} \quad \text{for} \quad x\cdot\bm{e}>0, \\
	\frac{H_{\alpha_k,\bm{e}}-H_{\lambda_+,\bm{e}}}{y_2^0\beta_k\epsilon_k}\rightarrow l\frac{x\cdot\bm{e}}{\lambda_-^2} \quad \text{for} \quad x\cdot\bm{e}>0.
	\end{cases}
	\end{equation*}
	Combining $v_k=w_k - h_k$
	we get the Hausdorff convergence of $\Gamma_k^+$ to $\Gamma_+$ and the pointwise convergence for $v_{+,k}$ to $v_+$. The argument for $v_{-,k}$ is symmetric.
	
\end{proof}

\subsubsection{The linearized problem}

After proper extension for $v_{\pm}$ in $B_{1/2}$, set
\begin{equation*}
v=v_++v_-.
\end{equation*}
Note that $v_-$ is not necessarily the negative part of $v$.

We next show that the limiting function $v$ solves the following linearized problem. Unlike the situation in \cite{PSV21}, the viscosity boundary conditions for $v_k$ do not remain constant, which involves the blow-up radius $r_k$ for $u_k$. Hence we have put an additional assumption that $r_k=O(\epsilon_k^2)$ as in the beginning of Section 3 to get over the technical difficulty. Moreover, in \cite{PSV21} the authors dealt with the special case $\bm{e}=(0,1)$, while we are assuming that $\bm{e}$ is arbitrary.

Let $u_k$ be a blow-up sequence with $\alpha_k, \epsilon_k$ satisfying (\ref{uk}), and let $v_k, l$ be as in (\ref{vk}) (\ref{l}). Notice that the free boundaries $\partial\{u_0>0\}\cup\partial\{u_0<0\}$ of the blow-up limit $u_0$ at an interior point will include both two-phase boundary points and one-phase boundary points, see Figure 10. Then the limiting function $v$, defined as above, solves the following linearized problem.

\begin{figure}[!h]
	\includegraphics[width=170mm]{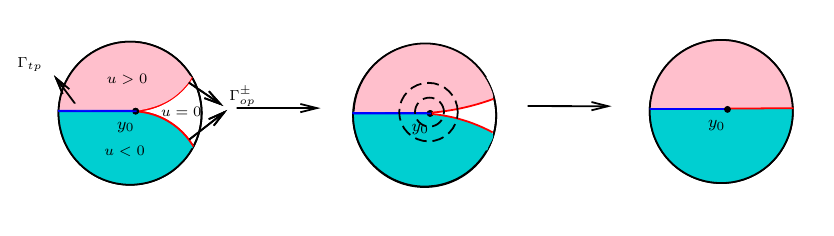}
	\caption{The blow-up at the branch point.}
\end{figure}

\begin{prop}  \label{lim1}
	(The limit linearized problem for $0\leq l<\infty$)
	In the case $0\leq l<\infty$, $v$ is a viscosity solution to a "two-membrane problem":
	\begin{equation} \label{tm}
	\begin{cases}
	\Delta v_{\pm}=0 \qquad\qquad\ \ \text{in} \quad \{(x\cdot\bm{e})^{\pm}>0\}\cap B_{1/2},\\
	v_+\leq v_- \qquad\qquad\ \ \text{on} \quad \{x\cdot \bm{e}=0\}\cap B_{1/2},\\
	\lambda_{\pm}^2\partial_ev_{\pm}+l\geq 0 \quad\ \ \, \text{on} \quad \{x\cdot \bm{e}=0\}\cap B_{1/2},\\
	\lambda_{\pm}^2\partial_ev_{\pm}+l=0 \quad\ \ \  \text{on} \quad \mathcal{J},\\
	\lambda_+^2\partial_ev_+=\lambda_-^2\partial_ev_- \ \ \, \, \text{on} \quad \mathcal{C},
	\end{cases}
	\end{equation}
	where $\partial_e$ denotes the derivative in the direction $\bm{e}$.
\end{prop}

Now we establish the convergence of $v_k$ at hand, the main difficulty here is to check the boundary condition in viscosity sense. We need to construct a series of comparison functions of $u_k$ to reach the desired conclusion. A useful touching lemma will be given in Appendix E for the completeness of the proof.

\begin{proof}
	We divide the proof into 3 steps.
	
	\emph{Step 1.} We expect to prove $\lambda_{\pm}^2\partial_e v_{\pm}+l\geq0$ on $B_{1/2}\cap\{x\cdot\bm{e}=0\}$.
	
	Next we focus on $v_-$.
	
	Suppose that there is a strictly subharmonic function $\bar{P}$ with $\partial_e\bar{P}=0$, the comparison function
	$$P=p(x\cdot\bm{e})+\bar{P}$$
	touches $v_-$ strictly from below at $x_0\in B_{1/2}\cap\{x\cdot\bm{e}=0\}$ with $\lambda_-^2p+l<0$.
	
	Exploiting lemma \ref{touch} in Appendix E, there is a sequence of $\{x_k\}\rightarrow x_0$, $x_k\in\partial\Omega_{u_k}^-$ and  a series of comparison functions $Q_k$ touching $-u_k^-$ from below at $x_k$, such that
	$$\nabla Q_k^-=-y_2^0\beta_k\bm{e}+y_2^0\beta_k\epsilon_k\nabla P_-(x_0)+O(\epsilon_k^2).$$
	Hence
	\begin{equation*}
	\begin{aligned}
	(y_2^0+r_kx_2)^2\lambda_-^2 &\leq|\nabla Q_k^-(x_k)|^2 \\
	&=(y_2^0\beta_k)^2+2(y_2^0\beta_k)^2\epsilon_kp+O(\epsilon_k^2).
	\end{aligned}
	\end{equation*}
	Noticing $l<\infty$, we have $\beta_k=\lambda_-+O(\epsilon_k)$. Recalling that $r_k=O(\epsilon_k^2)$, the above inequality leads to
	$$-\frac{l}{\lambda_-^2}=\lim_{k\rightarrow\infty}\frac{\lambda_-^2-\beta_k^2}{2\beta_k^2\epsilon_k}\leq\lim_{k\rightarrow\infty}\left[ p-\frac{r_kx_2(2y_2^0+r_kx_2)\lambda_-^2}{2(y_2^0\beta_k)^2\epsilon_k}+O(\epsilon_k) \right]=p<-\frac{l}{\lambda_-^2},$$
	which is a contradiction.
	
	Hence, $\lambda_-^2\partial_ev_-+l\geq 0$ on $\{x\cdot \bm{e}=0\}\cap B_{1/2}$. The argument for $v_+$ is symmetric as for $v_-$.
	
	\emph{Step 2.} We expect to prove $\lambda_{\pm}^2\partial_ev_{\pm}+l=0$ on $\mathcal{J}$.

	Again we focus on $v_-$.
	The previous steps show that we only need to check for a strictly superharmonic function $\bar{P}$ with $\partial_e\bar{P}=0$, that if $P=p(x\cdot\bm{e})+\bar{P}$ touches $v_-$ strictly from above at $x_0\in\mathcal{J}$, then $\lambda_-^2p+l\leq0$.
	
	In fact if not, because of lemma \ref{touch}, there is a sequence of $\{x_k\}\rightarrow x_0$, $x_k\in\partial\Omega_{u_k}^-$ and  a series of comparison functions $Q_k$ touching $-u_k^-$ from above at $x_k$, such that
	$$\nabla Q_k^-=-y_2^0\beta_k\bm{e}+y_2^0\beta_k\epsilon_k\nabla P_-(x_0)+O(\epsilon_k^2).$$
	Combined with the optimal conditions,
	\begin{equation*}
	\begin{aligned}
	(y_2^0+r_kx_2)^2\lambda_-^2 &\geq|\nabla Q_k^-(x_k)|^2 \\
	&=(y_2^0\beta_k)^2+2(y_2^0\beta_k)^2\epsilon_kp+O(\epsilon_k^2),
	\end{aligned}
	\end{equation*}
	and we have
	$$-\frac{l}{\lambda_-^2}=\lim_{k\rightarrow\infty}\frac{\lambda_-^2-\beta_k^2}{2\beta_k^2\epsilon_k}\geq\lim_{k\rightarrow\infty}\left[ p-\frac{r_kx_2(2y_2^0+r_kx_2)\lambda_-^2}{2(y_2^0\beta_k)^2\epsilon_k}+O(\epsilon_k) \right]=p>\frac{l}{\lambda_-}$$
	since $r_k=O(\epsilon_k^2)$, which is impossible.
	
	\emph{Step 3.} We expect to prove the fact that $\lambda_+^2\partial_ev_+=\lambda_-^2\partial_ev_-$ on $\mathcal{C}$.
	
	First we claim that $\lambda_+^2\partial_ev_+\leq\lambda_-^2\partial_ev_-$ on $\mathcal{C}$, and then a symmetric argument yields $\lambda_+^2\partial_ev_+\geq\lambda_-^2\partial_ev_-$ on $\mathcal{C}$, which leads to the conclusion.
	
	Suppose that there are $p,q\in\mathbb{R}$ with $\lambda_+^2p-\lambda_-^2q>0$ and a strictly subharnomic function $\bar{P}$ with $\partial_e\bar{P}=0$ such that
	$$P=p(x\cdot\bm{e})^+-q(x\cdot\bm{e})^-+\bar{P}$$
	touches $v_{\pm}$ strictly from below at $x_0\in\mathcal{C}$.
	
	By lemma \ref{touch}, there is a sequence of $\{x_k\}\rightarrow x_0$, $x_k\in\partial\Omega_{u_k}$ and  a series of comparison functions $Q_k$ touching $u_k$ from below at $x_k$, such that
	\begin{equation*}
	\nabla Q_k^+(x_k)=y_2^0\alpha_k\bm{e}+y_2^0\alpha_k\epsilon_kp\bm{e}+O(\epsilon_k^2)
	\end{equation*}
	and
	\begin{equation*}
	\nabla Q_k^-(x_k)=-y_2^0\beta_k\bm{e}-y_2^0\beta_k\epsilon_kq\bm{e}+O(\epsilon_k^2).
	\end{equation*}
	In particular, $P$ touches $v$ from below we have $q>0$ and thus $p>0$, which implies $x_k\notin\partial\Omega_{u_k}^-\backslash\partial\Omega_{u_k}^+$. It is remained to discuss the cases for $x_k\in\Gamma_{\text{op}}^-$ and $x_k\in\Gamma_{\text{tp}}$.
	
	Case 1. $x_k\in\partial\Omega_{u_k}^+\backslash\partial\Omega_{u_k}^-$.
	
	The definition of the viscosity solution gives
	\begin{equation*}
	\begin{aligned}
	(y_2^0+r_kx_2)^2\lambda_+^2 &\geq|\nabla Q_k^+(x_k)|^2 \\
	&=(y_2^0\alpha_k)^2+2(y_2^0\alpha_k)^2\epsilon_kp+O(\epsilon_k^2).
	\end{aligned}
	\end{equation*}
	This together with $r_k=O(\epsilon_k^2)$ implies
	$$\lambda_+^2p+l=\lambda_+^2\lim_{k\rightarrow\infty}\left(p+\frac{\alpha_k^2-\lambda_+^2}{2\alpha_k^2\epsilon_k}\right)\leq\lambda_+^2\lim_{k\rightarrow\infty}\frac{r_kx_2(2y_2^0+r_kx_2)\lambda_+^2}{2(y_2^0\alpha_k)^2\epsilon_k}=0,$$
	in contradiction with the fact $\lambda_+^2p+l>0$.
	
	Case 2. $x_k\in\partial\Omega_{u_k}^+\cap\partial\Omega_{u_k}^-$.
	
	In this case
	\begin{equation*}
	\begin{aligned}
	(y_2^0+r_kx_2)^2(\lambda_+^2-\lambda_-^2) &\geq|\nabla Q_k^+(x_k)|^2-|\nabla Q_k^-(x_k)|^2 \\
	&=(y_2^0)^2(\alpha_k^2-\beta_k^2)+2(y_2^0)^2\epsilon_k(\alpha_k^2p-\beta_k^2q)+O(\epsilon_k^2).
	\end{aligned}
	\end{equation*}
	Combined with the condition $r_k=O(\epsilon_k^2)$, it yields
	$$\frac{r_kx_2(2y_2^0+r_kx_2)(\lambda_+^2-\lambda_-^2)}{2(y_2^0)^2\epsilon_k}\geq\alpha_k^2p-\beta_k^2q+O(\epsilon_k^2),$$
	and thus
	$$0\geq\lambda_+^2p-\lambda_-^2q,$$
	in contradiction with the assumption $\lambda_+^2p-\lambda_-^2q>0$.
	
	This completes the proof.
\end{proof}

\subsubsection{Flatness decay}

This subsection follows as in \cite{PSV21} to get the improvement of flatness at branch point in a standard way. We sketch the key argument here.

\begin{prop}  \label{flat1}
	(Improvement of flatness: branch points)
	For every $L\geq\lambda_+\geq\lambda_->0$, $\gamma\in(0,1/2)$ and any $M>0$, there exist $\epsilon_1$, $C_1$ and $\rho\in(0,1/4)$ depending on $\gamma, L$ such that the following holds.
	
	Suppose that $u_k$ is a blow-up of the minimizer $u$ for $k$ large, and $0$ is a two-phase free boundary point of $u_k$. If
	$$\Vert u_k-H_{\alpha_k,\bm{e}} \Vert_{L^\infty(B_1)}\leq\epsilon_1$$
	with $0\leq\alpha_k-\lambda_+\leq M\Vert u-H_{\alpha_k,\bm{e}} \Vert_{L^\infty(B_1)}$, then there exist a unit vector $\bm{e}_k$ and a constant $\tilde{\alpha}_k\geq\lambda_+$ such that
	$$|\bm{e}_k-\bm{e}|+|\tilde{\alpha}_k-\alpha_k|\leq C_1\Vert u_k-H_{\alpha_k,\bm{e}}\Vert_{L^\infty(B_1)}$$
	and
	$$\Vert u_{\rho,k}-H_{\tilde{\alpha}_k,\bm{e}_k}\Vert_{L^\infty(B_1)}\leq\rho^\gamma\Vert u_k-H_{\alpha_k,\bm{e}}\Vert_{L^\infty(B_1)},$$
	where $u_{\rho,k}:=\frac{u_k(\rho x)}{\rho}=\frac{u(y_0+r_k\rho x)}{r_k\rho}$.
\end{prop}

\begin{proof}
	We argue by contradiction. Assume for $\Vert u_k-H_{\alpha_k,\bm{e}} \Vert_{L^\infty(B_1)}\leq \epsilon_k\rightarrow0$ and $0\leq\alpha_k-\lambda_+\leq M\epsilon_k$, we have that for any $\bm{e}_k\in\partial B_1$ and any $\tilde{\alpha}_k\geq\lambda_+$, there is a $\gamma\in(0,1/2)$ such that either
	$$|\bm{e}_k-\bm{e}|+|\tilde{\alpha}_k-\alpha_k|> C_1\Vert u_k-H_{\alpha_k,\bm{e}}\Vert_{L^\infty(B_1)}$$
	or
	$$\Vert u_{\rho,k}-H_{\tilde{\alpha}_k,\bm{e}_k}\Vert_{L^\infty(B_1)}>\rho^\gamma\Vert u_k-H_{\alpha_k,\bm{e}}\Vert_{L^\infty(B_1)}$$
	for any choice of $\rho\in(0,1/4)$ and $C_1$.
	
	Recall that $\frac{\alpha_k-\lambda_+}{\epsilon_k}\leq M$, the definition of $l$ gives $l\leq M\lambda_+.$
	
	Let $v_k$ be the sequence of functions defined as in (\ref{vk}), $v_k\rightarrow v$ and $\Vert v \Vert_{L^\infty(B_{1/2})}\leq1$. Then $v$ solves a two-membrane problem and thus using the regularity theorem in Appendix E.1, there exist $t\in\mathbb{R}$ and $p, q\in\mathbb{R}$ satisfying $\lambda_+^2p=\lambda_-^2q\geq-l$, such that for all $r\in(0,1/4)$,
	$$\sup_{B_r}\frac{\left|v(x)-v(0)-[t(x\cdot\bm{e}^\perp)+p(x\cdot\bm{e})^+-q(x\cdot\bm{e})^-]\right|}{r^{3/2}}\leq C(\lambda_{\pm}, M).$$
	For $\gamma\in(0,1/2)$, take small $r$ and $\rho$ depending on $C$ and $\gamma$ that $\rho<r$. Then
	$$\sup_{B_\rho}\left|v(x)-0-[t(x\cdot\bm{e}^\perp)+p(x\cdot\bm{e})^+-q(x\cdot\bm{e})^-]\right|\leq \frac{1}{y_2^0\alpha_k}\rho^{\gamma+1}.$$
	Recall the definition of $v$ we have
	\begin{equation*}
	\begin{cases}
	\frac{u_k(x)-y_2^0\alpha_k(x\cdot\bm{e})^+}{y_2^0\alpha_k\epsilon_k}-[t(x\cdot\bm{e}^\perp)+p(x\cdot\bm{e})^+-q(x\cdot\bm{e})^-]\leq \frac{1}{y_2^0\alpha_k}\rho^{\gamma+1} \quad \text{in} \quad \{u_k>0\}\cap B_\rho, \\
	\frac{u_k(x)+y_2^0\beta_k(x\cdot\bm{e})^-}{y_2^0\beta_k\epsilon_k}-[t(x\cdot\bm{e}^\perp)+p(x\cdot\bm{e})^+-q(x\cdot\bm{e})^-]\leq \frac{1}{y_2^0\alpha_k}\rho^{\gamma+1} \quad \text{in} \quad \{u_k<0\}\cap B_\rho.
	\end{cases}
	\end{equation*}
	Now set $\tilde{\alpha}_k:=\alpha_k(1+\epsilon_kp)+\delta_k\epsilon_k$ and $\bm{e}_k:=\frac{\bm{e}+\epsilon_kt\bm{e}^\perp}{\sqrt{1+\epsilon_k^2t^2}}$, where $\delta_k\rightarrow0$ is chosen such that $\tilde{\alpha}_k\geq\lambda_+$.
	
	Combining this with
	$$|\tilde{\alpha}_k-\alpha_k|=\epsilon_kp\alpha_k+\delta_k\epsilon_k\leq C_1\epsilon_k$$
	and
	$$|\bm{e}_k-\bm{e}|=|\epsilon_kt\bm{e}^\perp+o(\epsilon_k)|\leq C_1\epsilon_k,$$
	one can easily get
	$$|\tilde{\alpha}_k-\alpha_k|+|\bm{e}_k-\bm{e}|\leq C_1\epsilon_k$$
	for a constant $C_1$ independent of $k$.
	
	We claim that
	$$\Vert u_{\rho,k}-H_{\tilde{\alpha}_k,\bm{e}_k}\Vert_{L^\infty(B_1)}\leq\rho^\gamma\epsilon_k.$$
	In fact, for
	\begin{equation*}
	H_k:=
	\begin{cases}
	\frac{H_{\tilde{\alpha}_k,\bm{e}_k}-H_{\alpha_k,\bm{e}}}{y_2^0\alpha_k\epsilon_k} \quad \text{for} \quad x\cdot\bm{e}>0, \\
	\frac{H_{\tilde{\alpha}_k,\bm{e}_k}-H_{\alpha_k,\bm{e}}}{y_2^0\beta_k\epsilon_k} \quad \text{for} \quad x\cdot\bm{e}<0,
	\end{cases}
	\end{equation*}
	it is straightforward in $\{x\cdot\bm{e}>0\}$ that
	\begin{equation*}
	\begin{aligned}
	H_k &=\frac{y_2^0\tilde{\alpha}_k(x\cdot\bm{e}_k)^+ - y_2^0\tilde{\beta}_k(x\cdot\bm{e}_k)^- - y_2^0\alpha_k(x\cdot\bm{e})^+}{y_2^0\alpha_k\epsilon_k} \\
	&=\frac{\tilde{\alpha}_k\left(x\cdot\frac{\bm{e}}{\sqrt{1+\epsilon_k^2t^2}}\right)-\alpha_k(x\cdot\bm{e})+\tilde{\alpha}_k\left(x\cdot\frac{\epsilon_kt\bm{e}}{\sqrt{1+\epsilon_k^2t^2}}\right)}{\alpha_k\epsilon_k} \quad ( \text{for sufficiently large} \ k)\\
	&=\frac{\left(\alpha_k(1+\epsilon_kp)\right)\left(x\cdot\frac{\bm{e}}{\sqrt{1+\epsilon_k^2t^2}}\right)-\alpha_k(x\cdot\bm{e})+\left(\alpha_k(1+\epsilon_kp)\right)\left(x\cdot\frac{\epsilon_kt\bm{e}}{\sqrt{1+\epsilon_k^2t^2}}\right)+O(\delta_k\epsilon_k)}{\alpha_k\epsilon_k} \\
	&\rightarrow p(x\cdot\bm{e})^++t(x\cdot\bm{e}^\perp),
	\end{aligned}
	\end{equation*}
	and similarly
	$$H_k\rightarrow -q(x\cdot\bm{e})^-+t(x\cdot\bm{e}^\perp) \quad \text{in} \quad \{x\cdot\bm{e}<0\}.$$
	Hence
	$$H_k\rightarrow p(x\cdot\bm{e})^+-q(x\cdot\bm{e})^-+t(x\cdot\bm{e}^\perp),$$
	and we get
	$$\lim_{k\rightarrow\infty}\left|\frac{u_k(x)-H_{\alpha_k,\bm{e}}}{y_2^0\alpha_k\epsilon_k}-H_k\right|\leq \frac{1}{y_2^0\alpha_k}\rho^{\gamma+1} \quad \text{in} \quad B_\rho.$$
	Therefore
	$$|u_{\rho,k}(x)-H_{\tilde{\alpha}_k,\bm{e}_k}(x)|\leq\rho^\gamma\epsilon_k \quad \text{in} \quad B_1,$$
	a contradiction with our assumption. Thus the improvement of flatness is verified.
	
\end{proof}

\subsection{The case for non-branch points}

\subsubsection{Compactness} 
In this case Proposition \ref{comp} (The compactness of $v_k$) and Corollary \ref{cor1} still hold, and the proofs follow in a similar manner as in 3.1.1.  The arguments for non-branch case can also be found in \cite{SFS14}, so we just give the partial boundary Harnack lemma and omit the details here.

As in Section 3.1, thanks to the fact that $\Vert u_k-H_{\alpha_k,\bm{e}} \Vert_{L^\infty(B_1)}\leq\epsilon_k$ we know that
$$H_{\alpha_k,\bm{e}}(x-\frac{\epsilon_k}{y_2^0\beta_k}\bm{e})\leq u_k\leq H_{\alpha_k,\bm{e}}(x+\frac{\epsilon_k}{y_2^0\beta_k}\bm{e}).$$

\begin{lem} \label{pbh2}
	(Partial Boundary Harnack for non-branch case)
	Suppose that $\{u_k\}$ is a blow-up sequence of $u$ and L be the uniform Lipschitz constant. Then there exist constants $\bar{\epsilon}=\bar{\epsilon}(\lambda_{\pm},L)$, $M=M(\lambda_{\pm},L)$ and $\tilde{c}=\tilde{c}(\lambda_{\pm},L)\in(0,1)$ such that the following property holds.
	
	If
	$$H_{\alpha_k,\bm{e}}(x+b_0\bm{e})\leq u_k\leq H_{\alpha_k,\bm{e}}(x+a_0\bm{e}) \quad \text{in} \quad B_4$$
	for $a_0, b_0\in(-1/10,1/10)$ with $0\leq a_0-b_0\leq\bar{\epsilon}$ and for $\lambda_++M\epsilon\leq\alpha_k\leq2L$ with small $\epsilon$, then there exist $a_1,b_1\in(-1/10,1/10)$ with $0\leq a_1-b_1\leq\tilde{c}(a_0-b_0)$ such that for $x\in B_{1/6}$,
	\begin{equation} \label{f4}
	H_{\alpha_k,\bm{e}}(x+b_1\bm{e})\leq u_k\leq H_{\alpha_k,\bm{e}}(x+a_1\bm{e}).
	\end{equation}
	In particular, for $\epsilon_k$ defined as before, let $k$ be sufficiently large we have $a_0=\frac{\epsilon_k}{y_2^0\beta_k}$, $b_0=-\frac{\epsilon_k}{y_2^0\beta_k}$ and $a_1, b_1$ with $a_1-b_1\leq\bar{c}\frac{2\epsilon_k}{y_2^0\beta_k}$ satisfying (\ref{f4}).
\end{lem}

\subsubsection{The linearized problem}
As in 3.1.2 we set
\begin{equation*}
v=v_++v_-
\end{equation*}
after proper extension of $v_{\pm}$, and we require the additional assumption $r_k=O(\epsilon_k^2)$ to get over the technical difficulty. Moreover, we are dealing with a more general case with arbitrary $\bm{e}$ than in \cite{PSV21} with special $\bm{e}=(0,1)$. Here the free boundaries $\partial\{u_0>0\}\cup\partial\{u_0<0\}$ for the blow-up limit $u_0$ at an interior point will include only two-phase points, see Figure 11.

\begin{figure}[!h]
	\includegraphics[width=130mm]{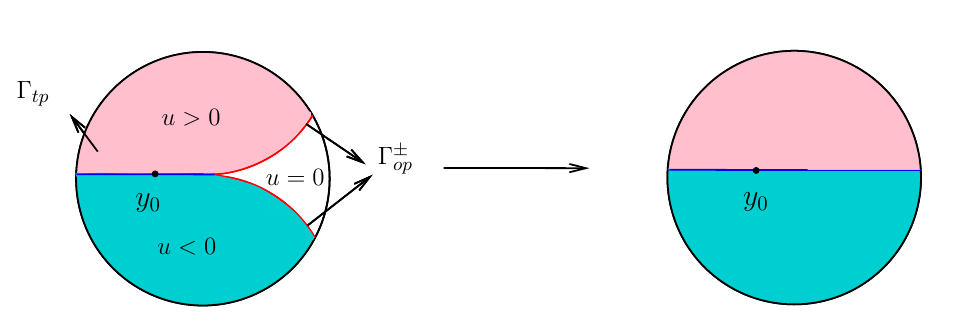}
	\caption{The blow-up at the non-branch point.}
\end{figure}

\begin{prop}  \label{lim2}
	(The limit linearized problem for $l=\infty$)
	In the case $l=\infty$, $v$ solves a "transmission problem":
	\begin{equation} \label{t}
	\begin{cases}
	\Delta v_{\pm}=0 \qquad\qquad\quad \text{in} \quad \{(x\cdot\bm{e})^{\pm}>0\}\cap B_{1/2},\\
	\alpha_{\infty}^2\partial_ev_+ = \beta_{\infty}^2\partial_ev_- \quad \text{on} \quad \{x\cdot \bm{e}=0\}\cap B_{1/2},
	\end{cases}
	\end{equation}
	where $\alpha_{\infty}=\lim_{k\rightarrow\infty}\alpha_k, \beta_{\infty}=\lim_{k\rightarrow\infty}\beta_k$ and $\partial_e$ denotes the derivative in the direction $\bm{e}$. Moreover, $\mathcal{J}=\varnothing$ and $\{x\cdot \bm{e}=0\}\cap B_{1/2}=\mathcal{C}$ in this case.
\end{prop}

\begin{proof}
	We divide the proof into 2 steps, using the touching lemma in Appendix E.
	
	\emph{Step 1.} We first show that $\mathcal{J}=\varnothing$, which means that $B_{1/2}\cap\{x\cdot\bm{e}=0\}=\mathcal{C}$.
	
	Assume not, then by the continuity of $v_{\pm}$ we know that the set $\mathcal{J}=\{v_+<v_-\}\subset B_{1/2}\cap\{x\cdot\bm{e}=0\}$ is relatively open. Define $\bm{e}^{\perp}$ to be the unit vector normal to $\bm{e}$, namely, $\bm{e}^{\perp}\cdot\bm{e}=0$. Without loss of generality suppose that there is a point $Y=(y_1,y_2)$ such that the segment
	$$(Y\cdot\bm{e}^{\perp}-\epsilon,Y\cdot\bm{e}^{\perp}+\epsilon)\subset\mathcal{J}$$
	for $Y\cdot\bm{e}=0$ and some $\epsilon\in\mathbb{R}$ small. See Figure 12.
	
	\begin{figure}[!h]
		\includegraphics[width=50mm]{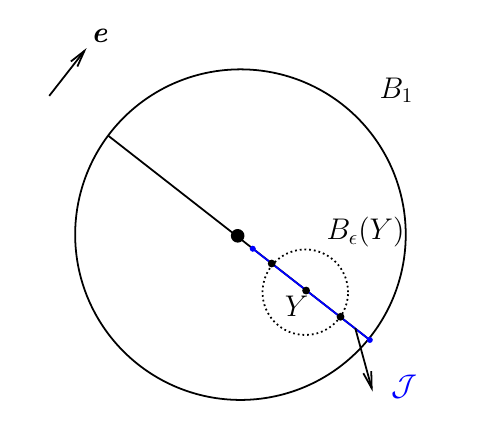}
		\caption{The figure of the ball $B_1$.}
	\end{figure}
	
	Recall that $\bm{e}=(e_1,e_2)$. Let $P(x)$ be the polynomial 
	\begin{equation*}
	\begin{aligned}
	P(x) &= A\bigl[ -|x\cdot\bm{e}^{\perp}-Y\cdot\bm{e}^{\perp}|^2+2(x\cdot\bm{e})^2 \bigr]-B(x\cdot\bm{e}) \\
	&=A\bigl[ -(-e_2 x_1+e_1 x_2+e_2 y_1-e_1 y_2)^2+2(e_1 x_1+e_2 x_2)^2 \bigr]-B(e_1 x_1+e_2 x_2)
	\end{aligned}
	\end{equation*}
	for $A, B\in\mathbb{R}$ to be determined. After calculating we have
	$$\partial_eP=2(x\cdot\bm{e})-B \quad \text{and} \quad \Delta P=2A>0.$$
	We first choose $A$ large enough such that $P<v^+$ on $\{|x\cdot\bm{e}^{\perp}-Y\cdot\bm{e}^{\perp}|=\epsilon\}\cap\{x\cdot\bm{e}=0\}$ and then choose $B$ larger to make sure $P<v^+$ on the ends of the mentioned segment, $i.e.$ on $\partial B_{\epsilon}(Y)\cap\{x\cdot\bm{e}=0\}$.
	
	Now translate $P(x)$ first down and then up to find a constant $C$ such that $P(x)+C$ touches $v^+$ from below at $x_0\in B_{\epsilon}(Y)\cap\{x\cdot\bm{e}\geq0\}$. By the strict subharmonicity of $P(x)$, we have $x_0\in B_{\epsilon}(Y)\cap\{x\cdot\bm{e}=0\}$.
	
	Utilizing Lemma \ref{touch} in Appendix E, there is a sequence of $\{x_k\}\rightarrow x_0$, $x_k\in\partial\Omega_{u_k}^+$ and a series of comparison functions $Q_k$ touching $u_k^+$ from below at $x_k$, such that
	$$\nabla Q_k^+(x_k)=y_2^0\alpha_k\bm{e}+y_2^0\alpha_k\epsilon_k\nabla P_+(x_0)+O(\epsilon^2).$$
	Combining with $x\in\mathcal{J}$, we know from $\liminf_{k\rightarrow\infty}dist(x,\partial\Omega_{u_k}^+\cap\partial\Omega_{u_k}^-)>0$ for $x\in\mathcal{J}$ that $x_k\in\partial\Omega_{u_k}^+\backslash\partial\Omega_{u_k}^-$. Then the definition of viscosity solutions gives
	\begin{equation*}
	\begin{aligned}
	(y_2^0+r_kx_2)^2\lambda_+^2 &\geq |\nabla Q_k^+(x_k)|^2 \\
	&=(y_2^0\alpha_k)^2+2(y_2^0\alpha_k)^2\epsilon_k\partial_eP(x_0)+O(\epsilon_k^2).
	\end{aligned}
	\end{equation*}
	Hence for $\alpha_k\geq\lambda_+$ and $r_k=O(\epsilon_k^2)$,
	\begin{equation*}
	\begin{aligned}
	\partial_eP(x_0)&=-B \\
	&\leq\frac{(y_2^0+r_kx_2)^2\lambda_+^2-(y_2^0\alpha_k)^2}{2(y_2^0\alpha_k)^2\epsilon_k}+O(\epsilon_k) \\
	&=\frac{\lambda_+^2-\alpha_k^2}{2\alpha_k^2\epsilon_k}+O(\epsilon_k) \\
	&\rightarrow -\frac{l}{\lambda_+^2}=-\infty.
	\end{aligned}
	\end{equation*}
	This contradiction implies $\mathcal{J}=\varnothing$.
	
	\emph{Step 2.} We next prove the transmission condition.
	
	Recall the optimal conditions, we need to verify the following facts
	\begin{equation} \label{3.4.1}
	\begin{cases}
	\alpha_{\infty}^2p-\beta_{\infty}^2q\leq 0 \quad \text{when} \ P \ \text{touches} \ v \ \text{from below}, \\
	\alpha_{\infty}^2p-\beta_{\infty}^2q\geq 0 \quad \text{when} \ P \ \text{touches} \ v \ \text{from above}.
	\end{cases}
	\end{equation}
	
	Suppose that there are $p, q\in\mathbb{R}$ with $\alpha_{\infty}^2p>\beta_{\infty}^2q$ and a strictly subharmonic function $\bar{P}(x)$ with $\partial_e\bar{P}(x)=0$ such that
	$$P(x)=p(x\cdot\bm{e})^+-q(x\cdot\bm{e})^-+\bar{P}(x)$$
	touches $v$ strictly from below at $x_0\in\{x\cdot\bm{e}=0\}\cap B_{1/2}$.
	By lemma \ref{touch} there is a sequence of $\{x_k=(x_{k,1},x_{k,2})\}\rightarrow x_0$, $x_k\in\partial\Omega_{u_k}$ and a series of comparison functions $Q_k$ touching $u_k$ from below at $x_k$,
	\begin{equation*}
	\begin{aligned}
	\nabla Q_k^+(x_k)&=y_2^0\alpha_k\bm{e}+y_2^0\alpha_k\epsilon_kp\bm{e}+O(\epsilon_k^2), \\
	\nabla Q_k^-(x_k)&=-y_2^0\beta_k\bm{e}-y_2^0\beta_k\epsilon_kq\bm{e}+O(\epsilon_k^2).
	\end{aligned}
	\end{equation*}
	In particular, $P$ touches $v$ from below and we have $q>0$ and thus $p>0$, which implies $x_k\notin\partial\Omega_{u_k}^-\backslash\partial\Omega_{u_k}^+$.
	
	Furthermore, we claim that $x_k\in\partial\Omega_{u_k}^-\cap\partial\Omega_{u_k}^+$. Otherwise $(y_2^0+r_kx_2)^2\lambda_+^2\geq|\nabla Q_k^+(x_k)|^2$, and we can reach a contradiction $p\rightarrow-\infty$ as above.
	
	Hence
	\begin{equation*}
	\begin{aligned}
	(y_2^0+r_kx_{k,2})^2(\lambda_+^2-\lambda_-^2) &\geq |\nabla Q_k^+(x_k)|^2-|\nabla Q_k^-(x_k)|^2 \\
	&=(y_2^0)^2(\alpha_k^2-\beta_k^2)+2(y_2^0)^2\epsilon_k(\alpha_k^2p-\beta_k^2q)+O(\epsilon_k^2)
	\end{aligned}
	\end{equation*}
	and we get that for $\alpha_k^2-\beta_k^2=\lambda_+^2-\lambda_-^2$ and $r_k=O(\epsilon_k^2)$,
	\begin{equation*}
	\begin{aligned}
	\alpha_k^2p-\beta_k^2q &\leq \frac{(\lambda_+^2-\lambda_-^2)^2-(\alpha_k^2-\beta_k^2)^2}{2\epsilon_k}+\frac{r_kx_2(2y_2^0+r_kx_2)(\lambda_+^2-\lambda_-^2)}{2(y_2^0)^2\epsilon_k}+O(\epsilon_k) \\
	&=O(\epsilon_k) \\
	&\rightarrow 0,
	\end{aligned}
	\end{equation*}
	a contradiction with $\alpha_{\infty}^2p>\beta_{\infty}^2q$. The second inequality in (\ref{3.4.1}) follows analogously.
\end{proof}

\subsubsection{Flatness decay}

\begin{prop}  \label{flat2}
	(Improvement of flatness: non-branch points)
	For every $L\geq\lambda_+\geq\lambda_->0$ and $\gamma\in(0,1/2)$, there exist $\epsilon_2$, $C_2$, $\bar{M}$ and $\rho\in(0,1/4)$ depending on $\gamma, L$ such that the following holds.
	
	Suppose that $u_k$ is a blow-up of the minimizer $u$ for $k$ large, and $0$ is a two-phase free boundary point of $u_k$. If
	$$\Vert u_k-H_{\alpha_k,\bm{e}} \Vert_{L^\infty(B_1)}\leq\epsilon_2$$
	with $\alpha_k-\lambda_+\geq \bar{M}\Vert u-H_{\alpha_k,\bm{e}} \Vert_{L^\infty(B_1)}$, then there exists a unit vector $\bm{e}_k$ and a constant $\tilde{\alpha}_k\geq\lambda_+$ such that
	$$|\bm{e}_k-\bm{e}|+|\tilde{\alpha}_k-\alpha_k|\leq C_2\Vert u_k-H_{\alpha_k,\bm{e}}\Vert_{L^\infty(B_1)}$$
	and
	$$\Vert u_{\rho,k}-H_{\tilde{\alpha}_k,\bm{e}_k}\Vert_{L^\infty(B_1)}\leq\rho^\gamma\Vert u_k-H_{\alpha_k,\bm{e}}\Vert_{L^\infty(B_1)},$$
	where $u_{\rho,k}:=\frac{u_k(\rho x)}{\rho}=\frac{u(y_0+r_k\rho x)}{r_k\rho}$.
	
\end{prop}

\begin{proof}
	Assume $\{\alpha_k\}$ and $\{M_k\}\rightarrow\infty$ satisfy $\Vert u_k-H_{\alpha_k,\bm{e}}\Vert_{L^\infty(B_1)}\leq\epsilon_k\rightarrow0$ and $\alpha_k-\lambda_+\geq M_k\epsilon_k$, but for any $\bm{e}_k\in\partial B_1$ and any $\tilde{\alpha}_k\geq\lambda_+$, there is a $\gamma\in(0,1/2)$ such that either
	$$|\bm{e}_k-\bm{e}|+|\tilde{\alpha}_k-\alpha_k|> C_2\Vert u_k-H_{\alpha_k,\bm{e}}\Vert_{L^\infty(B_1)}$$
	or
	$$\Vert u_{\rho,k}-H_{\tilde{\alpha}_k,\bm{e}_k}\Vert_{L^\infty(B_1)}>\rho^\gamma\Vert u_k-H_{\alpha_k,\bm{e}}\Vert_{L^\infty(B_1)}$$
	for any choice of $\rho\in(0,1/4)$ and $C_2$. This implies $l=\infty$ and $v$ solves a transmission problem. We conclude the proof as in Proposition \ref{flat1} by using the regularity theorem F.2 in Appendix F.
\end{proof}

\subsection{Improvement of flatness}
We summarize the above process and get the following proposition.

\begin{prop}\label{flat}
	(Flatness decay)
	For every $L\geq\lambda_+\geq\lambda_->0$ and $\gamma\in(0,1/2)$, there exist $\epsilon_0$, $C$ and $\rho\in(0,1/4)$ depending on $\gamma, L$ such that the following holds.
	
	Suppose that $u_k$ is a blow-up of the minimizer $u$ for $k$ large, and $0$ is a two-phase free boundary point of $u_k$. If
	$$\Vert u_k-H_{\alpha_k,\bm{e}} \Vert_{L^\infty(B_1)}\leq\epsilon_0$$
	with $\alpha_k\geq\lambda_+$, then there exists a unit vector $\bm{e}_k$ and a constant $\tilde{\alpha}_k\geq\lambda_+$ such that
	$$|\bm{e}_k-\bm{e}|+|\tilde{\alpha}_k-\alpha_k|\leq C\Vert u_k-H_{\alpha_k,\bm{e}}\Vert_{L^\infty(B_1)}$$
	and
	$$\Vert u_{\rho,k}-H_{\tilde{\alpha}_k,\bm{e}_k}\Vert_{L^\infty(B_1)}\leq\rho^\gamma\Vert u_k-H_{\alpha_k,\bm{e}}\Vert_{L^\infty(B_1)},$$
	where $u_{\rho,k}:=\frac{u_k(\rho x)}{\rho}=\frac{u(y_0+r_k\rho x)}{r_k\rho}$.
\end{prop}

\begin{proof}
	Combine Proposition \ref{flat1} and \ref{flat2}. Take $M$ in Proposition \ref{flat1} to be $M=\bar{M}$, where $\bar{M}$ is the constant in Proposition \ref{flat2}. Set $\epsilon_0=\min\{\epsilon_1/2, \epsilon_2/2\}$, then we can draw the conclusion.
\end{proof}

\section{Proof of the main result}

In this section, we derive the $C^{0,\eta}$ regularity of $\alpha(x)$ and $\bm{e}(x)$ , and verify that $u_+$ and $u_-$ solve the classical one-phase Bernoulli problems respectively in $\{u^\pm>0\}$. Then we take virtue of the regularity results for free boundaries in \cite{S11} for one-phase problem to get the $C^{1,\eta}$ regularity of $\partial\{u^\pm>0\}$.

First we establish the uniqueness of the blow-up limit utilizing flatness decay.

\begin{lem}  \label{uni}
	(Uniqueness of the blow-up limit)
	Suppose that $u$ is a local minimizer of $J_{\rm a,tp}$ in $B\Subset D$. Then at every point $x_0\in\partial\Omega_{u}^{\pm}\cap B$, there is a unique blow-up limit at $x_0$.
\end{lem}

\begin{proof}
	For $x_0\in\Gamma_{\text{op}}^{\pm}$, $u$ is locally a minimizer of a one-phase functional $J_{a,acf}$, and we can apply the results in \cite{W99}\cite{CJK04} or in \cite{V23} for one-phase Bernoulli problem and deduce that the blow-up limit is unique. So we just have to consider the case $x_0=(x_1^0, x_2^0)\in\Gamma_{\text{tp}}$.
	
	Suppose that there is a two-plane solution $H_{\alpha,\bm{e}}$ satisfying that for any $\epsilon_0$, there exists $r_0$ such that
	$$\Vert u_{x_0,r_0}-H_{\alpha,\bm{e}}\Vert_{L^\infty(B_1)}\leq\epsilon_0.$$
	
	Utilizing the flatness decay, for any integer $n>0$, there are $\alpha_n$ and $\bm{e}_n$ satisfying $|\alpha_n-\alpha_{n-1}|+|\bm{e}_n-\bm{e}_{n-1}|\leq C\epsilon_{n-1}=C\rho^{(n-1)\gamma}\epsilon_0$, such that
    $$\Vert u_{x_0,\rho^n}-H_{\alpha_n,\bm{e}_n}\Vert_{L^\infty(B_1)}\leq\rho^{n\gamma}\epsilon_0=\epsilon_n$$
	for the constants $C$, $\gamma$ be as in Proposition \ref{flat} and for $\rho\in(0,1/4)$.
	
	Let $r<r_0$ be arbitrary and $n$ be the integer such that $\rho^{n+1}<r\leq\rho^{n}$. We denote the limit of the Cauchy sequences $\{\alpha_n\}$ and $\{\bm{e}_n\}$ to be $\alpha_0$ and $\bm{e}_0$ respectively, and the direct calculation gives
	$$\Vert u_{x_0,\rho^n}-H_{\alpha_0,\bm{e}_0}\Vert_{L^\infty(B_1)}\leq C\rho^{n\gamma}.$$
	Now for any $r\in(\rho^{n+1},\rho^n]$, using the fact that $\rho\in(0,1/4)$, there must exist $\tau\in(0,1]$ such that $r=\tau\rho^n$. Hence,
	$$\Vert u_{x_0,r}-H_{\alpha_0,\bm{e}_0} \Vert_{L^\infty(B_1)}\leq \frac C{\tau^\gamma} r^\gamma \quad \text{for} \quad r\in(\rho^{n+1},\rho^n],$$
	and by the arbitrariness of $r$,
	$$\Vert u_{x_0,r}-H_{\alpha_0,\bm{e}_0} \Vert_{L^\infty(B_1)}\leq Cr^\gamma$$
	for any $r$ small enough. The uniqueness of the blow-up limit follows directly.
\end{proof}

Next we derive the $C^{0,\eta}$ regularity of $\alpha(x)$ and $\bm{e}(x)$. Here we only consider the case $x_0\in\Gamma_{\text{tp}}$. For $x_0\in \Gamma_{\text{op}}^+$ we invoke that $u_0(x)=x_2^0\lambda_+(x\cdot\bm{e}(x_0))^+$, and the $C^{1,\eta}$ regularity for the one-phase free boundary follows directly from \cite{V23}. The case for $x_0\in\Gamma_{\text{op}}^-$ is quite similar.

\begin{lem}  \label{holder}
	Suppose that $u$ is a local minimizer of $J_{\rm a,tp}$ in $B\Subset D$, and $u_0$ is a blow-up limit at $x_0\in\Gamma_{\rm{tp}}\cap B$ of the form (\ref{bu}). 
	Then there exists $0<\eta<1$ such that for every open set $D'\Subset B$, there is a constant $C=C(D',\lambda_{\pm})$ such that for every $x_0, y_0\in\Gamma_{\rm{tp}}\cap D'$,
	\begin{equation*}
	|\alpha(x_0)-\alpha(y_0)|\leq C|x_0-y_0|^\eta, \quad
	|\bm{e}(x_0)-\bm{e}(y_0)|\leq C|x_0-y_0|^\eta.
	\end{equation*}
\end{lem}

\begin{proof}
	Consider $x_0\in\Gamma_{\text{tp}}$ and the blow-up limit  $H_{\alpha(x_0),\bm{e}(x_0)}$ at $x_0$.
	The flatness decay together with Proposition \ref{prop5} shows that there are $\alpha(x_0)$, $\bm{e}(x_0)$ and small $r$ such that
	$$\Vert u_{y_0,r}-H_{\alpha(x_0),\bm{e}(x_0)}\Vert_{L^\infty(B_1)}\leq Cr^\gamma$$
	for any $y_0\in B_\rho(x_0)$ and $\gamma\in(0,1/2)$. A covering argument implies the validity of the above estimate for all $x_0\in\Gamma_{\text{tp}}\cap D'$.
	
	Now set $r:=|x_0-y_0|^{1-\eta}<r_0$ for $\eta:=\frac{\gamma}{1+\gamma}$ and any $x_0, y_0\in\Gamma_{\text{tp}}$. Then
	\begin{equation*}
	\begin{aligned}
	&\quad \Vert H_{\alpha(x_0),\bm{e}(x_0)}-H_{\alpha(y_0),\bm{e}(y_0)} \Vert_{L^\infty(B_1)} \\
	&\leq \Vert u_{x_0,r}-H_{\alpha(x_0),\bm{e}(x_0)} \Vert_{L^\infty(B_1)}+\Vert u_{x_0,r}-u_{y_0,r} \Vert_{L^\infty(B_1)}+\Vert u_{y_0,r}-H_{\alpha(y_0),\bm{e}(y_0)} \Vert_{L^\infty(B_1)} \\
	&\leq Cr^\gamma+\frac{L|x_0-y_0|}{r}+Cr^\gamma \\
	&=(L+2C)|x_0-y_0|^\eta,
	\end{aligned}
	\end{equation*}
	which means
	\begin{equation*}
	\begin{aligned}
	&\quad |x_2^0\bigl(\alpha(x_0)(x\cdot\bm{e}(x_0))^+ - \beta(x_0)(x\cdot\bm{e}(x_0))^-\bigr) - y_2^0\bigl(\alpha(y_0)(x\cdot\bm{e}(y_0))^+ - \beta(y_0)(x\cdot\bm{e}(y_0))^-\bigr)| \\
	&\leq(L+2C)|x_0-y_0|^\eta
	\end{aligned}
	\end{equation*}
	in $B_1$, and we get further that
	$$|x_2^0\alpha(x_0)(x\cdot\bm{e}(x_0))^+ - y_2^0\alpha(y_0)(x\cdot\bm{e}(y_0))^+|\leq(L+2C)|x_0-y_0|^\eta$$
	in $B_1$.
	
	Insert that for any unit vector $\bm{e}_1, \bm{e}_2\in\mathbb{R}^n$,
	$$|e_1-e_2|\leq C(n)\Vert (x\cdot\bm{e}_1)^+-(x\cdot\bm{e}_2)^+\Vert_{L^\infty(B_1)},$$
	and it yields
	\begin{equation*}
	\begin{aligned}
	|x_2^0\alpha(x_0)\bm{e}(x_0)-y_2^0\alpha(y_0)\bm{e}(y_0)| &\leq \Vert x_2^0\alpha(x_0)(x\cdot\bm{e}(x_0))^+ - y_2^0\alpha(y_0)(x\cdot\bm{e}(y_0))^+ \Vert_{L^\infty(B_1)} \\
	&\leq C|x_0-y_0|^\eta
	\end{aligned}
	\end{equation*}
	by taking $\bm{e}_1=x_2^0\alpha(x_0)\bm{e}(x_0)$, $\bm{e}_2=y_2^0\alpha(y_0)\bm{e}(y_0)$. Taking square of both sides of the above inequality, it leads to
	$$|x_2^0\alpha(x_0)-y_2^0\alpha(y_0)|\leq C|x_0-y_0|^\eta.$$
	Similarly we get
	$$|x_2^0\beta(x_0)-y_2^0\beta(y_0)|\leq C|x_0-y_0|^\eta.$$
	
	Now since
	\begin{equation*}
	\begin{aligned}
	|x_2^0\alpha(x_0)\bm{e}(x_0)-y_2^0\alpha(y_0)\bm{e}(y_0)|&\geq b|\alpha(x_0)\bm{e}(x_0)-\alpha(y_0)\bm{e}(y_0)| \\
	&\geq b\left[ |\alpha(x_0)||\bm{e}(x_0)-\bm{e}(y_0)|+|\bm{e}(y_0)||\alpha(x_0)-\alpha(y_0)| \right]
	\end{aligned}
	\end{equation*}
	for $x_2^0, y_2^0\geq b$, we arrrive at
	\begin{equation*}
	\begin{aligned}
	|\bm{e}(x_0)-\bm{e}(y_0)|&\leq \frac{1}{|\alpha(x_0)|}\bigl( C|x_0-y_0|^\eta+|\bm{e}(y_0)||\alpha(x_0)-\alpha(y_0)| \bigr) \\
	&\leq C|x_0-y_0|^\eta.
	\end{aligned}
	\end{equation*}
	This completes the proof.
\end{proof}

Consider $u_+$ and $u_-$ respectively, it is easy to see that $u_{\pm}$ solve the classical one-phase Bernoulli problems. We sketch the proof here for the sake of completeness.

\begin{lem}
	Let $u$ be a local minimizer of $J_{\rm a,tp}$ in $D'\Subset D$. Then there are $C^{0,\eta}$ boundary functions $\alpha:\partial\Omega_{u}^+\rightarrow\mathbb{R}$ and $\beta:\partial\Omega_{u}^-\rightarrow\mathbb{R}$ such that
	\begin{equation*}
	\alpha\geq\lambda_+, \quad
	\beta\geq\lambda_-,
	\end{equation*}
	and that $u^+=\max\{u,0\}$, $u^-=-\min\{u,0\}$ solve the following one-phase problems respectively,
	\begin{equation}\label{4.1}
	\begin{cases}
	\mathcal{L} u^+=0 \qquad\ \ \text{in} \quad \Omega_{u}^+, \\
	|\nabla u^+|=x_2\alpha \quad \text{on} \quad \partial\Omega_{u}^+,
	\end{cases}
	\end{equation}
	and
	\begin{equation}\label{4.2}
	\begin{cases}
	\mathcal{L} u^-=0 \qquad\ \ \text{in} \quad \Omega_{u}^-, \\
	|\nabla u^-|=x_2\beta \quad \text{on} \quad \partial\Omega_{u}^-.
	\end{cases}
	\end{equation}
\end{lem}

\begin{proof}
	We only consider $u^+$ as follows.
	
	Clearly $\mathcal{L} u^+=0$ in $\Omega_{u}^+$. By the flatness decay, Proposition \ref{flat}, we know that there exists a constant $C$ such that
	$$\Vert u_{x_0,r}^+-H_{\alpha(x_0),\bm{e}(x_0)} \Vert_{L^\infty(B_1)}\leq Cr^\gamma$$
	for $x_0\in\partial\Omega_{u}^+$ and small $r$, which means
	$$|u^+(x_0+rx)-H_{\alpha(x_0),\bm{e}(x_0)}(rx)|\leq Cr^{\gamma+1}$$
	for all $x\in B_1$ and small $r$. Now for $y\in B_r(x_0)\cap\{u^+>0\}$,
	$$|u^+(y)-x_2^0\alpha(x_0)(y-x_0)\cdot\bm{e}(x_0)|\leq C|y-x_0|^{\gamma+1}$$
	for small $r$, and thus
	$$\frac{u^+(y)-u^+(x_0)}{|y-x_0|}\leq x_2^0\alpha(x_0)+C|y-x_0|^\gamma.$$
	In particular, $u^+$ is differentiable in $\Omega_{u}^+$ up to $x_0$, and
	$$|\nabla u^+(x_0)|=x_2^0\alpha(x_0)$$
	for $x_0\in\partial\Omega_{u}^+$.
	
	On the other hand if $x_0\in\Gamma_{\text{op}}^+$, then $|\nabla u^+(x_0)|=x_2^0\lambda_+$ in the viscosity sense, thus
	$$\alpha(x_0)=\lambda_+$$
	for $x_0\in\Gamma_{\text{op}}^+$. Remember that $\tilde{\alpha}(x_0)=x_2^0\alpha(x_0)$ is $C^{0,\eta}$ for $x_0\in\Gamma_{\text{tp}}$ by the previous lemma, we only need to prove at a branch point $x_0$ that $\alpha(x_0)=\lambda_+$. In fact for such $x_0$, there exists a sequence $\{x_k\}\in\Gamma_{\text{op}}^+$ such that $x_k=(x_1^k, x_2^k)\rightarrow x_0=(x_1^0,x_2^0)$. Let $\{y_k\}\in\Gamma_{\text{tp}}$ be another sequence such that $dist(x_k, \Gamma_{\text{tp}})=|x_k-y_k|$. Set $r_k:=|x_k-y_k|$ and $u_k(x)=\frac{u^+(x_k+r_kx)}{r_k}$, then $u_k$ is a viscosity solution of
	\begin{equation*}
	\begin{cases}
	\mathcal{L}_{r_k}u_k=0 \qquad\ \text{in} \quad \Omega_{u_k}^+\cap B_1, \\
	|\nabla u_k|=x_2^k\lambda_+ \quad \text{on} \quad \partial\{u_k>0\}\cap B_1.
	\end{cases}
	\end{equation*}
	Since $u_k$ are uniformly Lipschitz, the limit function $u_\infty$ is a viscosity solution of
	\begin{equation*}
	\begin{cases}
	\Delta u_\infty=0 \qquad\quad \, \text{in} \quad \Omega_{u_\infty}^+\cap B_1, \\
	|\nabla u_\infty|=x_2^0\lambda_+ \quad \text{on} \quad \partial\{u_\infty>0\}\cap B_1.
	\end{cases}
	\end{equation*}
	Hence from the uniqueness of blow-up limit we have
	$$u_\infty(x)=x_2^0\alpha(x_0)(x\cdot\bm{e}(x_0))^+$$
	and
	$$\alpha(x_0)=\lambda_+.$$
	So we get the desired conclusion.
\end{proof}

Now we are fully prepared to prove the main theorem.

\begin{proof}[Proof of Theorem \ref{thm1}]
	We only consider points $x_0\in\Gamma_{\text{tp}}$. Due to the classification of blow-up limits at two-phase points, we have that for any $\epsilon>0$, there exists $r_0$ such that
	$$\Vert u_{x_0,r_0}-H_{\alpha,\bm{e}} \Vert_{L^\infty(B_1)}<\epsilon,$$
	and $u^{\pm}$ solves (\ref{4.1}) and (\ref{4.2}) respectively. Then the regularity result of free boundary for one-phase problem in \cite{S11} gives that $\partial\Omega_{u}^{\pm}$ are locally $C^{1,\eta}$ graphs.
\end{proof}

\appendix

\section{The study on the free boundary conditions}

In this section we verify the free boundary conditions of the minimizer $u$ for $J_{\text{a,tp}}$ in $D$.

\begin{prop}
	Suppose that $u$ is a minimizer of $J_{\rm a,tp}$ in $D$ mentioned in Section 1. Then $u$ solves
	\begin{equation}\label{A1}
	\Delta u -\frac{\partial_2 u}{x_2}=0 \quad \text{in} \quad \{u\neq0\},
	\end{equation}
	and satisfies the free boundary conditions
	\begin{equation}\label{A2}
	\begin{cases}
	|\nabla u^+|^2-|\nabla u^-|^2=(x_2)^2(\lambda_+^2-\lambda_-^2) \ \ \quad \text{on} \quad \Gamma_{\rm{tp}}, \\
	|\nabla u^{\pm}| = x_2\lambda_{\pm} \qquad\qquad\qquad\qquad\qquad \text{on} \quad \Gamma_{\rm{op}}^{\pm}, \\
	|\nabla u^{\pm}|\geq x_2\lambda_{\pm} \qquad\qquad\qquad\qquad\qquad \text{on} \quad \Gamma_{\rm{tp}}	.
	\end{cases}
	\end{equation}
\end{prop}

\begin{proof}
	We only prove the free boundary conditions (\ref{A2}) of $u$, and (\ref{A1}) is due to the Euler-Lagrange equation.
	
	Let $\phi_t(x)=x+t\xi(x)$ for $\xi=(\xi_1,\xi_2)\in C_0^\infty(D\cap\{x_2>0\};\mathbb{R}^2)$ and $t\neq0$. Define $u_t\in \mathcal{K}$ by $u_t(\phi_t(x))=u(x)$. Since $u$ is a minimizer of $J_{\text{a,tp}}$ in $D$,
	\begin{equation*}
	\begin{aligned}
	0&=\frac{d}{dt}J_{\text{a,tp}}(u_t)|_{t=0} \\
	&=\frac{d}{dt}\Bigg[\int_D \Bigg( \frac{|\nabla u|^2-2t\nabla uD\xi\nabla u +o(t)}{x_2+t\xi_2}+\left(x_2+t\xi_2(x)\right)\left(\lambda_+^2I_{\{u>0\}}+\lambda_-^2I_{\{u<0\}}\right) \\
	& \quad + \frac{tdiv\xi}{x_2+t\xi_2}\left(|\nabla u|^2-2t\nabla uD\phi\nabla u+o(t)\right) + t\left(x_2+t\xi_2\right)div\xi\left(\lambda_+^2I_{\{u>0\}}+\lambda_-^2I_{\{u<0\}}\right)+o(t) \Bigg)dX \Bigg]\Bigg|_{t=0} \\
	&=\int_D\Big[ \frac{-\xi_2(x)}{x_2}|\nabla u|^2+\frac{-2\nabla uD\xi\nabla u}{x_2}+\left(\xi_2(x)+x_2\right)\left(\lambda_+^2I_{\{u>0\}}+\lambda_-^2I_{\{u<0\}}\right) + \frac{x_2div\xi}{(x_2)^2}|\nabla u|^2  \Big]dX.
	\end{aligned}
	\end{equation*}
	Integrating by parts,
	\begin{equation*}
	\begin{aligned}
	0&=\int_D\left(\frac{|\nabla u|^2}{x_2}div\xi-\frac{2}{x_2}\nabla uD\xi\nabla u+\frac{-\xi_2(x)}{(x_2)^2}|\nabla u|^2+\xi_2(x)\left(\lambda_+^2I_{\{u>0\}}+\lambda_-^2I_{\{u<0\}}\right)\right)dX \\
	&\quad +\int_{D\cap\{u>0\}}x_2\lambda_+^2div\xi dX +\int_{D\cap\{u<0\}}x_2\lambda_-^2div\xi dX \\
	&=\int_{(\{u>0\}\cup\{u<0\})\cap D}\left( \frac1{x_2}div\left(|\nabla u|^2\xi-2(\xi\cdot\nabla u)\nabla u\right)+2\frac{\xi\cdot\nabla u}{(x_2)^2}\partial_2 u \right)dX \\
	&\quad +\int_D\frac{-\xi_2(x)}{(x_2)^2}|\nabla u|^2 dX +\int_{D\cap\{u>0\}}\xi_2(x)\lambda_+^2dX+\int_{D\cap\{u<0\}}\xi_2(x)\lambda_-^2dX \\
	&\quad +\lim_{\delta\rightarrow0}\int_{D\cap\partial\{u>\delta\}}x_2\lambda_+^2\xi\cdot\nu_1 dS-\int_{D\cap\{u>0\}} \xi_2(x)\lambda_+^2dX \\
	&\quad +\lim_{\epsilon\rightarrow0}\int_{D\cap\partial\{u<-\epsilon\}}x_2\lambda_-^2\xi\cdot\nu_2 dS-\int_{D\cap\{u>0\}} \xi_2(x)\lambda_-^2dX \\
	&=\lim_{\delta\rightarrow0}\int_{D\cap\partial\{u>\delta\}}\left( -\frac1{x_2}|\nabla u|^2+x_2\lambda_+^2 \right)(\xi\cdot\nu_1)dS + \lim_{\epsilon\rightarrow0}\int_{D\cap\partial\{u<-\epsilon\}}\left( -\frac1{x_2}|\nabla u|^2+x_2\lambda_-^2 \right)(\xi\cdot\nu_2)dS
	\end{aligned}
	\end{equation*}
	where $\delta$, $\epsilon>0$ and $\nu_1$, $\nu_2$ are the outward normal vectors to $\partial\{u>\delta\}$ and $\partial\{u<-\epsilon\}$. This gives the first two equalities in (\ref{A2}).
	
	Define $v_t(x)=u^+(x+t\xi(x))-u^-(x)$ and $w_t(x)=u^+(x)-u^-(x+t\xi(x))$, and the domain variation gives straightforward the last inequality in (\ref{A2}).
\end{proof}

\section{The non-degeneracy of the minimizer}

We give the detailed proof for the non-degeneracy of the minimizer.

\begin{proof}[Proof of (1) in Proposition \ref{prop2}]
	We only prove the conclusion for $u^+$. Denote $B_r=B_r(x_0)$ for any $r\leq\frac b2$. Then for any $x=(x_1,x_2)\in B_r$, $x_2\geq\frac b2$. Set $\gamma=\frac{1}{r}\left(\fint_{\partial B_r}(u^+)^2dS\right)^{1/2}$.
	
	Recalling that $\mathcal{L}=\Delta-\frac{1}{x_2}\partial_2$, we introduce an auxiliary function $v$ satisfying
	\begin{equation*}
	\begin{cases}
	\mathcal{L}v=0 \quad \text{in} \quad (B_r\backslash B_{\kappa r})\cap\{u>0\}, \\
	v=0 \quad\ \; \text{in} \quad B_{\kappa r}\cap \{u>0\}, \\
	v=u \quad\ \ \text{in} \quad (B_r\cap\{u\leq0\})\cup\partial B_r.
	\end{cases}
	\end{equation*}
	In fact the existence of the solution to this Dirichlet boundary problem can be attained by approximation of
	\begin{equation*}
	\begin{cases}
	\mathcal{L}v_\epsilon=0 \quad \text{in} \quad (B_r\backslash B_{\kappa r})\cap\{u>\epsilon\}, \\
	v_\epsilon=\epsilon \quad\ \ \text{in} \quad B_{\kappa r}\cap \{u>\epsilon\}, \\
	v_\epsilon=u \quad\ \; \text{in} \quad (B_r\cap\{u\leq\epsilon\})\cup\partial B_r,
	\end{cases}
	\end{equation*}
	which is solvable, for $\{u=\epsilon\}$ is $a.e.$ a smooth contact set.

	We obtain
	\begin{equation*}
	\begin{aligned}
	\int_{B_r}\frac{|\nabla u|^2-|\nabla v|^2}{x_2}dX &\leq \int_{B_r}\left[x_2\left(\lambda_+^2I_{\{v>0\}}-\lambda_+^2I_{\{u>0\}}+\lambda_-^2I_{\{v<0\}}-\lambda_-^2I_{\{u<0\}}\right)\right]dX \\
	&\leq \int_{B_{\kappa r}\cap\{u>0\}}-x_2\lambda_+^2dX,
	\end{aligned}
	\end{equation*}
	and hence
	\begin{equation*}
	\begin{aligned}
	\int_{B_r}\left(\frac{|\nabla u|^2}{x_2}+x_2\lambda_+^2\right)dX &\leq \int_{B_r}\frac{|\nabla v|^2}{x_2} dX+\int_{B_{\kappa r}\cap\{u>0\}}x_2\lambda_+^2dX \\
	&\leq \int_{B_{r}\cap\{u\leq0\}}\frac{|\nabla u|^2}{x_2}dX+\int_{D^+}\frac{|\nabla v|^2}{x_2}dX-\int_{B_{\kappa r}\cap\{u>0\}}x_2\lambda_+^2dX
	\end{aligned}
	\end{equation*}
	for $D^+:=(B_r\backslash B_{\kappa r})\cap \{u>0\}$.
	
	Now we can proceed as
	\begin{equation}
	\begin{aligned}
	\int_{B_{\kappa r}\cap\{u>0\}}\left(\frac{|\nabla u|^2}{x_2}+x_2\lambda_+^2\right)dX &\leq \int_{D^+}\frac{|\nabla v|^2-|\nabla u|^2}{x_2}dX \\
	&= \int_{D^+}\frac{\nabla(v-u)\cdot\nabla(u-v)}{x_2}dX+2\int_{D^+}\frac{\nabla v \cdot \nabla(v-u)}{x_2}dX \\
	&=-\int_{D^+}(v-u)div\frac{\nabla(u-v)}{x_2}dX+2\int_{D^+}\frac{\nabla v \cdot \nabla(v-u)}{x_2}dX \\
	&\leq\liminf_{\epsilon\rightarrow0}2\int_{\partial B_{\kappa r}\cap\{u>\epsilon\}}(u-\epsilon)\frac{|\nabla v_\epsilon|}{x_2}dS \\
	&=:M.
	\end{aligned}
	\end{equation}
	
	Next we estimate $M$. Consider the function
	\begin{equation*}
	\begin{cases}
	\mathcal{L}w=0 \quad \text{in} \quad B_r\backslash B_{\kappa r}, \\
	w=u \quad\ \ \text{on} \quad \partial B_r\cap\{u>\epsilon\}, \\
	w=\epsilon \quad\ \; \text{elsewhere}  \quad \text{on} \quad \partial(B_r\backslash B_{\kappa r}).
	\end{cases}
	\end{equation*}

	It is clear that from the elliptic estimate in \cite{GT01}, Chapter 8,
	\begin{equation*}
	|\nabla w|\leq C\gamma \quad \text{on} \quad \partial B_{\kappa r},
	\end{equation*}
	where $C$ is independent of $r$.
	
	Due to the fact that
	\begin{equation*}
	\begin{cases}
	\mathcal{L}(w-v_\epsilon)=0 \quad \text{in} \quad B_r\backslash B_{\kappa_r}, \\
	w-v_\epsilon \geq 0 \quad\quad\, \text{on} \quad \partial(B_r\backslash B_{\kappa r}), \\
	w-v_\epsilon=0 \quad\quad\ \text{on} \quad \partial B_r\cap\{u>\epsilon\},
	\end{cases}
	\end{equation*}
	we get that $w- v_\epsilon\geq0$ in the ring $B_r\backslash B_{\kappa r}$ and thus $|\nabla w|\geq|\nabla v_\epsilon|$ on $\partial B_{\kappa r}\cap\{u>\epsilon\}$. Hence
	\begin{equation*}
	|\nabla v_\epsilon|\leq C\gamma \quad \text{on} \quad \partial B_{\kappa r}\cap\{u>\epsilon\}.
	\end{equation*}
	By virtue of the trace-inequality,
	\begin{equation*}
	\begin{aligned}
	M &\leq \frac{2C\gamma}{b}\int_{\partial B_{\kappa r}}u^+dX \\
	&\leq C\gamma\left(\int_{B_{\kappa r}}|\nabla u^+| dX +\frac{1}{r}\int_{B_{\kappa r}}u^+dX\right) \\
	&\leq C\gamma\left[
	\frac{1}{\lambda_+}\int_{B_{\kappa r}}\left(\frac{|\nabla u|^2}{x_2}+x_2\lambda_+^2I_{\{u^+>0\}}\right)dX
	+\frac{1}{r}\frac{2}{b\lambda_+^2}\sup_{B_{\kappa r}}u^+\int_{B_{\kappa r}}x_2\lambda_+^2I_{\{u^+>0\}}
	dX\right] \\
	&\leq \frac{C\gamma}{\lambda_+}\left(1+\frac{2C\gamma}{b\lambda_+}\right)\int_{B_{\kappa r}\cap\{u>0\}}\left(\frac{|\nabla u|^2}{x_2}+x_2\lambda_+^2I_{\{u>0\}}\right)dX.
	\end{aligned}
	\end{equation*}
	The last inequality comes from
	\begin{equation*}
	\begin{aligned}
	\sup_{B_{\kappa r}}u^+\leq C\left( \fint_{B_r} u^2 dX \right)^{1/2}\leq C\gamma r
	\end{aligned}
	\end{equation*}
	for $C$ independent of $r$ and $\epsilon$.
	Combining with (A.1) we have
	\begin{equation*}
	\int_{B_{\kappa r}\cap\{u>0\}}\left(\frac{|\nabla u|^2}{x_2}+x_2\lambda_+^2\right)dX
	\leq \frac{C\gamma}{\lambda_+}\left(1+\frac{2C\gamma}{b\lambda_+}\right)\int_{B_{\kappa r}\cap\{u>0\}}\left(\frac{|\nabla u|^2}{x_2}+x_2\lambda_+^2\right)dX.
	\end{equation*}
	We get that $u\equiv0$ in $B_{\kappa r}$ if we choose $\frac{\gamma}{\lambda_+}$ small, which gives the proposition.
\end{proof}

\section{The Lipschitz-regularity of the minimizer}

We first give the monotonicity formula for the functional $J_{\text{a,tp}}$ as in \cite{C88} by Caffarelli.

\begin{lem} \label{mono}
	Let $u_1$, $u_2$ be two non-negative continuous functions such that $div(a_{ij}D_iu)\geq0$ in $B_1$, with $a_{ij}(0)=\delta_{ij}$, $u_1(0)=u_2(0)=0$ and $u_1u_2=0$ in $B_1$. Assume that $a_{ij}\in C^{0,\gamma}(B_1)$, then the function
	\begin{equation}
	\phi(r)=\frac{\int_{B_r}|\nabla u^+|^2dX \int_{B_r}|\nabla u^-|^2dX}{g(r)}
	\end{equation}
	with $g(r)=r^4e^{-C_0r^\gamma}$ is increasing for $0<r\leq\frac12$.
\end{lem}

With the ACF-type monotonicity formula at hand, we can prove the Lipschitz regularity of the minimizer.

\begin{proof}[Proof of (2) in Proposition \ref{prop2}]
	Suppose that $v$ is a minimizer of $J_{\text{a,tp}}$ in $D$. In view of Proposition \ref{prop1} we will consider the points near the axis first. Then we will consider the points away from the axis using ACF's monotonicity formula. It was first proposed in \cite{ACF84} to get the Lipschitz regularity for the minimizer of the original functional $\tilde{J}_{\text{acf}}$ mentioned in Section 1. Now for the elliptic operator $\mathcal{L}$, we should establish another monotonicity formula as in \cite{C88} and \cite{SFS19}. There is no research on its details, so we sketch the proof here and divide it into two steps.
	
	\emph{Step 1: Estimate the gradient at the points near the $x_1$-axis}. This has been done in Proposition \ref{prop1} that
	$$|\nabla u|\leq Cb$$
	for some $C>0$, and $b$ is the uniform distance from the free boundaries to the $x_1$-axis.
	
	\emph{Step 2: Estimate the gradient at the points away from the $x_1$-axis}.
	
	We first show that $v^\pm$ are subsolutions of $div(\frac1{x_2}\nabla v)=0$. Consider a smooth approximation $H_\delta$ of the Heaviside function in $\mathbb{R}$. That is, $H_\delta\in C^\infty(\mathbb{R})$ such that $H_\delta'\geq0$ and
	\begin{equation*}
	\begin{cases}
	H_\delta(t)=0 \quad \text{for} \quad t<\delta/2, \\
	H_\delta(t)>0 \quad \text{for} \quad t\geq\delta/2, \\
	H_\delta(t)=1 \quad \text{for} \quad t>\delta.
	\end{cases}
	\end{equation*}
	Let $\eta\in C_0^\infty(D)$ be nonnegative and $\phi=\eta H_\delta(v)$. Let $\epsilon$ be a small positive number such that $2\epsilon\Vert \eta \Vert_{L^\infty(D)}\leq\delta$. Notice that $\{v<\epsilon\phi\}\subset\{v\leq0\}$, which gives $(v-\epsilon\phi)^+=v-\epsilon\phi$ and $(v-\epsilon\phi)^-=0$ in $\{v>\delta/2\}$. Furthermore, $I_{\{v-\epsilon\phi>0\}}-I_{\{v>0\}}\leq0$ and $I_{\{v-\epsilon\phi<0\}}-I_{\{v<0\}}=0$ in $\{v>\delta/2\}$. Then it follows from the minimality condition that
	\begin{equation*}
	\begin{aligned}
	0&\leq\int_{\{v>\delta/2\}\cap D}\frac{\frac{1}{x_2}\left(|\nabla(v-\epsilon\phi)|^2-|\nabla v|^2\right)+x_2\lambda_+^2\left(I_{\{v-\epsilon\phi>0\}}-I_{\{v>0\}}\right)+x_2\lambda_-^2\left(I_{\{v-\epsilon\phi<0\}}-I_{\{v<0\}}\right)}{\epsilon}dX \\
	&\leq-\int_{\{v>\delta/2\}\cap D}\frac{1}{x_2}\nabla v\cdot\nabla\phi dX+o(1) \\
	&\leq-\int_{\{v>\delta/2\}\cap D}\frac{1}{x_2}H_\delta(v)\nabla v\cdot\nabla\eta dX+o(1).
	\end{aligned}
	\end{equation*}
	Letting $\epsilon\rightarrow0$ and then $\delta\rightarrow0$, the convergence $H_{\delta}(v)\rightarrow I_{\{v>0\}}$ a.e. gives
	$$\int_D \frac1{x_2}\nabla v\cdot\nabla\eta dX\leq0.$$
	By the arbitrariness of $\eta\in C_0^\infty(D)$, we conclude that $v^+$ is a subsolution. Similarly, $v^-$ is also a subsolution.
	
	For the point $X_0(a_0,b_0)$ $(b_0>b)$ and the ball $B_R(X_0)\subset D$ centered at $X_0$, we first make a transformation on $D$ to $D'$ to move $X_0$ to the origin:
	\begin{equation*}
	\begin{cases}
	y_1=\frac{x_1}{b_0}-\frac{a_0}{b_0}, \\
	y_2=\frac{x_2}{b_0}-1.
	\end{cases}
	\end{equation*}
	Then the axis $x_2=0$ is moved to $y_2=-1$, and the ball $B_R(X_0)$ is transformed to $B_{R/b_0}(0)\subset D'$. Without loss of generality suppose $R/b_0>1$. Let $v(x)=u(y)$. After such a transformation, $u^\pm$ are subsolutions in $B_1$ the following elliptic equation
	\begin{equation*}
	\mathcal{L}'u=\Delta u-\frac{1}{y_2+1}\partial_{y_2} u=0
	\end{equation*}
	for the elliptic coefficient
	$(a_{ij})_{2\times 2}=\bigl( \begin{smallmatrix}
	\frac{1}{y_2+1} & 0 \\
	0 & \frac{1}{y_2+1}
	\end{smallmatrix} \bigr)$, $a_{ij}(0)=\delta_{ij}$, $a_{ij}\in C^\gamma(B_{r_0})$ for $r_0<1$.

	Now set
	$$
	\phi(r)=\frac{\int_{B_r}|\nabla u^+|^2dY \int_{B_r}|\nabla u^-|^2dY}{g(r)}
	$$
	for $g(r)=r^4e^{-C_0r^\gamma}$ and $dY=dy_1dy_2$. Then $\phi'(r)\geq0$ for $0< r\leq\frac12$.
	
	With the monotonicity formula at hand, the subsequent proof is standard, referred to \cite{ACF84} or \cite{DEGT21}.
	
\end{proof}

\section{A preparing lemma for partial boundary Harnack inequality}

In this section, we show an important lemma, which is useful in Section 3 to imply the partial boundary Harnack inequality.

\begin{lem}
	Let $P=\frac12\bm{e}$ for $\bm{e}=(e_1,e_2)$ and suppose that $u_k$: $B_1\rightarrow\mathbb{R}$ solves $\mathcal{L}_ku_k=\Delta u_k-\frac{r_k}{y_2^0+r_kx_2}\partial_2u_k=0$ in $\{u_k>0\}$ for sufficiently large $k$, and satisfies
	$$\lambda(x\cdot\bm{e}+b)^+\leq u_k \leq\lambda(x\cdot\bm{e}+a)^+$$
	for some $a, b\in(-\frac{1}{10},\frac{1}{10})$. Then for all $0<\epsilon<\frac12$, there is a dimensional constant $\tau$ such that if
	\begin{equation*}
	u_k(P)\leq\lambda(1-\epsilon)(\frac12+a)^+ \ \big( \text{or} \ u_k(P)\geq\lambda(1+\epsilon)(\frac12+b)^+ \big),
	\end{equation*}
	then
	\begin{equation*}
	u_k\leq\lambda(1-\tau\epsilon)(x\cdot\bm{e}+a)^+ \ \big( \text{or} \ u_k\geq\lambda(1+\tau\epsilon)(x\cdot\bm{e}+b)^+ \big) \quad \text{in} \quad B_{1/4}.
	\end{equation*}
\end{lem}

\begin{proof}
	We only prove the first implication, and the latter follows in an analogous way.
	
	Noticing that $b\leq\frac{1}{10}$, the functions $u_k$ and $\lambda(x\cdot\bm{e}+a)^+$ are both positive in $B_{1/4}(P)$, and thus satisfying
	\begin{equation*}
	\mathcal{L}_ku_k=0 \quad \text{in} \quad B_{1/4}(P),
	\end{equation*}
	\begin{equation*}
	\mathcal{L}_k(\lambda(x\cdot\bm{e}+a)^+)=\frac{r_k\lambda e_2}{y_2^0+r_kx_2}>0 \quad \text{in} \quad B_{1/4}(P).
	\end{equation*}
	It follows from $u_k(P)\leq\lambda(1-\epsilon)(\frac12+a)^+$ that
	$$\lambda(\frac12+a)^+-u_k(P)\geq\lambda\epsilon(\frac12+a)^+\geq\frac25\lambda\epsilon.$$
	By Harnack's inequality in \cite{GT01}, there are constants $C_1, C_2$ such that 
	$$\lambda(x\cdot\bm{e}+a)^+-u_k\geq C_1\frac25\lambda\epsilon-C_2\frac{r_k\lambda e_2}{y_2^0}$$
	for $x\in B_{1/8}(P)$ and hence,
	\begin{equation*}
	\begin{aligned}
	u_k(x) &\leq \lambda(x\cdot\bm{e}+a)^+-C_1\frac25\lambda\epsilon+C_2\frac{r_k\lambda e_2}{y_2^0} \\
	&\leq \lambda(1-C\epsilon)(x\cdot\bm{e}+a)^+
	\end{aligned}
	\end{equation*}
	for $x\in B_{1/8}(P)$ and $C=C(C_1,C_2)$.
	
	Now introduce the function $w_k$, solving the following problem:
	\begin{equation*}
	\begin{cases}
	\mathcal{L}_kw_k=0 \qquad\qquad\qquad\quad\ \;\, \text{in} \quad B_1\backslash B_{1/8}(P)\cap\{x\cdot\bm{e}>-a\}, \\
	w_k=0 \qquad\qquad\qquad\qquad\ \ \ \text{on} \quad B_1\cap\{x\cdot\bm{e}=-a\}, \\
	w_k=\lambda(x\cdot\bm{e}+a)^+ \qquad\qquad\ \text{on} \quad \partial B_1\cap\{x\cdot\bm{e}>-a\}, \\
	w_k=\lambda(1-c\epsilon)(x\cdot\bm{e}+a)^+ \quad \text{on} \quad \partial B_{1/8}(P)\cap\{x\cdot\bm{e}>-a\}.
	\end{cases}
	\end{equation*}
	The existence of $w_k$ comes from the solvability of uniformly elliptic equation with Dirichlet boundary condition. Notice that the smooth approximation of boundary helps to deal with the intersection of the arc and the segment.
	By the Hopf boundary lemma for a strictly elliptic operator in \cite{HL11}, there exists a suitable constant $\tau$ such that for every $x\in B_{1/4}\cap\{x\cdot\bm{e}>-a\}$,
	$$w_k\leq\lambda(1-\tau\epsilon)(x\cdot\bm{e}+a)^+.$$
	Recall the property of $u_k$,
	\begin{equation*}
	\begin{cases}
	\mathcal{L}_k(u_k-w_k)=0 \quad \text{in} \quad \{u_k>0\}\cap\{w_k>0\}\cap B_{1/4}, \\
	u_k-w_k\leq0 \qquad\ \ \ \text{on} \quad \partial B_1\cap\{x\cdot\bm{e}>-a\}, \\
	u_k-w_k\leq0 \qquad\ \ \ \text{on} \quad \partial B_{1/8}(P)\cap\{x\cdot\bm{e}>-a\}, \\
	u_k-w_k\leq0 \qquad\ \ \ \text{on} \quad B_1\cap\{x\cdot\bm{e}=-a\}.
	\end{cases}
	\end{equation*}
	This together with $\{u_k>0\}\subset B_1\cap\{x\cdot\bm{e}>-a\}$ implies
	$$u_k-w\leq0 \quad \text{in} \quad B_{1/4}\cap\{x\cdot\bm{e}>-a\}.$$
	It completes the proof.
\end{proof}

\section{A touching lemma}

In this section, we prove a touching lemma, which is widely used in checking the boundary condition of the limiting "linearized" problem, see Proposition \ref{lim1} and Proposition \ref{lim2}.

\begin{lem}  \label{touch}
	Suppose that $u_k$ is a blow-up sequence at $y_0=(y_1^0,y_2^0)\in\Gamma_{\rm{tp}}$, and $\alpha_k, \epsilon_k, v_k$ are defined as before.
	
	(1) Let $P_+:B_{1/2}\cap\{x\cdot\bm{e}>0\}\rightarrow\mathbb{R}$ be a strictly subharmonic(superharmonic) function touching $v_+$ strictly from below(above) at $x_0\in B_{1/2}\cap\{x\cdot\bm{e}=0\}$. Then there is a sequence of points $x_k\in\partial\Omega_{u_k}^+$ converging to $x_0$ and a sequence of comparison functions $Q_k$ touching $u_k^+$ from below(above) at $x_k$ such that
	\begin{equation}
	\nabla Q_k^+(x_k)=y_2^0\alpha_k\bm{e}+y_2^0\alpha_k\epsilon_k\nabla P_+(x_0)+O(\epsilon_k^2).
	\end{equation}
	
	(2) Let $P_-:B_{1/2}\cap\{x\cdot\bm{e}<0\}\rightarrow\mathbb{R}$ be a strictly subharmonic(superharmonic) function touching $v_-$ strictly from below(above) at $x_0\in B_{1/2}\cap\{x\cdot\bm{e}=0\}$. Then there is a sequence of points $x_k\in\partial\Omega_{u_k}^-$ converging to $x_0$ and a sequence of comparison functions $Q_k$ touching $-u_k^-$ from below(above) at $x_k$ such that
	\begin{equation}
	\nabla Q_k^-(x_k)=-y_2^0\beta_k\bm{e}+y_2^0\beta_k\epsilon_k\nabla P_-(x_0)+O(\epsilon_k^2).	
	\end{equation}
	
	(3) Let $P=p(x\cdot\bm{e})^+-q(x\cdot\bm{e})^-+\bar{P}$ in $B_{1/2}$ for $p,q\in\mathbb{R}$, where $\bar{P}$ is strictly subharmonic(superharmonic) and $\partial_e\bar{P}=0$. Suppose that $P$ touches $v$ strictly from below(above) at $x\in\mathcal{C}$. Then there is a sequence of points $x_k\in\partial\Omega_{u_k}$ converging to $x_0$ and a sequence of comparison functions $Q_k$ touching $u_k$ from below(above) at $x_k$ such that
	\begin{equation}
	\begin{aligned}
	\nabla Q_k^+(x_k)&=y_2^0\alpha_k\bm{e}+y_2^0\alpha_k\epsilon_kp\bm{e}+O(\epsilon_k^2),\\
	\nabla Q_k^-(x_k)&=-y_2^0\beta_k\bm{e}-y_2^0\beta_k\epsilon_kq\bm{e}+O(\epsilon_k^2).
	\end{aligned}
	\end{equation}
	In particular, if $p>0$ and $Q_k$ touches $u_k$ from below, then $x_k\notin\partial\Omega_{u_k}^-\backslash\partial\Omega_{u_k}^+$; if $q<0$ and $Q_k$ touches $u_k$ from above, then $x_k\notin\partial\Omega_{u_k}^+\backslash\partial\Omega_{u_k}^-$.
\end{lem}

\begin{proof}
	We divide the proof into 3 steps.
	
	\emph{Step 1. Construction of a function $Q$ with the desired gradient.}
	
	Define $\bm{T}_\epsilon: B_{1/2}\cap\{x\cdot\bm{e}>0\}\rightarrow\mathbb{R}^2$ to be a function
	$$\bm{T}_\epsilon(x)=\bm{T}_\epsilon(x_1,x_2)=x-\epsilon P\bm{e}=(x_1-e_1\epsilon P,x_2-e_2\epsilon P)$$
	for $x=(x_1,x_2)\in B_{1/2}\cap\{x\cdot\bm{e}>0\}$ and $\bm{e}=(e_1,e_2)$. Here we only prove for $e_2>0$. For notational simplicity take $B_{1/2}^+:=B_{1/2}\cap\{x\cdot\bm{e}>0\}$ in this proof.
	
	Note $y_\epsilon=\bm{T}_\epsilon(x)$ and we have
	\begin{equation*}
	\frac{\partial y_\epsilon}{\partial x}=
	\begin{pmatrix}
	1-e_1\epsilon\partial_1P & -e_2\epsilon\partial_1P \\
	-e_1\epsilon\partial_2P & 1-e_2\epsilon\partial_2P \\
	\end{pmatrix},
	\quad	
	\frac{\partial x}{\partial y_\epsilon}=\frac{1}{1-\epsilon\nabla P\cdot\bm{e}}
	\begin{pmatrix}
	1-e_2\epsilon\partial_2P & e_2\epsilon\partial_1P \\
	e_1\epsilon\partial_2P & 1-e_1\epsilon\partial_1P \\
	\end{pmatrix}.
	\end{equation*}
	Thus for $\epsilon\ll\Vert P \Vert_{C^1}^{-1}$, $\bm{T}_\epsilon$ induces a bijection between $B_{1/2}^+$ and $U_\epsilon:=\bm{T}_\epsilon(B_{1/2}^+)\subset B_1$.
	
	Take $\bm{Q}_\epsilon=\bm{T}_\epsilon^{-1}$ and $Q_\epsilon=\alpha(\bm{Q}_\epsilon\cdot\bm{e})$ for $\alpha\in\mathbb{R}$,
	$$Q_\epsilon(x-\epsilon P\bm{e})=\alpha(x\cdot\bm{e}):U_\epsilon\rightarrow(0,1/2).$$
	Extend $Q_\epsilon$ to zero in $B_{1/2}\backslash \{Q_\epsilon>0\} $. After elementary calculations,
	\begin{equation*}
	\begin{aligned}
	\nabla_y Q_\epsilon(y_\epsilon) &=\nabla_x Q_\epsilon(\bm{T}_\epsilon(x))\frac{\partial x}{\partial y_\epsilon} \\
	&=\frac{\alpha}{1-\epsilon\nabla P\cdot\bm{e}}
	\begin{pmatrix}
	1-e_2\epsilon\partial_2P & e_2\epsilon\partial_1P \\
	e_1\epsilon\partial_2P & 1-e_1\epsilon\partial_1P \\
	\end{pmatrix}
	\begin{pmatrix}
	e_1 \\
	e_2 \\
	\end{pmatrix} \\
	&=\alpha
	\begin{pmatrix}
	e_1+\frac{\epsilon\partial_1P}{1-\epsilon\nabla P\cdot\bm{e}} \\
	\\
	e_2+\frac{\epsilon\partial_2P}{1-\epsilon\nabla P\cdot\bm{e}} \\
	\end{pmatrix} \\
	&=\alpha\cdot\bm{e}+\alpha\frac{\epsilon\nabla P}{1-\epsilon\nabla P\cdot\bm{e}} \\
	&=\alpha\cdot\bm{e}+\alpha\epsilon\nabla P+O(\epsilon^2),
	\end{aligned}
	\end{equation*}
	and
	\begin{equation*}
	\Delta Q_\epsilon=\alpha\epsilon\Delta P+O(\epsilon^2).
	\end{equation*}
	
	\emph{Step 2. Construction of touching points.}
	
	We only consider case (1) in the lemma, and case (2) can be obtained by a similar argument.
	
	Since $P_+$ touches $v_+$ strictly from below, $v_+-P_++\delta$ has a strictly positive minimum at $x_0$ for any positive number $\delta\rightarrow0$. Let $Q_k^\delta$ be the function introduced in step 1 with $\epsilon=\epsilon_k, \alpha=y_2^0\alpha_k$ and let $P=P_+-\delta$.
	
	Define
	$$P_k^\delta(x)=\frac{Q_k^\delta(x)-y_2^0\alpha_k(x\cdot\bm{e})^+}{y_2^0\alpha_k\epsilon_k}$$
	and
	$$\tilde{\Gamma}_k=\bigl\{ (x,P_k^\delta(x)), \ \ x\in\overline{ \{Q_k^\delta>0\}\cap B_{1/2} } \bigr\}.$$
	Using
	\begin{equation*}
	\begin{aligned}
	Q_k^\delta(x) &= Q_k^\delta(x-\epsilon_kP\bm{e}+\epsilon_kP\bm{e}) \\
	&=Q_k^\delta(x-\epsilon_kP\bm{e})+\nabla Q_k^\delta(x-\epsilon_kP\bm{e})\cdot\epsilon_kP\bm{e}+O(\epsilon_k^2) \\
	&=y_2^0\alpha_k(x\cdot\bm{e})+y_2^0\alpha_k\epsilon_kP+O(\epsilon_k^2),
	\end{aligned}
	\end{equation*}
	we can easily check the Hausdorff convergence
	$$\tilde{\Gamma}_k \rightarrow \tilde{\Gamma}:=\bigl\{ (x,P_+(x)-\delta), \ \ x\in \overline{B_{1/2}^+} \bigr\}.$$
	
	Now we claim: $\{Q_k^\delta>0\}\cap B_{1/2}\Subset\{u_k>0\}\cap B_{1/2}$, so that we can translate $Q_k$ to touch $u_k$ at some $x_0\in\partial\Omega_{u_k}^+$.
	Indeed, otherwise we would find a sequence $\{x_k\}\rightarrow\bar{x}\in\{u_k^+=0\}$ such that $Q_k^\delta(x_k)\geq0$ while $u_k^+(x_k)=0$. This together with
	$$\Gamma_k^+=\bigl\{ (x,v_{+,k}(x)), \ \ x\in\overline{\Omega_{u_k}^+\cap B_{1/2}} \bigr\} \quad \longrightarrow \quad \Gamma_+=\bigl\{ (x,v_+(x)), \ \ x\in \overline{B_{1/2}^+} \bigr\}$$
	implies that $P_k^\delta(x_k)\geq v_{+,k}(x_k)$ and $P^+(\bar{x})-\delta\geq v_+(x)$, in contradiction with the fact that $P^+-\delta<v$.
	
	Consequently there exists $\sigma=O(\delta)$ such that $Q_k^\delta(\cdot-\sigma\bm{e})$ touches $u_k^+$ from below at some $x_k^\delta$. Recall that $P$ is strictly subharmonic, $\Delta Q_k^\delta>0$ in $\{Q_k^\delta>0\}$ and thus
	$$\mathcal{L}_kQ_k^\delta=\Delta Q_k^\delta-\frac{r_k}{y_2^0+r_kx_2}\partial_2Q_k^\delta>0$$
	for $\partial_2Q_k^\delta=y_2^0 e_2 \alpha_k+y_2^0\alpha_k\epsilon_k\partial_2P+O(\epsilon^2)$. Hence
	\begin{equation*}
	\begin{cases}
	L_k(Q_k^\delta-u_k)>0 \quad \text{in} \quad \{Q_k^\delta>0\}\cap B_{1/2}, \\
	Q_k^\delta-u_k \leq 0 \qquad\ \ \text{on} \quad \partial\{Q_k^\delta>0\}\cap B_{1/2}.
	\end{cases}
	\end{equation*}
	By the maximum principle, the touching point $x_k^\delta$ lies on $\partial\{Q_k^\delta>0\}$ and thus, on $\partial\Omega_{u_k}^+$. Note that a proper translation ensures the touching point to be on $\partial\{Q_k^\delta>0\}$, not on $\partial B_{1/2}$.
	
	It remains to check the gradient condition for $Q_k^\delta$. In fact,
	\begin{equation*}
	\begin{aligned}
	P_+(\bm{Q}_k^\delta(x_k^\delta))&=P_+(x_k^\delta+\epsilon_kP\bm{e}) \\
	&=P_+(x_0)+\nabla P_+(x_0)\cdot(x_k^\delta+\epsilon_kP\bm{e}-x_0)+R_1
	\end{aligned}
	\end{equation*}
	with $R_1$ a Lagrange remainder in Taylor's expansion and
	$$\nabla_{x_k^\delta}P_+(\bm{Q}_k^\delta(x_k^\delta))=\nabla P_+(x_0)+O(\epsilon_k)$$
	with $|x_k^\delta-x_0|<\epsilon_k$. It is straightforward to deduce that
	\begin{equation*}
	\begin{aligned}
	\nabla Q_k^\delta(x_k^\delta) &= y_2^0\alpha_k\bm{e}+y_2^0\alpha_k\epsilon_k\nabla P_+(\bm{Q}_k^\delta(x_k^\delta))+O(\epsilon^2) \\
	&=y_2^0\alpha_k\bm{e}+y_2^0\alpha_k\epsilon_k\nabla P_+(x_0)+O(\epsilon_k^2).
	\end{aligned}
	\end{equation*}
	
	Furthermore, thanks to the convergence $x_k^\delta\rightarrow x_0\in B_{1/2}\cap\{x\cdot\bm{e}=0\}$ up to a subsequence as $k\rightarrow\infty$, we clearly obtain the desired conclusion.
	
	\emph{Step 3. Proof for item (3).}
	
	Denote
	\begin{equation*}
	P=
	\begin{cases}
	P_+=p(x\cdot\bm{e})+\bar{P} \quad\ \ \text{in} \quad B_{1/2}^+, \\
	P_-=-q(x\cdot\bm{e})+\bar{P} \quad \text{in} \quad B_{1/2}^-.
	\end{cases}
	\end{equation*}
	Let $\bm{T}^{\pm}$ be the corresponding transformations as in step 1. The key point is to get that $\bm{T}^+(B_{1/2}^+)\cap\bm{T}^-(B_{1/2}^-)=\varnothing$. In fact, assume there are $x\in B_{1/2}^+$ and $y\in B_{1/2}^-$ such that $\bm{T}^+(x)=\bm{T}^-(y)$, then
	$$x-\epsilon P_+\bm{e}=y-\epsilon P_-\bm{e}.$$
	For $\bm{e}^{\perp}$ normal to $\bm{e}$,
	$$x\cdot\bm{e}^{\perp}=y\cdot\bm{e}^{\perp}.$$
	This in addition with $\partial_e\bar{P}=0$ leads to
	$$(x-\epsilon P_+\bm{e})\cdot\bm{e}=(y-\epsilon P_-\bm{e})\cdot\bm{e},$$
	which means
	$$(1-\epsilon p)(x\cdot\bm{e})=(1+\epsilon q)(y\cdot\bm{e}).$$
	For $\epsilon$ small enough, either $x\cdot\bm{e}$ has the same sign with $y\cdot\bm{e}$, or they both vanish. This gives a contradiction.
	
	Hence $Q=Q^++Q^-$ is a well-defined comparison function. Arguing as in step 2 we arrive at the desired result.
	
	In particular, if $p>0$ and $Q_k$ touches $u_k$ from below and $x_k\in\partial\Omega_{u_k}^-\backslash\partial\Omega_{u_k}^+$, then $Q_k^+\equiv0$ in a neighborhood of $x_k$. Then there exists a point $z_k$ in this neighborhood such that $z_k\cdot\bm{e}\geq\delta_0$ for a positive constant $\delta_0$, and $z_k\rightarrow z$ up to a subsequence for $z\cdot\bm{e}>0$. Hence we have
	$$\lim_{k\rightarrow\infty}P_+(z_k)\lim_{k\rightarrow\infty}\frac{Q_k^+(z_k)-y_2^0\alpha_k(z_k\cdot\bm{e})^+}{y_2^0\alpha_k\epsilon_k}\leq0,$$
	a contradiction with $P_+(z)>0$.
	
\end{proof}

\section{Regularity theorems for the limiting problem}

We give some regularity results here for the limiting problems in Section 3, which are useful in Proposition \ref{flat1} and Proposition \ref{flat2} to get the improvement of flatness. The proofs can be found respectively in \cite{SFS14} and \cite{PSV21}.

\begin{prop}
	(Regularity for the two-membrane problem in 2-dimension)
	Suppose that $v\in C^0(B_{1/2})$ is a viscosity solution of (\ref{tm}) with $\Vert v\Vert_{L^\infty(B_{1/2})}\leq1$. Then there exist $C=C(\lambda_{\pm}, l)>0$ and $t, p, q\in\mathbb{R}$ satisfying $\lambda_+^2p=\lambda_-^2q\geq-l$ such that
	$$\sup_{B_r}\frac{|v(x)-v(0)-(t(x\cdot\bm{e}^\perp)+p(x\cdot\bm{e})^+-q(x\cdot\bm{e})^-)|}{r^{3/2}}\leq C.$$
\end{prop}

\begin{prop}
	(Regularity for the transmission problem in 2-dimension)
	Suppose that $v\in C^0(B_{1/2})$ is a viscosity solution of (\ref{t}) with $\Vert v\Vert_{L^\infty(B_{1/2})}\leq1$. Then there exist $C=C(\alpha_{\infty}, \beta_{\infty})>0$ and $t, p, q\in\mathbb{R}$ satisfying $\alpha_{\infty}^2p=\beta_{\infty}^2q\geq-l$ such that
	$$\sup_{B_r}\frac{|v(x)-v(0)-(t(x\cdot\bm{e}^\perp)+p(x\cdot\bm{e})^+-q(x\cdot\bm{e})^-)|}{r^{2}}\leq C.$$
\end{prop}


\vskip .4in
\begin{thebibliography}{60}
	\bibitem{ACF83}H. W. Alt, L. A. Caffarelli, A. Friedman, Axially symmetric jet flows, \emph{Arch. Ration. Mech. Anal.}, 81, 97-149, (1983).
	\bibitem{ACF84}H. W. Alt, L. A. Caffarelli, A. Friedman, Variational problems with two phases and their free boundaries, \emph{Trans. Amer. Math. Soc.}, 282, no. 2, 431-461, (1984).
	\bibitem{ACF84-1}H. W. Alt, L. A. Caffarelli, A. Friedman, Jets with two fluids. I. One free boundary, \emph{Indiana Univ. Math. J.}, 33, 213-247, (1984).
	\bibitem{ACF84-2}H. W. Alt, L. A. Caffarelli, A. Friedman, Jets with two fluids. II. Two free boundaries, \emph{Indiana Univ. Math. J.}, 33, 367-391, (1984).
	\bibitem{B56}G. K. Bachelor, On steady laminar flow with closed streamlines at large Reynolds number, \emph{J. Fluid Mech.}, 1, 177-190, (1956).
	\bibitem{C87}L. A. Caffarelli, A Harnack inequality approach to the regularity of free boundaries. Part I: Lipschitz free boundaries are $C^{1,\alpha}$, \emph{Rev. Mat. Iberoamericana.}, 3, 139-162, (1987).
	\bibitem{C89}L. A. Caffarelli, A Harnack inequality approach to the regularity of free boundaries. Part II: flat free boundaries are Lipschitz, \emph{Comm. Pure Appl. Math.}, 42, 55-78, (1989).
	\bibitem{C88}L. A. Caffarelli, A Harnack inequality approach to the regularity of free boundaries. Part III: existence theory, compactness and dependence on X, \emph{Ann. Scuola Norm. Sup. Pisa.}, 15(4), 583-602, (1988).
	\bibitem{C09}S. Childress, An introduction to theoretical fluid mechanics, Courant Lecture Notes in Mathematics, 19. \emph{Courant Institute of Mathematical Science, New York; American Mathematical society, Providence, RI,} (2009).
	\bibitem{CDZ21}J. F. Cheng, L. L. Du, Q. Zhang, Existence and uniqueness of axially symmetric compressible subsonic jet impinging on an infinite wall, \emph{Interface Free Bound.},23, 1-58, (2021).
	\bibitem{CJK04}L. A. Caffarelli, D. S. Jerison, C. E. Kenig, Global energy minimizers for free boundary problems and full regularity in three dimensions, Contemp. Math. 350. \emph{Amer. Math. Soc. Providence RI.}, 83-97, (2004).
	\bibitem{DEGT21}G. David, M. Engelstein, M. S. V. Garcia, T. Toro, Regularity for almost-minimizers of variable coefficient Bernoulli-type functionals, \emph{Math. Z.}, 299, 2131-2169, (2021).
	\bibitem{DEGT23}G. David, M. Engelstein, M. S. V. Garcia, T. Toro, Branching points for (almost-)minimizers of two-phase free boundary problems, \emph{Forum Math. Sigma.}, 11, Paper No. e1, 28, (2023).
	\bibitem{EM95}A. R. Elcrat, K. G. Miller, Variational formulas on Lipchitsz domains, \emph{Trans. Amer. Math. Soc.} 347, 2669-2678, (1995).
	\bibitem{FL94}A. Friedman, Y. Liu, A free boundary problem arising in magnetohydrodynamic system, \emph{Ann. Sc. Norm. Super. Pisa Cl. Sci.}, 22(4), 375-448, (1994).
	\bibitem{G56}P. R. Garabedian, Calculation of axially symmetric cavities and jets, \emph{Pacific J. Math.},6, 611-684, (1956).
	\bibitem{GT01}D. Gilbarg, N. S. Trudinger, Elliptic Partial Differential Equations of Second Order, Classic in Mathematics. \emph{Springer-Verlag, Berlin,} (2001).
	\bibitem{GVW16}M .S. V. Garcia, E. V$\check{a}$rv$\check{a}$ruc$\check{a}$, G. S. Weiss, Singularities in axisymmetric free boundaries for electrohydrodynamic equations, \emph{Arch. Ration. Mech. Anal.}, 222(2), 573-601, (2016).
	\bibitem{HL11}Q. Han, F. H. Lin, Ellpitic Partial Differential Equations, \emph{America Mathematical Society, New York}, (2011).
	\bibitem{LW06}C. Lederman, N. Wolanski, A two phase elliptic singular pertubation problem with a forcing term, \emph{J. Math. Pures Appl.}, 86(9), no.6, 552-589, (2006).
	\bibitem{MM13}G. V. Messa, S. Malavasi, Numerical investigation of solid-liquid slurry flow through an upward-facing step, \emph{J. Hydrol. Hydromech.}, 61(2), 126-133, (2013).
	\bibitem{PSV21}G. D. Philippis, L. Spolaor, B. Velichkov,  Regularity of the free boundary for the two-phase Bernoulli problem, \emph{Invent. Math.}, 225, 347-394, (2021).
	\bibitem{S52}J. Serrin, On plane and axially symmetric free boundary problems, \emph{J. Rational Mech. Anal.}, 1, 563-572, (1952).
	\bibitem{S11}D. D. Silva, Free boundary regularity for a problem with right hand side, \emph{Interface. Free. Bound.}, 13(2), 223-238, (2011).
	\bibitem{S18}H. Shahgholian, Rencent trends in free boundary regularity, Geometry of PDEs and related problems, Lecture Notes in Math., \emph{Springer, Cham,} 147-196, (2018).
	\bibitem{SCK21}T. Schouten, V. R. Cees, G. Keetels, Two-phase modelling for sediment water mixtures above the limit deposit velocity in horizontal pipelines, \emph{J. Hydrol. Hydromech.}, 69(3), 263-274, (2021).
	\bibitem{SFS14}D. D. Silva, F. Ferrari, S. Salsa, Two-phase problems with distributed sources: regularity of the free boundary, \emph{Anal. PDE.}, 7, 267-310, (2014).
	\bibitem{SFS15}D. D. Silva, F. Ferrari, S. Salsa, Free boundary regularity for fully nonlinear non-homogeneous two-phase problems, \emph{J. Math. Pure. Appl.}, 103(9), 658-694, (2015).
	\bibitem{SFS17}D. D. Silva, F. Ferrari, S. Salsa, Two-phase free boundary problems: from existence to smoothness, \emph{Adv. Nonlinear Stud.}, 17(2), 369-385, (2017).
	\bibitem{SFS19}D. D. Silva, F. Ferrari, S. Salsa, Recent progresses on elliptic two-phase free boundary problems, \emph{Discrete. Contin. Dyn. Syst.}, 39(12), 1-18, (2019).
	\bibitem{V23}B. Velichkov, Regularity of the one-phase free boundaries, Lecture Notes of the Unione Matematica Italiana, \emph{Springer Cham}, (2023).
	\bibitem{VW14}E. V$\check{a}$rv$\check{a}$ruc$\check{a}$, G. S. Weiss, Singularities of steady axisymmetric free surface flows with gravity, \emph{Comm. Pure Appl. Math.}, 67(8), 1263-1306, (2014).
	\bibitem{W99}G. S. Weiss, Partial regularity for a minimum problem with free boundary, \emph{J. Geom. Anal.}, 9(2), 317-326, (1999).
	
\end{thebibliography}
\end{document}